\documentclass[11pt,a4paper]{article}

\usepackage[utf8]{inputenc}
\usepackage[paperheight=297mm,paperwidth=210mm,textheight=26cm,textwidth=16cm,includehead]{geometry}
\usepackage[english]{babel}
\usepackage{amsmath,amssymb,mathrsfs,stmaryrd,color}
\usepackage{amsthm}
\usepackage{mathtools,bm}
\usepackage{enumerate}
\usepackage{comment}
\usepackage{hyperref}
\usepackage{todonotes}
\usepackage{fullpage}
\usepackage{stmaryrd}
\usepackage{graphicx}
\usepackage{caption}
\usepackage{subcaption}

\usepackage[style = alphabetic]{biblatex}
\numberwithin{equation}{section}

	\newcommand{\bbN}{\mathbb N}
	
	\newcommand{\bbR}{\mathbb R}

	\newcommand{\bbX}{\mathbb X}

	\newcommand{\bfX}{\mathbf X}
\newcommand{\bfY}{\mathbf Y}	\newcommand{\bfZ}{\mathbf Z}

\newcommand{\calI}{\mathcal I}
\newcommand{\calT}{\mathcal T}
\newcommand{\calY}{\mathcal Y}

\theoremstyle{plain}
\newtheorem{theorem}{Theorem}[section]

\newtheorem{corollary}[theorem]{Corollary}
\newtheorem{lemma}[theorem]{Lemma}
\newtheorem{proposition}[theorem]{Proposition}

\theoremstyle{definition}
\newtheorem{definition}[theorem]{Definition}

\newtheorem{example}[theorem]{Example}

\newtheorem{remark}[theorem]{Remark}

\newcommand{\norm}[1]{\left\Vert #1 \right\Vert}
\newcommand{\nnorm}[1]{{\left\vert\kern-0.25ex\left\vert\kern-0.25ex\left\vert #1 \right\vert\kern-0.25ex\right\vert\kern-0.25ex\right\vert}}
\newcommand{\Set}[1]{\left\{ #1 \right\}}
\newcommand{\abs}[1]{\left\vert #1 \right\vert}
\newcommand{\dif}{\text{d}}
\newcommand{\given}{~\vert~}

\newcommand{\tX}{\tilde{X}}
\newcommand{\tY}{\tilde{Y}}
\newcommand{\tZ}{\tilde{Z}}

\newcommand{\tphi}{\tilde{\phi}}

\newcommand{\tPhi}{\Phi}

\newcommand{\cI}{\mathcal I}

\newcommand{\Csew}{C_{\mathrm{sew}}}
\newcommand{\Lip}{\mathrm{Lip}}
\AtEveryBibitem{\clearfield{eprint}}
\AtEveryCitekey{\clearfield{eprint}}
\AtEveryBibitem{\clearfield{url}}
\AtEveryCitekey{\clearfield{url}}

\begin{document}

\title{Lipschitz estimates in the Besov settings for Young and rough differential equations}
\author{Peter Friz, Hannes Kern, Pavel Zorin-Kranich}

\maketitle

\begin{abstract}
We develop a set of techniques that enable us to effectively recover Besov rough analysis from $p$-variation rough analysis.
Central to our approach are new metric groups, in which some objects in rough path theory that have been previously viewed as two-parameter can be considered as path increments.
Furthermore, we develop highly precise Lipschitz estimates for Young and rough differential equations, both in the variation and Besov scale.
\end{abstract}
  
\tableofcontents

\section{Introduction}

Rough path theory gives meaning to differential equations of the form
\begin{equation} \label{equ:intro1}
dY_t = \phi(Y_t) dX_t
\end{equation}
with a (multidimensional) driver $X$ that has (locally) finite $r$-variation.
When $r < 2$, this equation can be understood as a Young integral equation.
For $r \ge 2$, it was understood by T. Lyons \cite{Lyo98} that $X$ needs to be enhanced with additional information to restore well-posedness of the problem. In particular, for $r \in [2,3)$, an appropriate driver is a level-$2$ rough path of the form $\bfX = (X, \bbX)$, where $\bbX = \bbX (s,t)$ is typically interpreted as $\int_s^t (X_u - X_s) \otimes  d X_u$.

While finite $r$-variation is the minimum regularity hypothesis on $X$ needed to make sense of \eqref{equ:intro1}, higher regularity hypotheses may lead to stronger well-posedness results for the solution $Y$.
A classical hypothesis is H\"older continuity, and several authors \cite{PT16, FP18, Liu19, FS22, LPT22} have contributed to extending rough path estimates to Sobolev and Besov scale.
We note that the Besov sewing lemma of \cite{FS22} has found recent applications  \cite{bechtold2024youngregimeslocallymonotone} in non-linear SPDE theory.

In this article, we present a way of upgrading $r$-variation estimates for differential equations of the form \eqref{equ:intro1} to Besov norm estimates.
At the highest level, our approach consists of two steps.
Firstly, we show equivalence between Besov norms and metrics on the space of drivers $\bfX$ and nested norms involving local $r$-variation and an outer $2$-parameter Besov norm.
Secondly, we show local $r$-variation estimates for differential equations that are directly applicable to these nested norms.
We found the previously existing $r$-variation estimates not to be sufficiently precise for this purpose.

It would be interesting to know whether an analogous treatment of \emph{partial} differential equations subjected to rough noise, also in space-time, is possible, starting with suitably localized estimates and recovering well-posedness results in a range of function spaces. Initial progress in this direction was made in the article \cite{ZK22}, which contains a localized reconstruction theorem that recovers the Besov space case first shown in \cite{HL17}.


\paragraph{Local $r$-variation}
The basic connection between Besov and local $r$-variation norms is contained in our first result.

\begin{theorem}[cf.\ Theorem~\ref{theo:VrBesovEstimate}]
\label{theo:IntroVrBesovEstimate}
Let $f:[0,T]\to X$ for some metric space $(X,d)$ be continuous. Let $r\ge 1, \alpha>0, p,q\in(0,\infty]$ such that
\[
	\frac 1\alpha <r\le p\,.
\]
Then
\begin{equation}
\label{eq:9}
\norm{V^r f}_{B^\alpha_{p,q}} \lesssim \norm{f}_{B^\alpha_{p,q}}\,.
\end{equation}
\end{theorem}
The $B^\alpha_{p,q}$ norm of the $2$-parameter function $(s,t) \mapsto V^r f_{[s,t]}$ that appears on the left-hand side of \eqref{eq:9} is introduced in Definition~\ref{def:2Besov}.

We note that, with our conventions for Besov norms, it is immediately clear that $\norm{f}_{B^\alpha_{p,q}} \leq \norm{V^r f}_{B^\alpha_{p,q}}$.
Hence, Theorem~\ref{theo:IntroVrBesovEstimate} can be viewed as a characterization of Besov functions in terms of bounds on their local $r$-variation.

\paragraph{Rough path differences}
It is well-known (e.g. \cite{FV10,FH20}) that a rough path $\bfX = (X, \bbX)$ can be regarded as a path with values in a metric group, with (group) increments given by $(s,t) \mapsto \bfX (s,t) \equiv (\delta X, \bbX)(s,t)$. It is much less obvious, however, how to apply a metric statement like Theorem \ref{theo:IntroVrBesovEstimate} to the difference $\Delta \bfX = (\Delta X, \Delta \bfX)$, with $\Delta X = X - \tilde X$ and
$$
   \Delta \bbX: (s,t) \mapsto \bbX_{s,t} - \tilde  \bbX_{s,t} .
$$
For this reason, metric arguments were considered insufficient to deal with this more general situation, see e.g.\ the remark above \cite[Theorem 3.3]{FH20}, let alone controlled rough paths $\bfY = (Y,Y')$ with
 $R^{\bfY,\bfX} = \delta Y - Y' \delta X$ and differences like
\[
\Delta R^{\bfY,\bfX}:  (s,t) \mapsto (\delta Y - Y' \delta X) - (\delta \tilde  Y -  \tilde Y' \delta  \tilde X).
\]
Our second contribution is a new range of metric groups (Lemma~\ref{lem:homogeneous-subadditive-functional}), which allow viewing various two-parameter functions relevant to rough path theory as increments of paths in these groups.
The functions that can be represented in this way include
  \begin{align*}
  (s,t)&\mapsto (\delta X_{s,t}, \delta\tilde X_{s,t}, \Delta X_{s,t}, \Delta\bbX_{s,t})\\
  (s,t)&\mapsto (\delta Y'_{s,t}, \delta X_{s,t}, R^{\bfY,\bfX}_{s,t})\\
  (s,t)&\mapsto (\delta\tilde Y'_{s,t}, \Delta Y'_{s,t}, \delta X_{s,t},\Delta X_{s,t}, \Delta R^{\bfY,\bfX}_{s,t})\,,
  \end{align*}
  where $\Delta Y' = Y'-\tilde Y'$. Applying Theorem \ref{theo:IntroVrBesovEstimate} to these paths leads to a complete range of estimates.

\begin{theorem}[cf.\ Proposition~\ref{prop:Vr-Besov-Embedding-RYX} and Corollary~\ref{cor:Vr_Besov_Embedding_Rough_Paths}]\label{theo:IntroVrBesovEstimateRP}
	Let $\bfX,\tilde\bfX$ be $r$-rough paths on $[0,T]$, for some $r\in[2,3)$ and let $\bfY,\tilde\bfY$ be controlled rough paths. Assume that $0<\frac 1\alpha<r\le p\le \infty$ and $0<q\le\infty$. Then:
	\begin{align*}
		\norm{V^{r/2} \bbX}_{B^{2\alpha}_{p/2,q/2}} &\lesssim \norm{\bbX}_{B^{2\alpha}_{p/2,q/2}} + \norm{X}_{B^\alpha_{p,q}}^2 \\
\norm{V^{r/2} R^{\bfY,\bfX}}_{B^{2\alpha}_{p/2,q/2}} &\lesssim \norm{R^{\bfY,\bfX}}_{B^{2\alpha}_{p/2,q/2}} + \norm{Y'}_{B^\alpha_{p,q}}\norm{X}_{B^\alpha_{p,q}} \\
\norm{V^{r/2}\Delta\bbX}_{B^{2\alpha}_{p/2,q/2}} &\lesssim \norm{\Delta\bbX}_{B^{2\alpha}_{p/2,q/2}} + \norm{\Delta X}_{B^\alpha_{p,q}}(\norm{X}_{B^{\alpha}_{p,q}} + \norm{\tilde X}_{B^\alpha_{p,q}}) \\
\norm{V^{r/2} \Delta R^{\bfY,\bfX}}_{B^{2\alpha}_{p/2,q/2}} &\lesssim \norm{\Delta R^{\bfY,\bfX}}_{B^{2\alpha}_{p/2,q/2}} + \norm{\Delta Y'}_{B^\alpha_{p,q}} \norm{X}_{B^\alpha_{p,q}} + \norm{\tilde Y'}_{B^\alpha_{p,q}} \norm{\Delta X}_{B^\alpha_{p,q}}\,.
	\end{align*}
\end{theorem}

\paragraph{Young differential equations}  As an application, we use these estimates to derive Lipschitz estimates to solutions of RDEs in the Besov setting from the variation setting. We start with the Young regime.
Let $Y$ be a solution to \eqref{equ:intro1} with driver $X \in V^r$, $r<2$ and consider another solution to \begin{equation}\label{eq:IntroYoungDE}
d \tilde Y_t = \tilde \phi(\tilde Y_t) d \tilde X_t
\end{equation}
started from $\tilde y_0$ and $\tilde X\in V^r, r<2$, on $[0,T]$.Then it is well known that one has local Lipschitz estimates of the form\footnote{Relevant seminorms are defined in \eqref{equ:Calpha}.}
\begin{equation*}
V^r(Y-\tilde Y) \lesssim \abs{y_0-\tilde y_0} +\norm{\phi-\tilde\phi}_{\sup} + \norm{\phi-\tilde\phi}_{C^\alpha} + V^r(X-\tilde X)\,.
\end{equation*}
More precisely, for intervals $[s,t] \subset [0,T]$,
\begin{equation*}
V^r(Y-\tilde Y)_{[s,t]} \le c_1(s,t)(\abs{y_0-\tilde y_0} +\norm{\phi-\tilde\phi}_{\sup} + \norm{\phi-\tilde\phi}_{C^\alpha}) + c_2(s,t) V^r(X-\tilde X)\,.
\end{equation*}

Our techniques require very precise forms of the constants $c_1,c_2$ as function of the interval $I=[s,t]$, so we derive these in the following proposition:

\begin{theorem}[cf. Theorem \ref{thm:Lipschitz-stability}]\label{thm:IntroYoungLipschitz}
Let $1 \leq r < 1+\alpha \leq 2$.
Let $\phi,\tilde\phi\in C^{1,\alpha}_b$ be bounded in the sense of \eqref{eq:C^0,alpha-bounds}, \eqref{eq:C^1,alpha-bounds}.
Let $Y,\tilde Y$ be solutions to \eqref{equ:intro1}, \eqref{eq:IntroYoungDE} with $X,\tilde X \in V^{r}$ and initial data $y_0,\tilde y_0$. Then, for small enough $T>0$ and all intervals $I\subset [0,T]$, we have
\begin{equation}
V^r\Delta Y_{I}
\leq
\gamma V^rX_{I} + 2\Phi_{0} V^r\Delta X_I
\end{equation}
with
\begin{align*}
\gamma \lesssim
((\Csew+1)\Phi_{1} + \Csew 2^{\alpha+1} \Phi_{1,\alpha} \Phi_0^{\alpha} \epsilon^{\alpha}) \Phi_{0} V^r \Delta X_{[0,T]}  +
(\Csew 2^{\alpha+1} \Phi_{1,\alpha} \Phi_0^{\alpha} \epsilon^{\alpha} + \Phi_{1}) \abs{y_{0}-\tilde{y}_{0}}
\\+ \Phi_0^{1-\alpha} \norm{\Delta\phi}_{C^{\alpha}}
+ \norm{\Delta\phi}_{\sup}.
\end{align*}
\end{theorem}

This precise form of $c_1,c_2$ allows us to apply Theorem \ref{theo:IntroVrBesovEstimate} to immediately derive the following:

\begin{theorem}[cf. Theorem \ref{thm:LipschitzBesovEstimateYoung}]\label{thm:IntroYoungLipschitzBesov}
Let $\alpha > \frac 12, \beta > 0$ such that $\frac 1\alpha < 1+\beta$ and $p\in[1,\infty]$ with $\frac 1\alpha <p$, as well as $q\in(0,\infty]$. Let $\phi,\tilde\phi\in C^{1,\beta}_b$ be bounded in the sense of \eqref{eq:C^0,alpha-bounds}, \eqref{eq:C^1,alpha-bounds}. Let $Y,\tilde Y$ be solutions to \eqref{equ:intro1}, \eqref{eq:IntroYoungDE} with $X,\tilde X \in B^\alpha_{p,q}$ and initial data $y_0,\tilde y_0$. Then, for small enough $T>0$ and all intervals $I\subset [0,T]$, we have

\begin{equation*}
\norm{Y-\tilde Y}_{B^{\alpha}_{p,q}} \lesssim \abs{y_0-\tilde y_0} + \norm{\phi-\tilde\phi}_{\sup}+\norm{\phi-\tilde\phi}_{C^\beta} + \norm{X-\tilde X}_{B^\alpha_{p,q}}\,.
\end{equation*}
\end{theorem}

\paragraph{Rough Differential equations}
Since we have similar estimates for rough paths in Theorem \ref{theo:IntroVrBesovEstimateRP}, we can do the same for rough differential equations. Consider the RDE
\begin{equation}\label{eq:IntroRDE}
	d\bfY_t = \phi(\bfY_t) d\bfX_t
\end{equation}
for a $r$-rough path $\bfX$, $r\in[2,3)$. As before, we use bold letters not only for rough paths $\bfX = (X, \bbX)$, but also for $\bfX$-controlled rough
path. That is, $\bfY = (Y, Y')$ with $Y' = \phi (Y)$ in
above situation; the pushforward $\phi (\bfY)$ is defined in Lemma
\ref{lem:controlled-composition}.
To keep technicalities in check, we only consider two solutions $\bfY,\tilde\bfY$ with different starting points $y_0,\tilde y_0$ and driving paths $\bfX,\tilde\bfX$, but the same nonlinearity $\phi$.
We present a Lipschitz estimate with very precise dependence on the underlying interval $I\subset[0,T]$:

\begin{theorem}[cf. Theorem \ref{thm:LipschitzDependenceRDE}]\label{thm:IntroRoughLipschitz}
Let $\phi\in C^{2,1}_{b}$, $\bfX,\tilde\bfX$ be continuous $r$-rough paths for $r\in[2,3)$ and let $y_0,\tilde y_0\in F$ for some Banach space $F$. Then the solutions $\bfY,\tilde\bfY$ to \eqref{eq:IntroRDE} fulfill for small enough $T>0$ and any interval $I\subset[0,T]$:
\begin{align*}
V^r(\Delta Y)_I &\lesssim (V^rX_I + V^{r/2}\bbX_{I} + V^{r/2}\tilde\bbX_{I}) \gamma + V^r(\Delta X)_I + V^{r/2}(\Delta \bbX)_{I} \\
V^{r/2}(\Delta R^{\bfY,\bfX})_{I}&\lesssim (V^rX_I^2 + V^rX_I \cdot V^r\tilde X_I +V^{r/2}\bbX_{I} + V^{r/2}\tilde \bbX_{I})\gamma\\
&\qquad + (V^r X_I+V^r\tilde X_I +V^{r/2}\bbX_I + V^{r/2}\tilde\bbX_{I})V^r(\Delta X)_I + V^{r/2}(\Delta\bbX)_{I}\nonumber
\end{align*}
where
\[
	\gamma = \abs{\Delta y_0} + V^r{\Delta X}_{[0,T]} + V^{r/2}{\Delta\bbX}_{[0,T]}\,.
\]
\end{theorem}

Applying Theorem \ref{theo:IntroVrBesovEstimateRP} to this immediately gives a Lipschitz estimate in the Besov setting:

\begin{theorem}[cf. Theorem \ref{thm:LipschitzBesovEstimateRDE}]\label{thm:IntroRoughLipschitzBesov}
Let $\phi\in C^{2,1}_b$, $\bfX,\tilde\bfX$ be $r$-rough paths for $r\in[2,3)$ and let $Y,\tilde Y$ be the corresponding solutions to \eqref{eq:IntroRDE} for small enough $T>0$. Further assume that $\frac 13 < \alpha$, $\frac 1\alpha < r\le p\le\infty$, $\frac 12\le q\le \infty$. We then get the local Lipschitz estimates in the Besov scale:
\begin{equation*}
\norm{Y-\tilde Y}_{B^\alpha_{p,q}} + \norm{R^{\bfY,\bfX}-R^{\tilde\bfY,\tilde\bfX}}_{B^{2\alpha}_{p/2,q/2}} \lesssim \abs{y_0-\tilde y_0} + \norm{X-\tilde X}_{B^\alpha_{p,q}} + \norm{\bbX-\tilde\bbX}_{B^{2\alpha}_{p/2,q/2}}
\end{equation*}
\end{theorem}

\subsubsection*{Structure of this paper}

In Section~\ref{sec:BesovSpaces}, we recall Besov and variation norms and show some embeddings in these spaces for which we found no convenient references.
We then show Theorem~\ref{theo:IntroVrBesovEstimate}.

In Section~\ref{sec:RoughPaths}, we construct the metric groups relevant to our analysis of rough paths.
We then apply the results from Section~\ref{sec:BesovSpaces} to these groups to show Theorem~\ref{theo:IntroVrBesovEstimateRP}.

In Section \ref{sec:Sewing}, we recall the sewing lemma in the setting of \cite{MR3770049}.
In line with the theme of deducing Besov estimates from variation norm estimates, we use it to recover the Besov sewing lemma from \cite{FS22}.
This requires a slight extension to the sewing lemma for $p$-metrics, which we show in Lemma~\ref{lem:sewing-limit}.

In Section~\ref{sec:LipschitzYoung}, we show variation norm Lipschitz estimates for Young differential equations, in particular Theorem~\ref{thm:IntroYoungLipschitz}.
In Section~\ref{sec:LipschitzRP}, we addess the similar but substantially more technical problem for rough differential equations, in particular obtaining Theorem~\ref{thm:IntroRoughLipschitz}.

Section \ref{sec:BesovEstimates} then puts all of these results together and derives Lipschitz estimates in the Besov setting. We first show estimates for the Young and rough integral, and then proceed to prove Theorem \ref{thm:IntroYoungLipschitzBesov} and finally Theorem \ref{thm:IntroRoughLipschitzBesov}.

\subsubsection*{Acknowledgements}
 PKF acknowledges seed support for the DFG CRC/TRR 388 “Rough Analysis and Stochastic Dynamics” and is also supported by DFG Excellence Cluster MATH+ and a MATH+ Distinguished Fellowship. HK acknowledges former support from DFG IRTG 2544 “Stochastic Analysis in Interaction”.

\section{Besov and Variation norms of rough paths}\label{sec:BesovSpaces}

\subsection{Besov norms}

Let us recall the definition of Besov spaces together with some basic embeddings. Similar embeddings for the real-valued case can be found in \cite{BF01765315}, \cite{triebel1992theory}.
Our setting is slightly more general, in that we consider Besov norms of two-parameter processes that are not necessarily the difference/distance of a one-parameter process.
We consider a fixed time horizon $T>0$ and denote the simplex by
\[
\Delta_T := \{(s,t)\in[0,T]^2~\vert~ s\le t\}\,.
\]
Then, given a $\chi:\Delta_T \to \bbR_+ := [0,\infty)$, we define
\[
	\Omega_p^\chi(t) := \sup_{0< h\le t}\left(\int_0^{T-h} \chi(s,s+h)^p \dif s\right)^{\frac 1p}\,,
\]
for any $0<p\le\infty$, with the usual modification for $p=\infty$. The Besov norm of $\chi$ is then defined as follows:

\begin{definition} \label{def:2Besov}
Let $\alpha>0, 0<p,q\le\infty$. We then set
\begin{equation*}
	\norm{\chi}_{B^\alpha_{p,q}} := \left(\int_0^T\left(\frac{\Omega_p^\chi(t)}{t^\alpha}\right)^q\frac{\dif t} t\right)^{\frac 1q}\,,
\end{equation*}
with the usual modification for $q=\infty$.
We write $\chi\in B^\alpha_{p,q}$ if the above norm is finite.
We further set
\begin{equation*}
\norm{\chi}^*_{B^\alpha_{p,q}} := \left(\int_0^T t^{-\alpha q} \left(\int_0^{T-t} \chi(s,s+t)^p\dif s\right)^{\frac qp} \frac{\dif t}t\right)^{\frac 1q}\,.
\end{equation*}
\end{definition}

\noindent This setting allows us to simultaneously consider one-parameter processes in metric spaces and two-parameter processes in normed spaces. If $f:[0,T]\to X$ for some metric space $(X,d)$, we set
\begin{equation*}
\norm{f}_{B^\alpha_{p,q}} := \norm{d(f(\cdot),f(\cdot))}_{B^\alpha_{p,q}}, \qquad \norm{f}^*_{B^\alpha_{p,q}} := \norm{d(f(\cdot),f(\cdot))}^*_{B^\alpha_{p,q}}
\end{equation*}
and if $A:\Delta_T\to E$ maps into a normed space $(E,\abs{\cdot})$, we set
\begin{equation*}
\norm{A}_{B^\alpha_{p,q}} := \norm{\abs{A_{\cdot,\cdot}}}_{B^\alpha_{p,q}}, \qquad \norm{A}^*_{B^\alpha_{p,q}} := \norm{\abs{A_{\cdot,\cdot}}}^*_{B^\alpha_{p,q}}\,.
\end{equation*}
The following assumption is motivated by the (quasi-)metric case but extends to the variation norm of $f$, see Remark~\ref{rem:VariationIsSubadditive}.
\begin{definition}
We say that $\chi$ is \emph{$\rho$-subadditive} for some $\rho>0$ if, for all $s,u,t\in[0,T]$ with $s\le u\le t$, we have
\begin{equation}\label{cond:rho-Subadditivity}
	\chi(s,t)^{\rho} \le \chi(s,u)^{\rho}+\chi(u,t)^{\rho}\,.
\end{equation}
In the case $\rho=1$, that is,
\begin{equation}\label{cond:Subadditivity}
	\chi(s,t) \le \chi(s,u)+\chi(u,t),
\end{equation}
we say that $\chi$ is \emph{subadditive}.
\end{definition}

\noindent Under this condition, it turns out that $\norm{\chi}_{B^\alpha_{p,q}}$ and $\norm{\chi}^*_{B^\alpha_{p,q}}$ are equivalent, as the following Lemma shows.

\begin{lemma}
Let $\chi:\Delta_T\to\bbR_+$ be $\rho$-subadditive in the sense of \eqref{cond:rho-Subadditivity} for some $\rho>0$.
Further assume $\alpha\in (0,1], p,q\in (0,\infty]$. Then
\begin{equation*}
\norm{\chi}^*_{B^\alpha_{p,q}}
\le \norm{\chi}_{B^\alpha_{p,q}}
\le (2^{\alpha \tilde\rho}-1)^{-1/\tilde\rho}\norm{\chi}^*_{B^\alpha_{p,q}},
\end{equation*}
where $\tilde\rho = \min(1,\rho,p,q)$.
\end{lemma}
\begin{proof}
It is clear from the definition that
\begin{equation}
	\norm{\chi}^*_{B^\alpha_{p,q}} \le \norm{\chi}_{B^\alpha_{p,q}}\,,
\end{equation}
so we only need to show the second inequality.
It suffices to consider that case $\norm{\chi}^*_{B^\alpha_{p,q}} < \infty$.

Since $\rho$-subadditivety implies $\hat\rho$-subadditivity for any $\hat\rho \in (0,\rho)$, we may assume that $\rho=\tilde\rho=\min(1,\rho,p,q)$.
Let
\[
a(h) := \left(\int_0^{T-h} \chi(s,s+h)^p \dif s\right)^{\frac 1p},
\quad 0 < h \leq T.
\]
Using $\rho$-subadditivity of $\chi$, we see that, for every $h_{1},h_{2}>0$ with $h_{1}+h_{2}\leq T$, we have
\begin{align*}
a(h_{1}+h_{2})^{\rho}
&=
\left(\int_0^{T-h_{1}-h_{2}} \chi(s,s+h_1+h_2)^{\rho (p/\rho)} \dif s \right)^{\rho^2 \frac 1{p/\rho}}
\\ &\leq
\left(\int_0^{T-h_{1}-h_{2}} (\chi(s,s+h_{1})^{\rho}+\chi(s+h_{1},s+h_{1}+h_{2})^{\rho})^{p/\rho} \dif s \right)^{\rho^2 \frac 1{p/\rho}}
\\ &\leq
\bigg( \left(\int_0^{T-h_{1}-h_{2}} (\chi(s,s+h_{1})^{\rho})^{p/\rho} \dif s \right)^{\frac 1{p/\rho}} \\&\qquad+ \left(\int_0^{T-h_{1}-h_{2}} (\chi(s+h_{1},s+h_{1}+h_{2})^{\rho})^{p/\rho} \dif s \right)^{\frac 1{p/\rho}} \bigg)^{\rho^2}
\\ &\leq
(a(h_{1})^{\frac 1\rho} + a(h_{2})^{\frac 1\rho})^{\rho^2} \\
&\leq a(h_1)^\rho + a(h_2)^\rho\,,
\end{align*}
so that $a$ is $\rho$-subadditive, in the sense that
\[
a(h_{1}+h_{2})^{\rho} \leq a(h_{1})^{\rho} + a(h_{2})^{\rho}
\]
whenever both sides are defined.

This in turn implies that
\begin{align*}
a(t)^{\rho}
&\leq
\inf_{h \in (t/3,2t/3)} a(t-h)^{\rho} + a(h)^{\rho}
\\ &\leq
\bigl( \frac{3}{t} \int_{t/3}^{2t/3} (a(t-h)^{\rho} + a(h)^{\rho})^{q/\rho} \dif h \bigr)^{\rho/q}
\\ &\lesssim
t^{\alpha\rho} \bigl( \int_{t/3}^{2t/3} (h^{-\alpha} a(h))^{q} \frac{\dif h}{h} \bigr)^{\rho/q}
\\ &\leq
t^{\alpha\rho} (\norm{\chi}^*_{B^\alpha_{p,q}})^{\rho}.
\end{align*}
In particular, $\lim_{t\to 0}a(t) = 0$.
Since every number in $(0,1]$ can be written as a sum of distinct negative integer powers of $2$, this allows us to use $\rho$-subadditivity of $a$ to conclude
\[
\sup_{h \in (0,t]} a(h)^{\rho} \leq \sum_{j=1}^{\infty} a(2^{-j}t)^{\rho}.
\]
Therefore,
\begin{align*}
\norm{\chi}_{B^\alpha_{p,q}}^{\rho}
&=
\left(\int_0^T\left(t^{-\alpha\rho} \sup_{0<h\leq t}a(h)^{\rho}\right)^{q/\rho} \frac{\dif t} t\right)^{\rho/q}
\\ &\leq
\left(\int_0^T\left(t^{-\alpha\rho} \sum_{j=1}^{\infty} a(2^{-j}t)^{\rho}\right)^{q/\rho} \frac{\dif t} t\right)^{\rho/q}
\\ &\leq
\sum_{j=1}^{\infty} \left(\int_0^T\left(t^{-\alpha\rho} a(2^{-j}t)^{\rho} \right)^{q/\rho} \frac{\dif t} t\right)^{\rho/q}
\\ &=
\sum_{j=1}^{\infty} 2^{-j\alpha\rho} \left(\int_0^{2^{-j}T}\left(t^{-\alpha\rho} a(t)^{\rho} \right)^{q/\rho} \frac{\dif t} t\right)^{\rho/q}
\\ &\leq
\sum_{j=1}^{\infty} 2^{-j\alpha\rho} (\norm{\chi}_{B^\alpha_{p,q}}^{*})^{\rho}.
\qedhere
\end{align*}
\end{proof}

\noindent In the remainder of this section, we show some Besov embeddings in our setting. We mainly follow \cite{BF01765315}, which showed similar embeddings for real-valued functions $f$.

\begin{lemma}[Dependence on $\alpha,q$]\label{lem:Besov-Embedding-Q-Alpha}
	Let $\chi:\Delta_T\to\bbR_+$, $0<\alpha,\tilde\alpha$ and $0<p,q,\tilde q\le\infty$. Assume that $\alpha <\tilde\alpha$ and $q\le\tilde q$. Then there exists a constant $c>0$ only depending on $\alpha,\tilde\alpha,q,\tilde q$ and $T$ such that
	\begin{equation*}
		\norm{\chi}_{B^\alpha_{p,q}}\le c\norm{\chi}_{B^{\tilde\alpha}_{p,\tilde q}}\,.
	\end{equation*}
\end{lemma}

\begin{proof}
The proof is simple for $q=\tilde q$, so we assume $q<\tilde q$. We also only show the proof for $q,\tilde q<\infty$, the other extensions to the infinite cases are straightforward.

Our main tool is the Hölder inequality. Set $\theta = \tilde q/q$ and $\tilde\theta$ such that $1 = \frac 1\theta + \frac 1{\tilde\theta}$. That is, $\tilde\theta = \frac{\tilde q}{\tilde q- q}$. It then follows that

\begin{align*}
	\int_0^T t^{-\alpha q} \Omega_p^\chi(t)^q \frac{\dif t} t \le \underbrace{\left(\int_0^T t^{[(\tilde\alpha-\alpha+\frac{1}{\tilde q})q-1]\tilde\theta}\dif t\right)^{\frac 1{\tilde\theta}}}_{=: c(\alpha,\tilde\alpha,q,\tilde q,T)}\left(\int_0^T t^{-\tilde\alpha \tilde q-1}\Omega_p^\chi(t)^{\tilde q} \dif t\right)^{\frac 1\theta}\,.
\end{align*}
The claim then follows from
\begin{equation*}
c(\alpha,\tilde\alpha,q,\tilde q,T)^{\tilde\theta} = \int_0^T t^{-1+(\tilde\alpha-\alpha)\frac{q\tilde q}{\tilde q-q}}\dif t <\infty\,,
\end{equation*}
since $(\tilde\alpha-\alpha)\frac{q\tilde q}{\tilde q-q} > 0$.
\end{proof}

\begin{lemma}[Dependence on p]\label{lem:Besov-Embedding-P}
	Let $\chi:\Delta_T\to\bbR_+$, $0<\alpha$ and $0<p,\tilde p,q\le\infty$ with $p\le \tilde p$. Then there exists $c(p,\tilde p, T)> 0$ such that
	\begin{equation*}
		\norm{\chi}_{B^\alpha_{p,q}}\le c(p,\tilde p,T) \norm{\chi}_{B^\alpha_{\tilde p,q}}\,.
	\end{equation*}
\end{lemma}

\begin{proof}
We again only show the case $p,\tilde p <\infty$. Set $\theta = \tilde p/p\ge 1$ and $\tilde \theta\ge 1$ such that $1= \frac 1\theta + \frac 1{\tilde \theta}$. Then the Hölder inequality gives us for all $h\in[0,T]$
\begin{align*}
\left(\int_0^{T-h} \chi(s,s+h)^p \dif s\right)^{\frac 1p} &\le \left(\int_0^{T-h} 1\dif s\right)^{\frac 1{\tilde\theta p}} \left(\int_0^{T-h} \chi(s,s+h)^{\tilde p}\dif s\right)^{\frac 1{\tilde p}} \\
&\le T^{\frac1{\tilde\theta p}}\left(\int_0^{T-h} \chi(s,s+h)^{\tilde p}\dif s\right)^{\frac 1{\tilde p}}\,.
\end{align*}
Putting this into the definition of $\norm{\chi}_{B^\alpha_{p,q}}$ shows the claim.
\end{proof}

\noindent Our next objective is analyzing the dependence on $q$, which is based on the following sum approximation.

\begin{lemma}
Let $\chi:\Delta_T\to\bbR_+$, $0<\alpha$, $0<p\le\infty$ and $0<q<\infty$. Then there exists constants $c(\alpha,q,T), C(\alpha,q,T)>0$, such that
\begin{equation}\label{ineq:BesovSumApproximation1}
c(\alpha,q,T)\left(\sum_{n=1}^\infty (2^{n\alpha } \Omega_p^\chi(2^{-n}T))^q\right)^{\frac 1q} \le \norm{\chi}_{B^\alpha_{p,q}} \le C(\alpha,q,T) \left(\sum_{n=0}^\infty (2^{n\alpha } \Omega_p^\chi(2^{-n}T))^q\right)^{\frac 1q} \,.
\end{equation}
If $\chi$ is subbadditive, it follows that there are constants $c(\alpha,q,p,T), C(\alpha,q,p,T)>0$, such that
\begin{equation}\label{ineq:BesovSumApproximation2}
c(\alpha,q,p,T)\left(\sum_{n=1}^\infty (2^{n\alpha } \Omega_p^\chi(2^{-n}T))^q\right)^{\frac 1q} \le \norm{\chi}_{B^\alpha_{p,q}} \le C(\alpha,q,p,T) \left(\sum_{n=1}^\infty (2^{n\alpha } \Omega_p^\chi(2^{-n}T))^q\right)^{\frac 1q} \,.
\end{equation}
\end{lemma}

\begin{proof}
If $\chi$ is subadditive, we immediately have
\begin{equation}
\Omega_p^\chi(T) \le 2\Omega_p^\chi(T/2)
\end{equation}
if $p\ge 1$. If $p<1$, the inequality still holds with a constant depending on $p$. Thus, \eqref{ineq:BesovSumApproximation2} is a direct consequence of \eqref{ineq:BesovSumApproximation1}. Let us now prove \eqref{ineq:BesovSumApproximation1}.

We use $\lesssim$ to hide dependences on $\alpha,q,T$. This proof is based on the observations that $\Omega_p^A(t)$ is monotonously increasing in $t$ and that $\int_{2^{-n-1}T}^{2^{-n}T} t^{-1-q\alpha}\dif t \approx T^{-q\alpha} 2^{nq\alpha} \approx \int_{2^{-n}T}^{2^{-n+1}T} t^{-1-q\alpha}\dif t$ holds. Together, these two facts allow us to calculate
\begin{align*}
\sum_{n=1}^\infty (2^{n\alpha}\Omega_p^\chi(2^{-nT}))^q &\lesssim \sum_{n=1}^\infty \int_{2^{-n}T}^{2^{-n+1}T} t^{-q\alpha-1}\dif t \cdot \Omega_p^\chi(2^{-nT})^q \\
&\lesssim \sum_{n=1}^\infty \int_{2^{-n}T}^{2^{-n+1}T} \left(\frac{\Omega_p^\chi(t)}{t^\alpha}\right)^q\frac{\dif t} t\\
&= \norm{\chi}_{B^\alpha_{p,q}}^q\,.
\end{align*}
The second inequality follows from a similar calculation.
\end{proof}

\begin{lemma}[Dependence on $q$]\label{lem:Besov-Embedding-Q}
Let $\chi:\Delta_T\to\bbR_+$, $\alpha>0,0<p\le\infty$ and $0<\tilde q\le q \le\infty$. Assume that $\chi$ is subadditive. Then there exists a $c(\alpha,p,\tilde q,q,T) > 0$ such that
\begin{equation*}
\norm{A}_{B^\alpha_{p,q}} \le c(\alpha,p,q,\tilde q,T) \norm{A}_{B^\alpha_{p,\tilde q}} 
\end{equation*}
\end{lemma}

\begin{proof}
The case $q=\infty$ is again simpler, so we concentrate on $q<\infty$. By \eqref{ineq:BesovSumApproximation2}, we have
\begin{align*}
\norm{\chi}_{B^\alpha_{p,q}} &\lesssim \left(\sum_{n=1}^\infty (2^{n\alpha } \Omega_p^\chi(2^{-n}T))^q\right)^{\frac 1q} \\
&\le \left(\sum_{n=1}^\infty (2^{n\alpha } \Omega_p^\chi(2^{-n}T))^{\tilde q}\right)^{\frac 1{\tilde q}} \\
&\lesssim \norm{\chi}_{B^\alpha_{p,\tilde q}}\,.
\end{align*}
\end{proof}

\noindent The final embedding we need is a dependence on $p$ and $\alpha$. We show a new, efficient way to prove this by making use of a monotonicity assumption:

\begin{definition}
We say that $\chi:\Delta_{T}\to\bbR_+$ is monotone, if for all $u,s,t,v\in[0,T], u\le s\le t\le v$ we have
\begin{equation}\label{cond:Monotonicity}
\chi(s,t)\le \chi(u,v)\,.
\end{equation}
\end{definition}

\begin{example}
This is motivated by processes of the form
\begin{equation*}
(s,t)\mapsto V^r \chi_{[s,t]}
\end{equation*}
where $V^r$ is the $r$-variation, see \eqref{eq:VariationNorm}.
\end{example}

\noindent Before we show the embedding itself, we need some preliminary lemmas. The first one shows that it suffices to consider small $t$ for the Besov norm:

\begin{lemma}\label{lem:smallTau}
Let $\chi:\Delta_T\to\bbR_+$ fulfill \eqref{cond:Subadditivity} and \eqref{cond:Monotonicity} and let $\alpha\in (0,1], p,q\in [1,\infty]$. Then
\begin{equation*}
	\norm{\chi}^*_{B^\alpha_{p,q}} \le 3^{1-\alpha}\left(\int_0^{T/3} t^{-\alpha q} \left(\int_0^{T-t} \chi(s,s+t)^p \dif s\right)^\frac q p \frac{\dif t}{t}\right)^{\frac 1q}\,.
\end{equation*}
\end{lemma}

\begin{proof}
We use the subadditivity of $\chi$ to split
\begin{align*}
\norm{\chi}^*_{B^\alpha_{p,q}} \le &\sum_{k=0}^2\left(\int_0^{T} t^{-\alpha q} \left(\int_0^{T-t} \chi(s+\frac k3\tau,s+\frac{k+1}3\tau)^p \dif s\right)^\frac q p \frac{\dif t}{t}\right)^{\frac 1q}
\end{align*}
For $k=0$, we substitute $t/3 = \tilde t$ to get
\begin{align*}
\left(\int_0^{T} t^{-\alpha q} \left(\int_0^{T-t} \chi(s,s+t/3)^p \dif s\right)^\frac q p \frac{\dif t}{t}\right)^{\frac 1q} &= 3^{-\alpha}\left(\int_0^{T/3} \tilde t^{-\alpha q} \left(\int_0^{T-3\tilde t} \chi(s,s+\tilde t)^p \dif s\right)^\frac q p \frac{\dif\tilde t}{\tilde t}\right)^{\frac 1q}\\
&\le 3^{-\alpha}\left(\int_0^{T/3} \tilde t^{-\alpha q} \left(\int_0^{T-\tilde t} \chi(s,s+\tilde t)^p \dif s\right)^\frac q p \frac{\dif\tilde t}{\tilde t}\right)^{\frac 1q}\,.
\end{align*}
Similar substitutions in the other terms give the claim.
\end{proof}

\begin{proposition}[Dependence on $p,\alpha$]\label{prop:Besov-Embedding-P-Alpha}
Let $\chi:\Delta_T\to\bbR_+$ fulfill the conditions \eqref{cond:Subadditivity} and \eqref{cond:Monotonicity}.
Let $\alpha,\tilde\alpha\in[0,1], p,\tilde p,q\in[1,\infty]$. Then, as long as
\begin{align*}
\alpha &\le\tilde\alpha \\
\alpha-1/p &= \tilde\alpha-1/\tilde p\,,
\end{align*}
we have
\begin{equation*}
\norm{\chi}^*_{B^\alpha_{p,q}} \le c(\alpha,\tilde\alpha,p) \norm{\chi}^*_{B^{\tilde \alpha}_{\tilde p,q}}\,,
\end{equation*}
where $c(\alpha,\tilde\alpha,p) = 3^{1-\alpha} (2^{\tilde\alpha+1/p+1}\sum_{j\ge 0} 2^{-j\alpha}+2^{\tilde\alpha})$.
\end{proposition}

\noindent Note that the conditions on $\alpha,\tilde\alpha, p,\tilde p$ imply $\tilde p\le p$. To prove this result, we split the Besov norm for $t\le T/3$ into three parts.
\begin{align*}
\left(\int_0^{T-t} \chi(s,s+t)^p\right)^{\frac 1p} &\le \underbrace{\left(\int_0^t \chi(s,s+t)^p\dif s\right)^{\frac 1p}}_{(I)} + \underbrace{\left(\int_t^{T-2t} \chi(s,s+t)^p\dif s\right)^{\frac 1p}}_{(II)} \\&\qquad+ \underbrace{\left(\int_{T-2t}^{T-t} \chi(s,s+t)^p\dif s\right)^{\frac 1p}}_{(III)}\,,
\end{align*}
and show preliminary results for these three terms:

\begin{lemma}\label{lem:smallS}
Let $\chi:\Delta_T\to\bbR_+$ fulfill \eqref{cond:Subadditivity} and \eqref{cond:Monotonicity}. For $p\in[1,\infty]$ and $t\le T/2$, we have
\begin{equation*}
(I) \le \sum_{j\ge 0} 2^{-\frac{j-1} p}t^{\frac 1p} \chi(2^{-j}t,2^{-j+1}t)\,,
\end{equation*}
and, similarly
\begin{equation*}
(III) \le \sum_{j\ge 0} 2^{-\frac{j-1} p}t^{\frac 1p} \chi(T-2^{-j+1}t,T-2^{-j}t)\,.
\end{equation*}
\end{lemma}

\begin{proof}
We only show the first inequality, the second one follows by setting $\tilde \chi(s,t) = \chi(T-t,T-s)$.
For all $s\in [0,t)$, we set $j_0(s)$ to be the unique integer $j_0\ge 1$, such that $s\in [2^{-j_0} t, 2^{-j_0+1}t)$. Then the monotonicity \eqref{cond:Monotonicity} implies
\begin{align*}
(I)^p &\le \int_0^t \chi(2^{-j_0(s)}t, 2t)^p \dif s\\
\eqref{cond:Subadditivity}&\le \int_0^t (\sum_{0\le j} 1_{\{j\le j_0(s)\}} \chi(2^{-j}t,2^{-j+1}t))^p \dif s\,.
\end{align*}
By applying Minkowski's inequality, we get
\begin{equation*}
(I) \le \sum_{0\le j} \left(\int_0^t 1_{\{j\le j_0(s)\}} \chi(2^{-j}t,2^{-j+1}t)^p\dif s\right)^{\frac 1p}\,.
\end{equation*}
One checks that $j\le j_0(s)$ if and only if $s<2^{-j+1}t$, giving
\begin{equation*}
(I) \le \sum_{0\le j} \chi(2^{-j}t,2^{-j+1}t) \left(\int_0^{2^{-{j+1}}t} 1 \dif s\right)^{\frac 1p}\,.
\end{equation*}
This shows the claim.
\end{proof}

\begin{lemma}\label{lem:bigS}
Let $\chi:\Delta_T\to\bbR_+$ fulfill \eqref{cond:Monotonicity} and let $p,\tilde p\in[1,\infty]$ with $\tilde p\le p$ and $t\le T/3$. Then
\begin{equation*}
(II) \le t^{\frac 1p- \frac 1{\tilde p}} \left(\int_0^{T-{2t}} \chi(s,s+2t)^{\tilde p} \dif s \right)^{\frac 1{\tilde p}}\,.
\end{equation*}
\end{lemma}
\begin{proof}
Let $a(s):=\chi(s,s+2t) 1_{s\in [0,T-2t]}$.
Then, using \eqref{cond:Monotonicity}, one can show that for $s\in [t,T-2t]$, we have $\chi(s,s+t) \leq (a*\rho)(s)$, where $\rho=t^{-1}1_{[0,t]}$.
Therefore, by Young's convolution inequality with exponents $1/p+1 = 1/\tilde{p}+1/r$, we obtain
\[
(II)
\leq
\left(\int_{\bbR} \rho(s)^r \dif s\right)^{\frac 1r} \left(\int_\bbR a(s)^{\tilde p}\dif s\right)^{\frac 1{\tilde p}}
=
t^{-1+1/r} \left(\int_\bbR a(s)^{\tilde p}\dif s\right)^{\frac 1{\tilde p}}
=
t^{\frac 1p- \frac 1{\tilde p}} \left(\int_\bbR a(s)^{\tilde p}\dif s\right)^{\frac 1{\tilde p}},
\]
which shows the claim.
\end{proof}

\noindent With these, we can now show the Proof of Proposition \ref{prop:Besov-Embedding-P-Alpha}:

\begin{proof}
By Lemma \ref{lem:smallTau}, we have
\begin{align*}
\norm{\chi}^*_{B^\alpha_{p,q}} \le 3^{1-\alpha}\left(\int_0^{T/3} t^{-\alpha q} ((I)+(II)+(III))^q \frac{\dif t}t\right)^{\frac 1q}\,.
\end{align*}
Using Lemma \ref{lem:smallS}, we have
\begin{align*}
\left(\int_0^{T/3} t^{-\alpha q}(I)^q \frac{\dif t}t \right)^{\frac 1q} &\le \sum_{j\ge 0} \left(\int_0^{T/3} t^{-\alpha q} \left(2^{-(j-1)/p}t^{\frac 1p} \chi(2^{-j}t,2^{-j+1}t)\right)^q\frac {\dif t}t\right)^{\frac 1q} \\
&\le \sum_{j\ge 0} \Bigg(\int_0^{T/3}t^{-\alpha q+\frac qp} 2^{-(j-1)\frac qp} (2^{-j}t)^{-\frac q{\tilde p}}\\&\qquad\qquad \times\left(\int_0^{2^{-j}t} \chi(s,s+2^{-j+1}t)^{\tilde p} \dif s\right)^{\frac q{\tilde p}} \frac{\dif t}{t}\Bigg)^{\frac 1q} \\
&= \sum_{j\ge 0} \left(\int_0^{T/3}(2^{-j}t)^{-\tilde \alpha q}2^{-j\alpha q+\frac q p}\left(\int_0^{2^{-j}t} \chi(s,s+2^{-j+1}t)^{\tilde p}\dif s\right)^{\frac q{\tilde p}}\frac{\dif t}t\right)^{\frac 1q}\\
(\text{substitute }h= 2^{-j+1}t)&= \sum_{j\ge 0} 2^{-j\alpha+\tilde\alpha+\frac 1 p} \left(\int_0^{2^{-j+1}T/3} h^{-\tilde \alpha q}\left(\int_0^{h/2} \chi(s,s+h)^{\tilde p} \dif s\right)^{\frac q{\tilde p}}\frac{\dif h} h\right)^{\frac 1q}\\
&\le c_1(\alpha,\tilde\alpha,p) \norm{\chi}^*_{B^{\tilde \alpha}_{\tilde p,q}}\,.
\end{align*}
With constant $c_1(\alpha,\tilde\alpha,p) = (\sum_{j\ge 0} 2^{-j\alpha})2^{\frac 1 p+\tilde\alpha}$. Similarly, the third term fulfills
\begin{equation*}
\left(\int_0^{T/3} t^{-\alpha q}(III)^q \frac{\dif t}t \right)^{\frac 1q} \le c_1(\alpha,\tilde\alpha,p)  \norm{\chi}^*_{B^{\tilde\alpha}_{\tilde p,q}}\,,
\end{equation*}
Finally, the second term fulfills
\begin{align*}
\left(\int_0^{T/3} t^{-\alpha q}(II)^q \frac{\dif t}t \right)^{\frac 1q} &\le \left(\int_0^{T/3} t^{-\alpha q+\frac qp-\frac q{\tilde p}}\left(\int_0^{T-2t}\chi(s,s+2t)^{\tilde p} ds\right)^{\frac q{\tilde p}} \frac{\dif t}t \right)^{\frac 1q} \\
(\text{substitute }h=2t)&= 2^{\tilde \alpha}\left(\int_0^{\frac 23 T} h^{-\tilde\alpha q}\left(\int_0^{T-h}\chi(s,s+h)^{\tilde p} ds\right)^{\frac q{\tilde p}}\frac{\dif h}h\right)^{\frac 1q}\\
&\le 2^{\tilde\alpha} \norm{\chi}^*_{B^{\tilde\alpha}_{\tilde p,q}}\,.
\end{align*}
Adding these three terms gives the result.
\end{proof}

\begin{remark}\label{rem:reduceQ}
Since $\norm{\cdot}_{B^\alpha_{p,q}}$ is a $\min(1,p,q)$-norm in the sense of Definition \ref{def:pNorm}, we assume that it is possible, although tedious, to show Proposition \ref{prop:Besov-Embedding-P-Alpha} for $p,\tilde p,q\in(0,\infty]$.
\end{remark}

\subsection{Variation norm and composition with Besov norms}

For any $r>0,\chi:\Delta_T\to \bbR_+$ and an interval $I\subset[0,T]$, we define the variation norm by
\begin{equation}\label{eq:VariationNorm}
	V^r \chi_I := \sup_\pi\left(\sum_{i=1}^{N-1} \chi_{ \pi_i,\pi_{i+1}} ^r\right)^{\frac 1r}\,,
\end{equation}
where the supremum goes over all finite increasing sequences $\pi = \{ \pi_1<\dots<\pi_N \} \subset I$.
Similar to the Besov-case, we set the variation norm of a one-parameter process $f:[0,T]\to X$ for a metric space $(X,d)$, as well as the variation norm of a two-parameter process $A:\Delta_T\to E$ for a normed space $(E,\abs{\cdot})$ to be
\[
V^r f_I := V^r(d(f(\cdot),f(\cdot)))_I\,, \qquad V^r A_I := V^r(\abs{A_{\cdot,\cdot}})_I\,.
\]
If $I= [0,T]$, we write $V^r\chi := V^r\chi_{[0,T]}$ and if $V^r\chi <\infty$, we say $\chi\in V^r$. Let us recall the following embedding between Besov and variation spaces.
\begin{proposition}[\cite{FS22}, Proposition 2.3]\label{prop:EmbeddingVrInBesov}
For every $r\geq 1$ and $T>0$, there exists a $c(r,T)>0$ such that, for every complete metric space $(E,d)$, every function $f : [0,T] \to E$ with $f \in B^{1/r}_{r,1}([0,T])$ coincides almost everywhere with a continuous function, again denoted by $f$, and this continuous function satisfies
\begin{equation*}
V^r f \le c(r,T)\norm{f}_{B^{1/r}_{r,1}}\,.
\end{equation*}
\end{proposition}

\begin{corollary}\label{cor:EmbeddingVrInBesov}
Let $r\ge 1, \alpha>0, 0<q,p\le \infty$ such that $0<\frac 1\alpha< r\le p$.
Then there exists a $c(r,\alpha,p,q,T)>0$ such that, for any complete metric space $(X,d)$, any function $f:[0,T]\to X$ with $f \in B^\alpha_{p,q}$ coincides almost everywhere with a continuous function, and that continuous function satisfies
\begin{equation*}
V^r f \le c(r,\alpha,p,q,T) \norm{f}_{B^\alpha_{p,q}}\,.
\end{equation*}
\end{corollary}

\begin{proof}
We hide constants depending on $r,\alpha,p,q,T$ in $\lesssim$.

Assume $q\ge 1$. Then we use Lemma \ref{lem:Besov-Embedding-Q-Alpha} and Lemma \ref{lem:Besov-Embedding-P} to get
\[
	\norm{f}_{B^{\frac 1r}_{r,1}} \lesssim \norm{f}_{B^{\alpha}_{r,q}} \lesssim \norm{f}_{B^{\alpha}_{p,q}}\,.
\]
If $q<1$, we use Lemma \ref{lem:Besov-Embedding-Q-Alpha}, Lemma \ref{lem:Besov-Embedding-P} and Lemma \ref{lem:Besov-Embedding-Q} to calculate
\[
	\norm{f}_{B^{\frac 1r}_{r,1}} \lesssim \norm{f}_{B^{\alpha}_{r,1}} \lesssim \norm{f}_{B^{\alpha}_{p,1}}\lesssim \norm{f}_{B^{\alpha}_{p,q}}\,.
\]
In either case, we conclude by Proposition~\ref{prop:EmbeddingVrInBesov}.
\end{proof}

\noindent We use the notation
\[
	V^r\chi_{s,t} := V^r\chi_{[s,t]}\,,
\]
allowing us to think of $V^r\chi$ as a two-parameter process itself.

\begin{remark}\label{rem:VariationIsSubadditive}
If $\chi$ is subadditive in the sense of \eqref{cond:Subadditivity} and $r\ge 1$, then one easily checks that $(s,t)\mapsto V^r\chi_{s,t}$ is also subadditive.
\end{remark}

\noindent It follows that the expression $\norm{V^r\chi}_{B^\alpha_{p,q}}$ is well defined. We aim to show an estimate of this nested norm. It is based on the following dyadic decomposition of the $V^r$ norm, similar to \cite[Lemma 2.5]{MEZK2020}.

\begin{lemma}\label{lem:VrDyadicExpansion}
Let $\chi:\Delta_T\to\bbR^+$ fulfill the subadditivity condition \eqref{cond:Subadditivity} and let $r\in[1,\infty)$. Then, for any $k\in\bbN$ such that $2^k\le T$, we have
\begin{equation*}
V^r\chi_{\bbN\cap[0,2^k]} \le 2^{\frac{r-1}r} \sum_{l\ge 0}\left(\sum_{m=0}^{2^l-1}\chi(2^{k-l}m, 2^{k-l}(m+1))^r\right)^{\frac 1r}\,.
\end{equation*}
\end{lemma}

\begin{proof}
The proof is the same as in \cite[Lemma 2.5]{MEZK2020}.
\end{proof}

\noindent A simple rescaling argument gives us the following.

\begin{corollary} \label{cor:VrDyadicExpansion}
Let $\chi:\Delta_T\to\bbR_+$ fulfill \eqref{cond:Subadditivity} and let $r\in[1,\infty)$. Then, we have for all $n\in\bbN$:
\begin{equation}\label{ineq:VrDyadicExpansion1}
 	V^r\chi_{[0,T]\cap 2^{-n}T\bbN}\le 2^{\frac{r-1}r}\sum_{l=0}^n\left(\sum_{m=0}^{2^l-1}\chi(2^{-l}Tm,2^{-l}T(m+1))^r\right)^{\frac 1r}\,.
\end{equation}
For continuous $\chi$, it follows that
\begin{equation}\label{ineq:VrDyadicExpansion2}
	V^r\chi_{[0,T]}\le 2^{\frac{r-1}r}\sum_{l=0}^\infty\left(\sum_{m=0}^{2^l-1}\chi(2^{-l}Tm,2^{-l}T(m+1))^r\right)^{\frac 1r}\,.
\end{equation}
\end{corollary}

\begin{proof}
\eqref{ineq:VrDyadicExpansion1} follows from applying Lemma \ref{lem:VrDyadicExpansion} to $\tilde \chi(s,t) := \chi(2^{-n}T s, 2^{-n}T t)$. A simple limit argument then gives \eqref{ineq:VrDyadicExpansion2}.
\end{proof}

\noindent Let us now show the main result of this section.

\begin{theorem}\label{theo:VrBesovEstimate}
Let $\chi:[0,T]\to\bbR_+$ be continuous and fulfill the subadditivity condition \eqref{cond:Subadditivity}. Let $r\in[1,\infty)$ and $\alpha>0, p,q\in(0,\infty]$ such that
\[
	\frac 1\alpha < r\le p\,.
\]
Then there exists a constant $c=c(\alpha, r,p,q)$ such that
\begin{equation*}
\norm{V^r\chi}_{B^\alpha_{p,q}}\le c\norm{\chi}_{B^\alpha_{p,q}}\,.
\end{equation*}
\end{theorem}

\begin{proof}
We introduce some notation for this proof. 
\[
\ell_l^r \chi := \left(\sum_{m=0}^{2^l-1} \chi(2^{-l} m, 2^{-l}(m+1))^r\right)^{\frac 1r}\,,
\]
and for some function $f:[0,T]\to \bbR$, density $\lambda$ on $[0,T]$ and $0\le s\le t\le T$, we write
\[
L^q_{\lambda(\tau)\dif\tau,[s,t]} f(\tau) := \left(\int_s^t \abs{f(\tau)}^q \lambda(\tau)\dif\tau\right)^{\frac 1q}\,.
\]
We apply \eqref{ineq:VrDyadicExpansion2} to $\chi$ and calculate
\begin{align*}
\norm{V^{r}\chi}_{B^{\alpha}_{p,q}} &\lesssim L^{q}_{d\tau/\tau,[0,T]}\tau^{-\alpha} \sup_{0\le h\le \tau} L_{\dif x,[0,T-h]}^{p}\sum_{l\ge 0} \ell^{r}_l \chi(x+h\cdot,x+h\cdot) \\
&\lesssim L^{q}_{\dif\tau/\tau,[0,T]}\sum_{l\ge 0} \tau^{-\alpha}\ell^{r}_l \sup_{0\le h\le \tau} L_{\dif x,[0,T-h]}^{p}\chi(x+h\cdot,x+h\cdot) \\
&= L^{q}_{\dif\tau/\tau,[0,T]}\sum_{l\ge 0} \tau^{-\alpha}\\ &\qquad\times\left( \sum_{j=0}^{2^l-1}\left(\sup_{0\le h\le 2^{-l}\tau} L_{\dif x,[0,T-2^{l}h]}^{p}\chi(x+hj,x+h(j+1))\right)^r\right)^{1/r}\\
&\le L^{q}_{\dif\tau/\tau,[0,T]}\sum_{l\ge 0} \tau^{-\alpha}2^{l/r}\sup_{0\le h\le 2^{-l}\tau} L_{\dif x,[0,T-h]}^{p}\chi(x,x+h)\\
(\text{substitute }\tau\mapsto2^l\tau)&\lesssim \sum_{l\ge 0} 2^{l(1/r-\alpha)} L^{q}_{\dif\tau/\tau,[0,2^{-l}T]} \tau^{-\alpha}\sup_{0\le h\le\tau} L_{\dif x,[0,T-h]}^{p} \chi(x,x+h)\\
&\lesssim \norm{\chi}_{B^{\alpha}_{p,q}}\,,
\end{align*}
where $\lesssim$ hides constants depending on $\alpha,r,p,q$.
\end{proof}

\begin{remark}\label{rem:VrBesovEstimate}
Theorem~\ref{theo:IntroVrBesovEstimate} follows from Theorem~\ref{theo:VrBesovEstimate} applied with $\chi(s,t)=d(f(s),f(t))$.
\end{remark}

\section{Rough paths and nested norms}\label{sec:RoughPaths}

In this section, we introduce rough paths and controlled rough paths along with some basic properties. We then develop metric groups relevant to rough paths, which enables us to reduce several two-parameter processes to paths within these groups. This reduction directly leads to Theorem \ref{theo:IntroVrBesovEstimateRP} as a consequence of Theorem \ref{theo:VrBesovEstimate}. Additionally, we present several Besov embeddings for rough paths.

\subsection{Rough paths}

Rough paths were introduced in \cite{Lyo98} and controlled paths in \cite{MR2091358}. For a long time, the theory concentrated on Hölder continuous paths; a good exposition of this case is in the book \cite{FH20}. The treatment of $V^{r}$ paths is adapted from \cite{MR3770049}, for more details see \cite{FV10}. For an introduction to Besov rough paths, see \cite{FS22}.

Let $(E,\abs{\cdot})$ be a normed space, and assume that the tensor product $E\otimes E$ is equipped with a compatible norm, i.e. a norm which fulfills
\begin{equation*}
\abs{x\otimes y} = \abs{x}\abs{y}
\end{equation*}
for all $x,y\in E$. For example, $\bbR^d, (\bbR^d)^{\otimes 2} \cong \bbR^{d\times d}$ equipped with the usual $p$-norm for any $p\ge 1$ fulfill this. A level $2$ rough path is then defined as follows:

\begin{definition}
Let $r\in[2,3)$. An $r$-rough path consists of a map $(s,t)\mapsto \bfX_{s,t} := (X_{s,t},\bbX_{s,t})\in E\oplus E^{\otimes 2}$, such that:
\begin{itemize}
\item $X$ is of finite $r$-variation and $\bbX$ is of finite $r/2$-variation.
\item We have for all $s\le u\le t$ that
\begin{equation}\label{eq:Chen}
	(1 + X_{s,u} + \bbX_{s,u})\otimes (1 + X_{u,t} + \bbX_{u,t}) = 1+X_{s,t}+\bbX_{s,t}\,,
\end{equation}
\end{itemize}
where $\otimes$ is understood as the truncated tensor product in $E\oplus E^{\otimes 2}$.
\end{definition}

\noindent Writing out the second property in components gives the following:
\begin{align*}
	X_{s,t} &= X_{s,u}+ X_{u,t}\\
	\bbX_{s,t} &= \bbX_{s,u}+\bbX_{u,t} + X_{s,u}\otimes X_{u,t}\,.
\end{align*}
where the second identity is called Chen's identity. Let us recall the definition of a controlled rough path:
\begin{definition}[Controlled path]
Let $\bfX=(X,\bbX)$ be an $r$-rough path in $E$ and let $F$ be some normed space.
An \emph{$\bfX$-controlled $r$-rough path} $\bfY = (Y,Y')$ consists of $Y : [0,T] \to F$, $Y' : [0,T] \to L(E, F)$ such that $Y' \in V^{r}$ and
\[
R^{\bfY,\bfX}_{s,t} := Y_{s,t} - Y'_{s} X_{s,t} \in V^{r/2}.
\]
\end{definition}

\noindent Note that we do not require $Y\in V^r$, as this implicitly follows from the bounds on $Y', R^{\bfY,\bfX}$. We further introduce the notion of a controlled $(r,\alpha)$-rough path, which consists of $Y: [0,T]\to F,Y':[0,T]\to L(E,F)$, such that
\[
	Y\in V^r, \qquad Y'\in V^{r/\alpha},\qquad R^{\bfY,\bfX}\in V^{r/(\alpha+1)}\,.
\]
This is motivated by the controlled rough path $\phi(\bfY):= (\phi(Y),D\phi(Y)Y')$ for a $1+\alpha$-Hölder function $\phi$, see Lemma \ref{lem:controlled-composition} for details. The case $\alpha =1$ recovers classical controlled $r$-rough paths. We denote the supremum norm by
\[
	\norm{Y}_{\sup} := \sup_{s\in[0,T]} \abs{Y(s)}\,.
\]
If we have two rough paths $\bfX,\tilde \bfX$, we denote by
\[
	\Delta X := X-\tilde X\,,
\]
with similar notation extending to $\Delta\bbX$, as well as $\Delta Y,\Delta Y',\Delta R^{\bfY,\bfX}$ if we have an $\bfX$-controlled rough path $\bfY$ and $\tilde\bfX$-controlled rough path $\tilde\bfY$.

\begin{lemma}[Implicit bound]
\label{lem:controlled-implicit}
Let $\bfX$ be an $r$-rough path and $\bfY$ an $\bfX$-controlled path.
Then
\begin{equation}\label{ineq:implicitBound}
V^r Y_I \leq \norm{Y'}_{\sup} V^r X_I + V^{r/2}(R^{\bfY,\bfX})_I.
\end{equation}
for all intervals $I\subset[0,T]$. Furthermore, if we have another rough path $\tilde\bfX$ and controlled rough path $\tilde\bfY$, we have the stability estimate
\begin{equation}\label{ineq:implicitBoundStability}
V^r(\Delta Y)_{I} \leq
\norm{\Delta Y'}_{\sup} V^r X_{I}
+ \norm{\tY'}_{\sup} V^r(\Delta X)_{I}
+ V^{r/2}(\Delta R^{\bfY,\bfX})_{I}.
\end{equation}
\end{lemma} 

\begin{proof}
\eqref{ineq:implicitBound} follows directly from the definition of $R^{\bfY,\bfX}$. To see \eqref{ineq:implicitBoundStability}, one writes
\begin{multline*}
Y_{s,t} - \tY_{s,t}
=
(Y'_{s} X_{s,t} + R^{\bfY,\bfX}_{s,t}) - (\tY'_{s} \tX_{s,t} + R^{\tilde\bfY,\tilde\bfX}_{s,t})
\\=
(Y'_{s}-\tY_{s}) X_{s,t}
+
\tY_{s} (X_{s,t} - \tX_{s,t})
+ (\Delta R^{\bfY,\bfX}_{s,t}).
\end{multline*}
\end{proof}

\begin{example}\label{ex:RPAreControlledRP}
For every rough path $(X,\bbX)$, the path $Y_t := \bbX_{0,t}$ is an $\bfX$-controlled rough path with $Y'_t = X_t$, as the following shows:
\[
R^{\bfY,\bfX}_{s,t} = \bbX_{0,t}-\bbX_{0,s} - X_{0,s}\otimes X_{s,t} = \bbX_{s,t}\in V^{r/2}\,.
\]
Note that we use the embedding $E$ embedds into $L(E,E\otimes E)$ via $x\mapsto (x\otimes\cdot)$.
\end{example}

\subsection{Groups relevant to rough paths}

One can equip the space $E\oplus E^{\otimes 2}$ with a group product $(x,y)\mapsto x*y$, such that for every rough path $\bfX$ we have
\[
	\bfX_{s,t} = \bfX_{0,s}^{-1}*\bfX_{0,t}\,,
\]
turning $\bfX$ into a one-parameter process. See e.g. \cite[Chapter 2]{FH20}. We want to construct similar groups to estimate the two-parameter processes $\Delta \bbX$ and $\Delta R^{\bfX,\bfY}$ for two rough paths $\bfX,\tilde\bfX$ and controlled rough paths $\bfY,\tilde\bfY$.

To this end, consider two linear spaces $X_1,X_2$ equipped with a linear map $B: X_1\times X_1 \to X_2$. We claim that $X_1\times X_2$ can be seen as a metric group:

\begin{lemma}\label{lem:homogeneous-subadditive-functional}
Let $X_{1},X_{2}$ be vector spaces and $B:X_{1}\times X_{1} \to X_{2}$ a bilinear map.
Then the set $X_{1}\times X_{2}$ equipped with the product
\begin{equation*}
(a,r) *_{B} (\tilde a,\tilde r) = (a+\tilde a, r+\tilde r+B(a,\tilde a))
\end{equation*}
is a group. Suppose that $X_{1},X_{2}$ are equipped with norms and that the bilinear map $B$ is bounded with respect to these norms with $\norm{B} \leq 2$, where $\norm{B}$ is the operator norm.
Then,
\begin{equation}\label{eq:defNorm}
\norm{(a,r)} = \max\left(\norm{a}_{X_{1}},\norm{r}_{X_{2}}^{\frac 12}\right)
\end{equation}
is subadditive in the sense
\[
\forall w,\tilde w \in X_{1}\times X_{2}
\quad \norm{w *_{B} \tilde w} \leq \norm{w} + \norm{\tilde w}.
\]
\end{lemma}
\begin{proof}
It is clear that $X_1\times X_2$ is a group, so let us show the subadditivity of $\norm{\cdot}$. Let $w = (a,r),\tilde w= (\tilde a,\tilde r)\in X_{1}\times X_{2}$. Observe that
\begin{align*}
\norm{r+\tilde r+ B(a,\tilde a)}_{X_2} &\le \norm{r}_{X_2} + \norm{\tilde r}_{X_2} + \norm{B(a,\tilde a)}_{X_2} \\
&\le \norm{w}^2 + \norm{\tilde w}^2 + 2\norm{w}\norm{\tilde w} = (\norm{w}+\norm{\tilde w})^2\,.
\end{align*}
The claim easily follows.
\end{proof}

\noindent As stated above, this space can be equipped with a metric as follows.

\begin{corollary}\label{cor:MetricInG}
In the situation of Lemma~\ref{lem:homogeneous-subadditive-functional}, the functional
\[
\norm{w}_{\mathrm{sym}} = \max(\norm{w}, \norm{w^{-1}})
\]
is subadditive and
\[
d(w,\tilde w) := \norm{w * \tilde w^{-1}}_{\mathrm{sym}}
\]
is a left-invariant metric on $(X_{1}\times X_{2}, *_{B})$.
\end{corollary}
\noindent Note that the construction of our norm is similar to \cite[Section 2.2]{TapiaZambotti2020}. See also \cite[Theorem 2]{MR1067309} for an alternative construction of homogeneous metrics. Since $(a,r)^{-1} = (-a,-r+B(a,a))$ and
\[
\norm{-r+B(a,a)}_{X_2} \leq \norm{r}_{X_2} + 2 \norm{a}_{X_2}^{2} \leq 3\norm{(a,r)}^{2},
\]
we have
\[
\norm{(a,r)} \leq \norm{(a,r)}_{\mathrm{sym}} \leq \sqrt{3} \norm{(a,r)}.
\]

\begin{example}\label{ex:GroupForX}
If we set $X_1 = E$, $X_2 = E^{\otimes 2}$ with the bilinear map $B(x,y) = x\otimes y$, we get the typical group for rough paths, see also \cite[Section 2.3]{FH20}. This shows that $\bfX \in E\oplus E^{\otimes 2}$ is indeed a one-parameter process in $G_1 = E\oplus E^{\otimes 2}$. By setting $\norm{x}_{X_1} = \abs{x}$ for all $x\in E$ and $\norm{y}_{X_2} = 2 \abs{y}$ for all $y\in E\otimes E$, we recover the subadditive norm
\[
	\norm{(X,\bbX)} = \max\left( \abs X, \left(2\abs{\bbX}\right)^{\frac 12}\right)
\] 
constructed in \cite[Section 2.2]{TapiaZambotti2020}.
\end{example}

\begin{example}\label{ex:GroupForDeltaX}
Now consider two rough paths $\bfX,\tilde\bfX$ over $E$. We claim that \\$F_{s,t} = (X_{s,t},\tilde X_{s,t}, \Delta X_{s,t},\Delta \bbX_{s,t})$ is a one-parameter process in the group
\[
	G_2 = (E\oplus E\oplus E) \oplus E^{\otimes 2}\,.
\]
To this end, recall that the $\delta$-operator of a two-parameter process $A:\Delta_T\to E$ is defined by
\begin{equation}\label{eq:DeltaOperator}
	\delta A_{s,u,t} = A_{s,t}-A_{s,u}-A_{u,t}\,.
\end{equation}
Direct calculation leads to
\[
	\delta(\Delta \bbX)_{s,u,t} = \Delta X_{s,u}\otimes X_{u,t} + \tilde X_{s,u}\otimes \Delta X_{u,t}\,, 
\]
showing that $F_{s,u} * F_{u,t} = F_{s,t}$ for the product
\[
	(x,\tilde x,\bar x, X) * (y,\tilde y,\bar y, Y) = (x+y, \tilde x+\tilde y, \bar x+\bar y, X + Y + \bar x\otimes y + \tilde x\otimes\bar y)\,.
\]
This fits our setting with 
\[
X_1 = E\oplus E\oplus E, \qquad \hspace{20pt}X_2 = E^{\otimes 2},\qquad B((x,\tilde x,\bar x),(y,\tilde y,\bar y)) = \bar x\otimes y+\tilde x\otimes\bar y\,.
\]
It will be useful for the proof of Lemma \ref{lem:BesovRPAreVariationRP} and Proposition \ref{prop:VrBesovEmbeddingDeltaX2} in Section \ref{sec:NestedNormsRP} to incorporate two constants $N,N_\Delta > 0$ into our norms, so let us note here that for any $N,N_\Delta > 0$, we can chose
\[
\norm{(x,\tilde x,\bar x)}_{X_1} = \max\left(\frac{\abs{x}} N, \frac{\abs{\tilde x}}N,\frac{\abs{\bar x}}{N_\Delta}\right) \qquad \norm{X}_{X_2} = \frac{\abs{X}}{NN_\Delta}\,.
\]
\end{example}

\begin{example}\label{ex:GroupForR}
Let us build a group for $\Delta R^{\bfY,\bfX}_{s,t}$. To this end, let $\bfX,\tilde\bfX$ be rough paths and let $\bfY$ be a $\bfX$-controlled rough path and $\tilde\bfY$ a $\tilde\bfX$-controlled rough path in $F$. We look at
\[
	Z_{s,t} = ((\tilde Y_{s,t}',\Delta Y_{s,t}',X_{s,t},\Delta X_{s,t}),\Delta R^{\bfY,\bfX}_{s,t})\,.		
\]
This lies in
\[
	G_3 = (L(E,F)\oplus L(E,F)\oplus E\oplus E) \oplus F\,.
\]
This can be turned into a group by using
\[
	X_1 = (L(E,F)\oplus L(E,F)\oplus E\oplus E), \qquad X_2 = F, \qquad B((a,b,c,d),(\tilde a,\tilde b,\tilde c,\tilde d)) = b\tilde c + a\tilde d\,.
\]
One then uses the $\delta$-operator
\[
	\delta \Delta R^{\bfY,\bfX}_{s,u,t} = \Delta Y'_{s,u}X_{u,t} + \tilde Y'_{s,u}\Delta X_{u,t}\,,
\]
together with direct calculation to show $Z_{s,u}*Z_{u,t} = Z_{s,t}$ for all $s\le u\le t$. Thus, $Z$ is a one-parameter process in $G_3$. As in Example \ref{ex:GroupForDeltaX}, it will be useful for the proof of Proposition \ref{prop:Vr-Besov-Embedding-RYX} to incorporate some constants $N_a, N_b,N_c, N_d >0$ into the norms. We set $N_r = N_a N_d+N_b N_c$ and set
\[
	\norm{(a,b,c,d)}_{X_1} = \max\left(\frac{\abs a}{N_a},\frac{\abs b}{N_b},\frac{\abs c}{N_c},\frac{\abs d}{N_d}\right), \qquad \norm{r}_{X_2} = \frac{\abs r}{N_r}\,.
\]
Direct calculation gives $\norm{B} \le 2$.
\end{example}

\subsection{Besov and variation estimates for rough paths}\label{sec:NestedNormsRP}

Let us now apply some of the estimates from Section \ref{sec:BesovSpaces} to the above examples. We start by applying Corollary \ref{cor:EmbeddingVrInBesov} to Example \ref{ex:GroupForX} to show that Besov rough paths à la \cite{FS22}, that is maps $\bfX:(s,t)\mapsto (X_{s,t},\bbX_{s,t})$ fulfilling \eqref{eq:Chen} as well as
\[
	\norm{X}_{B^\alpha_{p,q}} <\infty\,,\qquad \norm{\bbX}_{B^{2\alpha}_{p/2,q/2}}<\infty\,
\] 
for certain parameters $\alpha>0, p,q\in(0,\infty]$, are rough paths in the variation setting. 

\begin{lemma}\label{lem:BesovRPAreVariationRP}
	Let $r\in[2,3)$ and $\bfX = (X,\bbX)$ be an $r$-rough path. Assume that $\frac 1\alpha<r\le p\le\infty$ and $0<q\le\infty$. Then there exists a $c(r,\alpha,p,q,T)>0$ such that
	\begin{align*}
		V^r X &\le c(r,\alpha,p,q,T) \norm{X}_{B^\alpha_{p,q}}\\
		V^{r/2} \bbX &\le c(r,\alpha,p,q,T) (\norm{X}_{B^\alpha_{p,q}}^2 + \norm{\bbX}_{B^{2\alpha}_{p/2,q/2}})\,.
	\end{align*}
	If $\tilde\bfX$ is another $r$-rough path, we have (possibly for another constant $c(r,\alpha,p,q,T)>0$)
	\begin{align*}
	V^r(\Delta X) &\le c(r,\alpha,p,q,T) \norm{\Delta X}_{B^\alpha_{p,q}} \\ 
	V^{r/2}(\Delta \bbX) &\le c(r,\alpha,p,q,T) \bigg((\norm{X}_{B^\alpha_{p,q}}+ \norm{\tilde X}_{B^\alpha_{p,q}})\norm{\Delta X}_{B^\alpha_{p,q}} + \norm{\Delta\bbX}_{B^{2\alpha}_{p/2,q/2}}\bigg)\,. 
	\end{align*}
\end{lemma}

\begin{proof}
$V^r X \lesssim \norm{X}_{B^\alpha_{p,q}}$ and $V^r(\Delta X) \lesssim \norm{\Delta X}_{B^\alpha_{p,q}}$ are direct applications of Corollary \ref{cor:EmbeddingVrInBesov} to $X, \Delta X$. By Example \ref{ex:GroupForX}, $\bfX$ is a path in the metric space $G_1$ with metric $d(\bfX_s,\bfX_t) = \norm{\bfX_{s,t}}$ and $ \norm{(X_{s,t},\bbX_{s,t})}\sim \abs{X_{s,t}}+\abs{\bbX_{s,t}}^{\frac 12}$ for all $s\le t$. Applying Corollary \ref{cor:EmbeddingVrInBesov} to $\bfX$ gives us
\begin{align*}
(V^{r/2} \bbX)^{\frac 12} &= V^r(\abs{\bbX_{\cdot,\cdot}}^{\frac 12}) \lesssim V^r(\bfX) \lesssim \norm{\bfX}_{B^\alpha_{p,q}} \lesssim \norm{X}_{B^\alpha_{p,q}} + \norm{\bbX}_{B^{2\alpha}_{p/2,q/2}}^\frac 12\,.
\end{align*}
For our last inequality, we consider the path $F_{s,t} = (X_{s,t}, \tilde X_{s,t}, \Delta X_{s,t}, \Delta\bbX_{s,t})$ in $G_2$ from Example \ref{ex:GroupForDeltaX}. 
Applying Corollary \ref{cor:EmbeddingVrInBesov} to $F$ gives for any $N,N_\Delta > 0$:
\begin{align*}
V^{r/2} \left(\frac{\Delta\bbX}{NN_\Delta}\right)^{\frac 12} &\le V^r F \\
&\lesssim \norm{F}_{B^\alpha_{p,q}}\\
&\lesssim \frac{\norm{X}_{B^\alpha_{p,q}}}N +\frac{\norm{\tilde X}_{B^\alpha_{p,q}}}N + \frac{\norm{\Delta X}_{B^\alpha_{p,q}}}{N_\Delta} + \left(\frac{\norm{\Delta\bbX}_{B^{2\alpha}_{p/2,q/2}}}{NN_\Delta}\right)^{\frac 12}\,.  
\end{align*}
Choosing the constants
\[
N := \norm{X}_{B^\alpha_{p,q}}+\norm{\tilde X}_{B^\alpha_{p,q}},
\quad
N_{\Delta} := \norm{\Delta X}_{B^\alpha_{p,q}}
\]
gives
\[
V^{r/2} \left(\frac{\Delta\bbX}{N N_\Delta}\right)^{\frac 12} \lesssim 1 + \norm{\frac{\Delta \bbX}{N N_\Delta}}^{\frac 12}_{B^{2\alpha}_{p/2,q/2}}\,.
\]
Rearranging this gives
\[
	V^{r/2} (\Delta\bbX) \lesssim N N_\Delta +\norm{\Delta \bbX}_{B^{2\alpha}_{p/2,q/2}}\,,
\]
showing the claim by our choice of $N N_\Delta$.
\end{proof}

\noindent The same logic can be applied to controlled rough paths:
\begin{lemma}\label{lem:BesovControlledRPAreVariationControlledRP}
Let $\bfY= (Y,Y')$ be an $\bfX$-controlled rough path, $r\in[2,3)$ and assume that $\frac 1\alpha<r\le p\le \infty, 0<q\le \infty$. Then there exists a $c(r,\alpha,p,q,T)>0$ such that
	\begin{align*}
		V^r Y' &\le c(r,\alpha,p,q,T) \norm{Y'}_{B^\alpha_{p,q}}\\
		V^{r/2} R^{\bfY,\bfX} &\le c(r,\alpha,p,q,T) (\norm{Y'}_{B^\alpha_{p,q}}^2 + \norm{X}_{B^\alpha_{p,q}}^2 + \norm{R^{\bfY,\bfX}}_{B^{2\alpha}_{p/2,q/2}})\,.
	\end{align*}
\end{lemma}

\noindent We call $\bfY$ a Besov-controlled rough path with parameters $\alpha,p,q$, if
\[
\norm{Y}_{B^\alpha_{p,q}}<\infty\,,\qquad \norm{Y'}_{B^\alpha_{p,q}}<\infty\,,\qquad \norm{R^{\bfY,\bfX}}_{B^{2\alpha}_{p/2,q/2}} <\infty\,.
\]

\begin{proof}
$V^r Y' \lesssim \norm{Y'}_{B^\alpha_{p,q}}$ is again a simple application of Corollary \ref{cor:EmbeddingVrInBesov} to $t\mapsto Y'_t$. For the second inequality, we use the group $G_3$ from Example \ref{ex:GroupForR}. By setting $\tilde \bfY = \tilde\bfX = 0$, we simplify it to $(L(E,F)\oplus E) \oplus F$,
\[
	Z_{s,t} = (Y'_{s,t}, X_{s,t}, R^{\bfY,\bfX}_{s,t})\,.
\]
We use the norm
\[
	\norm{(y,x,r)} = \max\left(\abs{y},\abs{x},\left(\frac 12\abs{r}\right)^{\frac 12}\right)
\]
on this space, which is the norm of $G_3$ with constants $N_a = N_b = N_c = N_d = 1$. Applying Corollary \ref{cor:EmbeddingVrInBesov} to $Z$ gives
\[
	(V^{r/2} R^{\bfY,\bfX})^{\frac 12}\lesssim V^r(Z) \lesssim \norm{Z}_{B^\alpha_{p,q}} \lesssim \norm{Y'}_{B^\alpha_{p,q}} + \norm{X}_{B^\alpha_{p,q}} + \norm{R^{\bfY,\bfX}}_{B^{2\alpha}_{p/2,q/2}}^{\frac 12}\,.
\]
\end{proof}

\noindent We now move on to the main result of this section, namely the application of Theorem \ref{theo:VrBesovEstimate} to rough paths. By applying this theorem to Example \ref{ex:GroupForR}, we get the following.

\begin{proposition}\label{prop:Vr-Besov-Embedding-RYX}
Consider two $r$-rough paths $\bfX,\tilde\bfX$ with controlled rough paths $\bfY,\tilde\bfY$ for $r\in[2,3)$. Assume $0<\frac 1\alpha<r\le p\le\infty$ and $0<q\le\infty$. Then:
\[
	\norm{V^{r/2} \Delta R^{\bfY,\bfX}}_{B^{2\alpha}_{p/2,q/2}} \lesssim \norm{\Delta R^{\bfY,\bfX}}_{B^{2\alpha}_{p/2,q/2}} + \norm{\Delta Y'}_{B^\alpha_{p,q}} \norm{X}_{B^\alpha_{p,q}} + \norm{\tilde Y'}_{B^\alpha_{p,q}} \norm{\Delta X}_{B^\alpha_{p,q}}\,.
\]
\end{proposition}

\begin{proof}
By Exampel \ref{ex:GroupForR}, $Z_{s,t} = ((\tilde Y'_{s,t}, \Delta Y'_{s,t},X_{s,t},\Delta X_{s,t}),\Delta R^{\bfY,\bfX}_{s,t})\in G_3$ is a path in a metric space. Thus, Theorem \ref{theo:VrBesovEstimate} gives us
\begin{equation}\label{ineq:VrBesovZ}
	\norm{V^r Z}_{B^\alpha_{p,q}} \lesssim \norm{Z}_{B^\alpha_{p,q}}\,.
\end{equation}
We chose the following constants for the norm of $G_3$:
\begin{align*}
N_a &= \norm{\tilde Y'}_{B^\alpha_{p,q}} \\
N_b &= \norm{\Delta Y'}_{B^\alpha_{p,q}} \\
N_a &= \norm{X}_{B^\alpha_{p,q}} \\
N_a &= \norm{\Delta X}_{B^\alpha_{p,q}}\,.
\end{align*}
Recall that $N_r = N_a N_d + N_b N_c$. By the definition of $\norm{\cdot}$ in $G_3$ and \eqref{ineq:VrBesovZ}, we have
\begin{align*}
\norm{V^r (\abs{\frac{R^{\bfY,\bfX}}{N_r}}^{\frac 12})}_{B^\alpha_{p,q}} &\lesssim \frac{\norm{\tilde Y'}_{B^\alpha_{p,q}}}{N_a} + \frac{\norm{\Delta Y'}_{B^\alpha_{p,q}}}{N_b} + \frac{\norm{X}_{B^\alpha_{p,q}}}{N_c} + \frac{\norm{\Delta X}_{B^\alpha_{p,q}}}{N_d} + \norm{\abs{\frac{R^{\bfY,\bfX}}{N_r}}^\frac 12}_{B^\alpha_{p,q}}\\
&\lesssim 1 + \norm{\abs{\frac{R^{\bfY,\bfX}}{N_r}}^\frac 12}_{B^\alpha_{p,q}}\,.
\end{align*}
Thus, we conclude
\begin{equation*}
\norm{V^{r/2} R^{\bfY,\bfX}}_{B^{2\alpha}_{p/2,q/2}} \lesssim N_r + \norm{R^{\bfY,\bfX}}_{B^{2\alpha}_{p/2,q/2}}\,.
\end{equation*}
The claim then follows from $N_r = \norm{\Delta Y'}_{B^\alpha_{p,q}} \norm{X}_{B^\alpha_{p,q}} + \norm{\tilde Y'}_{B^\alpha_{p,q}} \norm{\Delta X}_{B^\alpha_{p,q}}$.
\end{proof}

\noindent From this result, one immediately gets estimates for the nested norms of $\bbX,R^{\bfY,\bfX}$ as well as $\Delta\bbX$.

\begin{corollary}\label{cor:Vr_Besov_Embedding_Rough_Paths}
Let $\bfX,\tilde \bfX$ be $r$-rough paths and $\bfY,\tilde \bfY$ be controlled $r$-rough paths. Let $r\in[2,3)$ and assume $0<\frac 1\alpha \le p\le\infty$, $0<q\le \infty$.
Then:
\begin{align}
\norm{V^{r/2} \bbX}_{B^{2\alpha}_{p/2,q/2}} &\lesssim \norm{\bbX}_{B^{2\alpha}_{p/2,q/2}} + \norm{X}_{B^\alpha_{p,q}}^2 \label{ineq:VrBesovEmbeddingX}\\
\norm{V^{r/2} R^{\bfY,\bfX}}_{B^{2\alpha}_{p/2,q/2}} &\lesssim \norm{R^{\bfY,\bfX}}_{B^{2\alpha}_{p/2,q/2}} + \norm{Y'}_{B^\alpha_{p,q}}\norm{X}_{B^\alpha_{p,q}} \label{ineq:VrBesovEmbeddingR}\\
\norm{V^{r/2}\Delta\bbX}_{B^{2\alpha}_{p/2,q/2}} &\lesssim \norm{\Delta\bbX}_{B^{2\alpha}_{p/2,q/2}} + \norm{\Delta X}_{B^\alpha_{p,q}}(\norm{X}_{B^{\alpha}_{p,q}} + \norm{\tilde X}_{B^\alpha_{p,q}}) \label{ineq:VrBesovEmbeddingDeltaX}
\end{align}
\end{corollary}

\begin{proof}
This is a direct consequence of Proposition \ref{prop:Vr-Besov-Embedding-RYX} using that $\bbX$ is an $\bfX$-controlled rough path with $\bbX_{s,t} = R^{\bbX,\bfX}_{s,t}$ (see Example~\ref{ex:RPAreControlledRP}), and setting $\tilde X,\tilde Y = 0$ whenever necessary.
\end{proof}

\noindent We conclude this section by showing a technical inequality necessary for Theorem \ref{thm:LipschitzBesovEstimateRDE}. In the proof of said proposition, we are in the situation that we need to estimate $V^{r/2} \Delta\bbX$ in the $B^{\alpha}_{p,q}$-norm instead of the $B^{2\alpha}_{p/2,q/2}$-norm, as the scaling of $\Delta\bbX$ would imply. Thanks to Example \ref{ex:GroupForDeltaX}, this can easily be done with classical Besov embeddings.

\begin{proposition}\label{prop:VrBesovEmbeddingDeltaX2}
Let $\bfX,\tilde\bfX$ be $r$-rough paths for $r\in[2,3)$ and assume that $0<\frac 1\alpha <r\le p\le \infty$ and $1\le p\le \infty, \frac 12\le q\le\infty$. Then:
\begin{equation}\label{ineq:VrBesovEmbeddingDeltaX2}
\norm{V^{r/2} \Delta \bbX}_{B^\alpha_{p,q}} \lesssim (\norm{X}_{B^\alpha_{p,q}} + \norm{\tilde X}_{B^\alpha_{p,q}})\norm{\Delta X}_{B^\alpha_{p,q}} + \norm{\Delta \bbX}_{B^{2\alpha}_{p/2,q/2}}
\end{equation}
\end{proposition}

\begin{proof}
The proof works similarly to the one of Lemma \ref{lem:BesovRPAreVariationRP}. We consider the process $F_{s,t} = ( X_{s,t},\tilde X_{s,t},\Delta X_{s,t}, \Delta\bbX_{s,t})$ in $G_2$ from Example \ref{ex:GroupForDeltaX}. Note that $V^r F$ is monotone in the sense of \eqref{cond:Monotonicity} and subadditive in the sense of \eqref{cond:Subadditivity}. Thus, we can use the Besov embedding given by Proposition \ref{prop:Besov-Embedding-P-Alpha} to see
\[
\norm{V^r F}_{B^{\alpha/2}_{2p,2q}} \lesssim \norm{V^r F}_{B^{\frac 12(\alpha+1/p)}_{p,2q}}\,.
\]
Using $\frac 12(\alpha+1/p)\le \alpha$ as well as Proposition \ref{lem:Besov-Embedding-Q} and \ref{lem:Besov-Embedding-Q-Alpha}, we see that
\[
\norm{V^r F}_{B^{\alpha/2}_{2p,2q}} \lesssim \norm{V^r F}_{B^{\alpha}_{p,q}}\,.
\]
We then have for any $N,N_\Delta > 0$
\begin{align*}
\norm{V^{r/2} \Delta\bbX/(NN_\Delta)}_{B^\alpha_{p,q}}^{\frac 12} &\le \norm{V^r F}_{B^{\alpha/2}_{2p,2q}} \\
&\lesssim \norm{V^r F}_{B^\alpha_{p,q}}\\
&\lesssim \frac{\norm{V^r X}_{B^\alpha_{p,q}}}N +\frac{\norm{V^r \tilde X}_{B^\alpha_{p,q}}}N + \frac{\norm{V^r \Delta X}_{B^\alpha_{p,q}}}{N_\Delta} + \left(\frac{\norm{V^{r/2} \Delta\bbX}_{B^{2\alpha}_{p/2,q/2}}}{NN_\Delta}\right)^{\frac 12}\,.  
\end{align*}
We chose the constants
\[
N := \norm{X}_{B^\alpha_{p,q}}+\norm{\tilde X}_{B^\alpha_{p,q}},
\quad
N_{\Delta} := \norm{\Delta X}_{B^\alpha_{p,q}}\,.
\]
By this choice and Corollary \ref{cor:Vr_Besov_Embedding_Rough_Paths}, we have
\[
	\frac{\norm{V^r X}_{B^\alpha_{p,q}}}N +\frac{\norm{V^r \tilde X}_{B^\alpha_{p,q}}}N + \frac{\norm{V^r \Delta X}_{B^\alpha_{p,q}}}{N_\Delta}\lesssim 1\,,
\]
implying
\begin{equation}\label{ineq:BesovEmbeddingForDeltaX}
\norm{V^{r/2} \Delta\bbX}_{B^\alpha_{p,q}} \lesssim N N_\Delta +\norm{V^{r/2}\Delta\bbX}_{B^{2\alpha}_{p/2,q/2}}\,.
\end{equation}
Applying Corollary \ref{cor:Vr_Besov_Embedding_Rough_Paths} one more time to estimate $\norm{V^{r/2}\Delta\bbX}_{B^{2\alpha}_{p/2,q/2}}$ shows the claim.
\end{proof}

\begin{remark}
Note that the right-hand side of \eqref{ineq:VrBesovEmbeddingDeltaX} and \eqref{ineq:VrBesovEmbeddingDeltaX2} are the same. By \eqref{ineq:BesovEmbeddingForDeltaX}, one can replace the $B^{2\alpha}_{p/2,q/2}$-norm of $V^{r/2} \Delta\bbX$ with $B^\alpha_{p,q}$ at the cost of adding $N N_\Delta$, which gets absorbed in the constant in $\lesssim$ in \eqref{ineq:VrBesovEmbeddingDeltaX}.
\end{remark}

\section{Sewing Lemma}\label{sec:Sewing}

The sewing lemma was originally introduced by Feyel and de La Pradelle \cite{MR2261056, Feyel2008} as well as Gubinelli \cite{MR2091358}.
It gives a criterion for convergence of Riemann sums and generalizes earlier work by Young \cite{MR1555421}.
It is widely used in rough paths theory \cite{Lyo98} to construct and estimate rough integrals.

In this section, we extend this lemma to the case of $p$-metrics.
This allows one to recover the Besov sewing lemma formulated in \cite{FS22} as a simple corollary of the variation sewing lemma.
We show this fact in Theorem \ref{theo:BesovSewing}.
We also recover a version of the sewing lemma from \cite{MR3770049}, which we present in Theorem \ref{thm:sewing}.

\subsection{Sewing in the $r$-variation setting}

Let $A:\Delta_T\to E$ be a two-parameter process into some Banach space $E$. For any $r>0$, we define the $r$-variation of its $\delta$-operator defined in \eqref{eq:DeltaOperator} by
\[
V^r \delta A := \sup_\pi \left(\sum_{j=1}^{k-2} \abs{\delta A_{\pi_j,\pi_{j+1},\pi_{j+2}}}^r\right)^{\frac 1r}\,,
\]
where the supremum runs over all partitions $\pi = \{0=\pi_1<\dots<\pi_k = T\}$ of $[0,T]$. Note that we explicitly allow $r<1$. Given a partition $\pi$, we consider the Riemann sum
\[
\calI^\pi A_{0,T} := \sum_{j=1}^{k-1} A_{\pi_j,\pi_{j+1}}\,.
\]
We present a sewing lemma with a short proof. As noted already by Young in \cite{MR1555421}, any $2$-parameter process whose $\delta$-operator has finite $r$-variation for $r<1$ satisfies a bound on its Riemann sums. Before we show this, let us recall the definition of a $p$-metric.

\begin{definition}\label{def:pNorm}
Let $(E,+)$ be a commutative group and $p\in(0,1]$. Then we call a map $\abs{\cdot}: E\to\bbR_+$ a $p$-norm, if for all $x,y\in E$
\begin{align*}
	\abs{x} &= 0 \qquad\Leftrightarrow \quad x = 0\\
	\abs{x+y}^p &\le \abs{x}^p + \abs{y}^p\,.
\end{align*} 
\end{definition}

\noindent By the Aoki--Rolewicz theorem, see e.g.\ \cite[Theorem 1.2]{MR808777}, any quasinorm is equivalent to a $p$-norm for some $p \in (0,1]$. We set for any $\theta > 0$
\begin{equation}\label{eq:Zeta}
	\zeta(\theta) := \sum_{k=1}^{\infty} k^{-\theta}\,.
\end{equation}

\begin{lemma}[Sewing, cf.\ {\cite[\textsection 5]{MR1555421}}]
\label{lem:sewing-bound}
Let $0<r<p\leq 1$.
Let $E$ be a commutative group equipped with a $p$-norm $\abs{\cdot}$.
Let $A : \Delta_T \to E$ with
\[
V^r\delta A  < \infty.
\]
Then, for any partition $\pi$ of $[0,T]$, we have
\[
\abs{ \calI^{\pi}A_{0,T}-A_{0,T}}
\leq
\zeta(p/r)^{1/p} V^{r}\delta A\,.
\]
\end{lemma}

\begin{proof}
By induction on the partition size $k\geq 2$, we will show that, for every partition $0=\pi_{1}<\dotsb<\pi_{k}=T$, we have
\[
\abs{A_{0,T} - \calI^{\pi}A_{0,T}}^{p}
\leq
\bigl(\sum_{j=1}^{k-1} j^{-p/r}\bigr) (V^{r}\delta A)^{p},
\]
For $k=2$, we have $\calI^{\pi} A_{0,T} = A_{0,T}$, which serves as the induction base.

Suppose that the claim is known for all partitions of size $k$ and let $0=\pi_{0}<\dotsb<\pi_{k}=T$ be a partition of size $k+1$.
Since
\[
\sum_{j=0}^{k-1} \abs{\delta A_{\pi_{j},\pi_{j+1},\pi_{j+2}}}^{r}
\leq
(V^{r}\delta A)^{r},
\]
there exists $j$ such that
\[
\abs{\delta A_{\pi_{j},\pi_{j+1},\pi_{j+2}}}
\leq
k^{-1/r} V^{r}\delta A.
\]
Let $\pi' := \pi \setminus \Set{\pi_{j+1}}$, so that
\[
\calI^{\pi} A_{0,T} - \calI^{\pi'} A_{0,T}
=
\delta A_{\pi_{j},\pi_{j+1},\pi_{j+2}}.
\]
Then, by the inductive hypothesis,
\begin{align*}
\abs{A_{0,T} - \calI^{\pi} A_{0,T}}^{p}
&\leq
\abs{A_{0,T} - \calI^{\pi'} A_{0,T}}^{p}
+
\abs{\calI^{\pi'} A_{0,T} - \calI^{\pi} A_{0,T}}^{p}
\\ &\leq
\bigl(\sum_{j=1}^{k-1} j^{-p/r}\bigr) (V^{r}\delta A)^{p}
+
k^{-p/r} (V^{r}\delta A)^{p}
\\ &=
\bigl(\sum_{j=1}^{k} j^{-p/r}\bigr) (V^{r}\delta A)^{p}.
\end{align*}
This finishes the inductive step.
\end{proof}

\noindent Recall that partitions $\pi$ of $[0,T]$ are a directed set with the preorder given by refinements. Thus, $\pi\mapsto \calI^\pi A_{0,T}$ is a net. We denote the net limit over partitions (also called RRS limit in \cite{MR3770049}) with $\lim_{\pi}$. One can show the following sewing result for net limits:

\begin{lemma}
\label{lem:sewing-limit}
Let $0 < r < p \leq 1$ and let $E$ be a complete $p$-metric group.
Let $A : \Delta_T \to E$ be such that $V^{r}\delta A < \infty$ and
\begin{equation}\label{eq:continuityDeltaA}
\lim_{\pi}\sup_{j, \pi_{j}\leq s<t<u \leq \pi_{j+1}} \abs{\delta A_{s,t,u}} = 0.
\end{equation}
Then, the following net limit exists,
\[
\calI A_{0,T}:=\lim_{\pi} \calI^\pi A_{0,T},
\]
and one has the estimate
\[
\abs{ \calI A_{0,T}- A_{0,T}}
\leq
\zeta(p/r)^{1/p} V^{r}\delta A.
\]
\end{lemma}

\begin{remark}
For the Sections \ref{sec:LipschitzYoung} and \ref{sec:LipschitzRP}, it suffices to show Lemma \ref{lem:sewing-limit} for $p=1$, i.e. for metric groups. The extension to $p$-metric groups allows us to recover \cite[Theorem 3.1]{FS22}, which we show in Theorem \ref{theo:BesovSewing}.
\end{remark}

\begin{proof}
Let $r < r' < p$.
For partitions $\pi' \subseteq \pi$, by Lemma~\ref{lem:sewing-bound} with $r'$ in place of $r$, we estimate
\begin{align*}
\abs{\calI^{\pi} A_{0,T} - \calI^{\pi'} A_{0,T}}^{p}
&\leq
\sum_{j} \abs{\calI^{\pi} A_{\pi'_{j},\pi'_{j+1}} - A_{\pi'_{j},\pi'_{j+1}}}^{p}
\\ &\leq
\zeta(p/r') \sum_{j} (V^{r'}_{\pi'_{j},\pi'_{j+1}} \delta A)^{p}
\\ &\leq
\zeta(p/r') \Bigl(\sup_{j, \pi'_{j} \leq s < t < u \leq \pi'_{j+1}} \abs{\delta A_{s,t,u}}^{p(1-r/r')} \Bigr)
\sum_{j} (V^{r}_{\pi'_{j},\pi'_{j+1}} \delta A)^{pr/r'}.
\end{align*}
The supremum converges to $0$ by the hypothesis \eqref{eq:continuityDeltaA}, and
\[
\sum_{j} (V^{r}_{\pi'_{j},\pi'_{j+1}} \delta A)^{pr/r'}
\leq
\bigl( \sum_{j} (V^{r}_{\pi'_{j},\pi'_{j+1}} \delta A)^{r} \bigr)^{p/r'}
\leq
\bigl( (V^{r} \delta A)^{r} \bigr)^{p/r'}
\]
is bounded.
This shows that $\pi \to \calI^{\pi} A_{0,T}$ is a Cauchy net, so it convergence by the completeness hypothesis.
Finally, the estimate follows from a separate application of Lemma~\ref{lem:sewing-bound}.
\end{proof}

\noindent We finally introduce the sewing lemma we use throughout Sections \ref{sec:LipschitzYoung} and \ref{sec:LipschitzRP}.
Recall that a \emph{control} is a function $\omega : \Delta_{T} \to [0,\infty)$
 that is \emph{super-additive} in the sense $\omega(s,t)+\omega(t,u) \le \omega(s,u)$ for all $s\leq t\leq u$ (setting $s=t=u$, this implies in particular that $\omega(s,s)=0$).

\begin{theorem}[Sewing, cf.\ {\cite[Theorem 2.5]{MR3770049}}]
\label{thm:sewing}
Let $N\geq 0$, $\theta>1$, and $\alpha_{1,n},\alpha_{2,n} \geq 0$ with $\alpha_{1,n}+\alpha_{2,n}\geq \theta$ for all $n \in \Set{1,\dotsc,N}$.
Let $\omega_{1,n}, \omega_{2,n}$ be controls and $A : \Delta_T \to E$ with a Banach space $E$.
Assume
\begin{equation}
\label{eq:sewing:delta-bd}
\abs{\delta A_{s,u,t}}
\leq
\sum_{n=1}^N \omega_{1,n}^{\alpha_{1,n}}(s,u) \omega_{2,n}^{\alpha_{2,n}}(u,t).
\end{equation}
Then the following net limit exists,
\[
\calI A_{0,T}:=\lim_{\pi} \calI^\pi A_{0,T},
\]
and one has the estimate
\begin{equation}
\label{eq:2}
\abs{ \calI A_{0,T}- A_{0,T}}
\leq
\zeta(\theta) \Bigl( \sum_{n=1}^N \omega_{1,n}^{\alpha_{1,n}/\theta}(0,T-) \omega_{2,n}^{\alpha_{2,n}/\theta}(0+,T) \Bigr)^{\theta},
\end{equation}
where $\zeta(\theta) = \sum_{k=1}^{\infty} k^{-\theta}$.
\end{theorem}

\noindent Our proof saves a factor $2^{\theta}$ on the right-hand side of \eqref{eq:2} compared to \cite{MR3770049}.

\begin{proof}
For any partition $\pi$, by hypothesis \eqref{eq:sewing:delta-bd}, H\"older's inequality, monotonicity of $\ell^{q}$ norms, and superadditivity of controls, we have
\begin{align*}
\sum_{k=0}^{J-1} \abs{\delta A_{\pi_{k},\pi_{k+1},\pi_{k+2}}}^{1/\theta}
&\leq
\sum_{k=0}^{J-1} \Bigl( \sum_{n=1}^N
\omega_{1,n}^{\alpha_{1,n}}(\pi_{k},\pi_{k+1}) \omega_{2,n}^{\alpha_{2,n}}(\pi_{k+1},\pi_{k+2}) \Bigr)^{1/\theta}
\\ &\leq
\sum_{k=0}^{J-1} \sum_{n=1}^N
\bigl( \omega_{1,n}^{\alpha_{1,n}}(\pi_{k},\pi_{k+1}) \omega_{2,n}^{\alpha_{2,n}}(\pi_{k+1},\pi_{k+2}) \bigr)^{1/\theta}
\\ &\leq
\sum_{n=1}^N \Bigl( \sum_{k=0}^{J-1} \omega_{1,n}^{\alpha_{1,n}+\alpha_{2,n}}(\pi_{k},\pi_{k+1}) \Bigr)^{\frac{\alpha_{1,n}}{\alpha_{1,n}+\alpha_{2,n}}/\theta}
\\ &\qquad\qquad\times
\Bigl( \sum_{k=0}^{J-1} \omega_{2,n}^{\alpha_{1,n}+\alpha_{2,n}}(\pi_{k+1},\pi_{k+2}) \Bigr)^{\frac{\alpha_{2,n}}{\alpha_{1,n}+\alpha_{2,n}}/\theta}
\\ &\leq
\sum_{n=1}^N \Bigl( \sum_{k=0}^{J-1} \omega_{1,n}(\pi_{k},\pi_{k+1}) \Bigr)^{\alpha_{1,n}/\theta}
\Bigl( \sum_{k=0}^{J-1} \omega_{2,n}(\pi_{k+1},\pi_{k+2}) \Bigr)^{\alpha_{2,n}/\theta}
\\ &\leq
\sum_{n=1}^N \omega_{1,n}(\pi_{0},\pi_{J})^{\alpha_{1,n}/\theta}
\omega_{2,n}(\pi_{1},\pi_{J+1})^{\alpha_{2,n}/\theta}.
\end{align*}
Moreover, \eqref{eq:sewing:delta-bd} implies \eqref{eq:continuityDeltaA} by \cite[Lemma 2.1]{MR3770049}. 
The claim follows from Lemma~\ref{lem:sewing-limit} with $r=1/\theta$.
\end{proof}

\subsection{Connection to Besov sewing}

In this section, we want to show that $r$-variation norms can be used to quickly recover the Besov sewing lemma from \cite{FS22}. To do so, we recall the Besov norm of a three-parameter process:
\begin{definition}
Let $\delta A: \{(s,u,t)\in[0,T]~\vert~ s\le u\le t\}\to E$ for some normed space $E$, $0<p,q\le\infty$, $\alpha>0$. We then set
\begin{align*}
	\bar\Omega_p^{\delta A}(t) :=& \sup_{0\le u\le v\le t} \left(\int_0^{T-v} \abs{\delta A_{s,s+u,s+v}}^p\dif s\right)^{\frac 1p}\\
	\norm{\delta A}_{\bar B^\alpha_{p,q}} :=& \left(\int_0^T \left(\frac{\bar\Omega_p^{\delta A}(t)}{t^\alpha}\right)^q\frac{\dif t}t\right)^{\frac 1q}\,,
\end{align*}
with the usual extension to $p=\infty$ and $q=\infty$.
\end{definition}

\noindent We can then show the main statement of \cite{FS22}, Theorem 3.1:

\begin{theorem}\label{theo:BesovSewing}
Let $A:\Delta_T\to \bbR^m$ and let $0<p,q\le\infty, \gamma > \max(1,1/p,1/q)$.
Assume that
\[
	\norm{A}_{B^\gamma_{p,q}}, \norm{\delta A}_{\bar B^\gamma_{p,q}} <\infty\,.
\]
Then the sewing $\calI A = \lim_{\pi} \calI^\pi A$ exists and fulfills for some $c(\gamma,p,q,T) > 0$
\[
	\norm{A-\cal I A}_{B^\gamma_{p,q}} \le c(\gamma,p,q,T) \norm{\delta A}_{\bar B^\gamma_{p,q}}\,.
\]
\end{theorem}

\begin{proof}
The key technique of the proof is to define for any $0\le a \le b\le T$ the $2$-parameter object
\[
\Xi_{s,t}(a,b) := A_{a+s(b-a),a+t(b-a)}\,,
\]
such that $\Xi(a,b):\Delta_1\to B^\gamma_{p,q} =: E$. Note that the Besov-norm $\norm{\cdot}_{B^\alpha_{p,q}}$ is a $\min(1,p,q)$-norm, so we can apply Lemma \ref{lem:sewing-limit} to $\Xi$. By construction, it holds for any partition $\pi= \{0<\pi_1<\dots<\pi_k =1\}$ of $[0,1]$, that $\Xi_{0,1}(a,b) = A_{a,b}$ and
\[
	\calI^\pi \Xi(a,b) = \calI^{\pi'(a,b)}A\,,
\]
where $\pi'(a,b) = \{a+\pi_j(b-a)~\vert~\pi_j\in\pi\}$ is the corresponding partition of $(a,b)$. Thus, $\calI A$ exists if and only if $\calI\Xi$ exists and we have
\[
	\calI \Xi_{s,t}(a,b) = \calI A_{a+s(b-a), a+t(b-s)}\,.
\]
It holds that
\begin{align*}
\delta\Xi_{s,t,u}(a,b)
&=
A_{a+s(b-a),a+u(b-a)} - A_{a+s(b-a),a+t(b-a)} - A_{a+t(b-a),a+u(b-a)}
\\ &=
\delta A_{a+s(b-a),a+t(b-a),a+u(b-a)}\,.
\end{align*}
Hence, we have the following for all $s\le t\le u$.
\begin{align*}
\Omega_{p}^{\delta\Xi_{s,t,u}}(\tau)
&=
\sup_{0\leq h \leq \tau}
\left(\int_0^{T-h}\abs{\delta\Xi_{s,t,u}(r,r+h)}^p\dif r\right)^{\frac 1p}
\\ &=
\sup_{0\leq h \leq \tau}
\left(\int_0^{T-h} \abs{\delta A_{r+sh,r+th,r+uh}}^p\dif r\right)^{\frac 1p}
\\ &\leq
\bar{\Omega}_{p}^{\delta A}(\tau(u-s)).
\end{align*}
Thus,
\begin{align*}
\norm{\delta\Xi_{s,t,u}}_{E}
&=
\left(\int_0^T \left(\tau^{-\gamma} \Omega_{p}^{\delta\Xi_{s,t,u}}(\tau)\right)^q \frac{\dif \tau}\tau\right)^{\frac 1q}
\\ &\leq
\left(\int_0^T \left(\tau^{-\gamma} \bar{\Omega}_{p}^{\delta A}(\tau(u-s))\right)^q\frac{\dif\tau}\tau\right)^{\frac 1q}
\\ &\leq
(u-s)^{\gamma} \left(\int_0^T \left(\tau^{-\gamma} \bar{\Omega}_{p}^{\delta A}(\tau)\right)^q \frac{\dif\tau}\tau\right)^{\frac 1q}
\\ &=
(u-s)^{\gamma} \norm{\delta A}_{\bar B^\gamma_{p,q}}\,.
\end{align*}
This especially shows \eqref{eq:continuityDeltaA}. It further follows that
\begin{align*}
V^{r}\delta\Xi
&\leq
\sup_{\pi} \left(\sum_{j=1}^{k-2}\abs{\pi_{j+2}-\pi_{j}}^{\gamma r}\right)^{\frac 1r}
\norm{\delta A}_{\bar B^\gamma_{p,q}}
\\ &\lesssim
\norm{\delta A}_{\bar B^\gamma_{p,q}} <\infty
\end{align*}
provided that $r\gamma\geq 1$. By choosing $r = \frac 1\gamma <\max(p, 1)$, we can thus apply Lemma \ref{lem:sewing-limit} to get the existence of $\calI \Xi$ (and thus $\calI A$) as well as
\begin{equation*}
\norm{A-\calI A}_{B^\gamma_{p,q}} = \norm{\calI \Xi_{0,1}-\Xi_{0,1}}_{B^\gamma_{p,q}} \lesssim \norm{\delta A}_{\bar B^\gamma_{p,q}}\,.
\end{equation*}
\end{proof}

\section{Lipschitz estimates for Young differential equations}\label{sec:LipschitzYoung}

Consider the Young equation
\begin{equation}\label{eq:YoungODE}
	Y_t = y_0 + \int_0^t \phi(Y_s) \dif X_s
\end{equation}
where $Y_0 = y_0\in \bbR^d$ is the initial condition, $\phi$ is a given, smooth enough function and we assume that the driving noise $X\in V^r$ for some $r\in[1,2)$. $\int_0^t \phi(Y_s) dX_s$ is given by the Young integral originally introduced in \cite{MR1555421}. The goal of this section is to establish the Lispchitz dependence of the solution $Y$ on the data $X,y_0$ and $\phi$. To this end, we consider a second solution
\begin{equation*}
\tilde Y_t = \tilde y_0 +\int_0^t \tilde\phi(\tilde Y_s)\dif\tilde X_s
\end{equation*} 
for some data $\tilde X,\tilde y_0,\tilde \phi$ and show an estimate of the form
\begin{equation}\label{ineq:LipscitzEstimateYoungForm}
	V^r \Delta Y_{s,t} \le c_1(s,t) (\abs{\Delta y_0} + V^r \Delta X_{0,T}+\norm{\Delta\phi}_{C^\alpha}+ \norm{\Delta\phi}_{\sup}) +c_2(s,t)V^r\Delta X_{s,t}\,.
\end{equation}
Here, $\Delta X = X-\tilde X$ with similar notation extending to $\Delta Y, \Delta y_0$ and $\Delta \phi$. Estimates of the form \eqref{ineq:LipscitzEstimateYoungForm} are by no means new, see for example \cite{Galeati2023, LEJAY20101777}. We improve on these by giving precise forms of $c_1(s,t), c_2(s,t)$ for $0\le s\le t\le T$. This is necessary to estimate the Besov-norm of $(s,t)\mapsto V^r\Delta Y_{s,t}$, as the Besov-norm of the right-hand side of \eqref{ineq:LipscitzEstimateYoungForm} might explode otherwise.

We derive our estimate by constructing $Y$ in a solution space that dynamically depends on $X$. That is, we consider the solution space
\begin{equation}\label{eq:SolutionSpaceYoung}
	\calY(X) = \{Y~\vert~ V^r Y_I \le 2\Phi_0 V^r X_I\text{ for all intervals }I\subset[0,T]\}\,.
\end{equation}
for a constant $\Phi_0 > 0$ such that $\norm{\phi}_{\sup}\le\Phi_0$. This allows us to easily get bounds on $V^r Y_I$ for any $I\subset[0,T]$. 

In this section, we show an a priori estimate in Proposition \ref{prop:Young-apriori}, local existence and uniqueness (Proposition \ref{prop:Picard-contraction}) of the solution of \eqref{eq:YoungODE}, and finally the desired Lipschitz estimate in Theorem \ref{thm:Lipschitz-stability}.

Let us start by fixing some notation. We fix an $r\in[1,2)$, $\alpha\in(0,1]$ with $r<1+\alpha$ and set $\theta = (1+\alpha)/r$. Recall that $C^\alpha$ denotes the space of Hölder functions with seminorm
\begin{equation}\label{equ:Calpha}
	\norm{\phi}_{C^\alpha} = \sup_{s<t}\frac{\abs{\phi(t)-\phi(s)}}{\abs{t-s}^\alpha}\,,
\end{equation}
where we make the usual extension that for all $k\in\bbN,\alpha\in(0,1]$, $C^{k,\alpha}_b$ denotes the space of $k$-times bounded, differentiable functions with bounded derivatives and $D^k\phi\in C^\alpha$. This space is equipped with the usual norm
\[
	\norm{\phi}_{C^{k,\alpha}_b} := \sum_{n\le k} \norm{D^n\phi}_{\sup}+\norm{D^k\phi}_{C^\alpha}\,.
\]
Throughout this section, we assume $X:[0,T]\to \bbR^n, X\in V^r$ and $\phi\in C^{1,\alpha}_b$. We further assume $Y:[0,T]\to \bbR^d$, that is $\phi:\bbR^d\to L(\bbR^n,\bbR^d)$. We allow jumps, as in \cite{MR3770049}.
All estimates are of Davie form \cite{MR2387018} (see also \cite[Chapter 10]{FV10}).

\subsection{A $4$-point Taylor formula with fractional reminder}

Before we solve \eqref{eq:YoungODE}, we give the following 4-point Taylor formula as a preliminary result. Similar formulas can be found in \cite[Lemma 15]{Bechtold2023} and \cite[Lemma 10.22]{FV10}.

\begin{lemma}
\label{lem:4point-Taylor-ord-1}
For every $\alpha \in (0,1]$, every $\phi \in C^{1,\alpha}_{b}$, every $\tphi\in C^{0,\alpha}_{b}$, and any $Y_{s},Y_{t},\tY_{s},\tY_{t}\in\bbR^d$, we have
\begin{multline}
\label{eq:4point-Taylor-ord-1}
\abs{ \Delta \Bigl( \phi(Y_{t}) - \phi(Y_{s}) \Bigr) }
\leq
\norm{D\phi}_{\sup} \abs{\Delta Y_{s,t}}
+
\frac{\norm{D\phi}_{C^{\alpha}}}{1+\alpha}
(\abs{Y_{s,t}}^{\alpha} + \abs{\tilde Y_{s,t}}^{\alpha}) \abs{\Delta Y_{s}}
+ \norm{\Delta \phi}_{C^{\alpha}} \abs{\tY_{s,t}}^{\alpha}.
\end{multline}
\end{lemma}

\begin{proof}
We split
\begin{align}
\abs{ \Delta \Bigl( \phi(Y_{t}) - \phi(Y_{s}) \Bigr) }
&\leq
\abs{ ( \phi(Y_{t}) - \phi(Y_{s}) ) - ( \phi(\tY_{t}) - \phi(\tY_{s}) ) }
\label{eq:4}
\\ &+
\abs{ ( \phi(\tY_{t}) - \phi(\tY_{s}) ) - ( \tilde\phi(\tY_{t}) - \tilde\phi(\tY_{s}) ) }.
\label{eq:5}
\end{align}
Note that \eqref{eq:5} is easily bound by $\norm{\Delta \phi}_{C^\alpha}\abs{\tilde Y_{s,t}}^\alpha$. To estimate \eqref{eq:4}, we use the fundamental theorem of calculus to get
\begin{align}
\abs{ ( \phi(Y_{t}) - \phi(Y_{s}) ) - ( \phi(\tY_{t}) - \phi(\tY_{s}) ) } &= \abs{\Delta Y_t \int_0^1 D\phi(\tilde Y_t + u\Delta Y_t)\dif u + \Delta Y_s \int_0^1 D\phi(\tilde Y_s + u\Delta Y_s)\dif u}\nonumber\\
&\le \abs{\Delta Y_{s,t}} \int_0^1 \abs{D\phi(\tilde Y_t+u\Delta Y_t)} du \label{eq:Taylor6}\\
&\qquad+ \abs{\Delta Y_s}\int_0^1 \abs{D\phi(\tilde Y_t +u\Delta Y_t) - D\phi(\tilde Y_s + u\Delta Y_s)} du\label{eq:Taylor7}
\end{align}
\eqref{eq:Taylor6} is bounded by $\norm{D\phi}_{\sup} \abs{\Delta Y_{s,t}}$.
The difference inside the integral of \eqref{eq:Taylor7} is bounded by
\begin{align*}
\MoveEqLeft
\norm{D\phi}_{C^{\alpha}} \abs{(\tY_{t}+u \Delta Y_{t}) - (\tY_{s}+u \Delta Y_{s})}^{\alpha}
\\ &\leq
\norm{D\phi}_{C^{\alpha}} (\abs{u Y_{s,t}} + (1-u) \abs{\tilde Y_{s,t}})^{\alpha}.
\\ &\leq
\norm{D\phi}_{C^{\alpha}} ( \abs{u}^{\alpha}\abs{Y_{s,t}}^{\alpha} + (1-u)^{\alpha} \abs{\tilde Y_{s,t}}^{\alpha}).
\end{align*}
Integrating this estimate in $u$, we obtain the claimed estimate for \eqref{eq:4}.
\end{proof}

\subsection{A priori estimate}
We show an a priori estimate for the solution of \eqref{eq:YoungODE}. This is based on the following estimate for Young integrals.

\begin{lemma}\label{lem:1}
Let $r_1,r_2 \ge 1$ with $\frac 1{r_1} + \frac 1{r_2} =: \theta > 1$. Let $Y\in V^{r_1}, X\in V^{r_2}$. 
Then, the Young integral $Z_{t} := \int_{0}^{t} Y \dif X, t\in[0,T]$ exists and satisfies
\begin{equation}\label{ineq:YoungEstimateRVar1}
\abs{Z_{s,t}- Y_{s} X_{s,t}}
\leq
\zeta(\theta)  V^{r_1} Y_{[s,t)} V^{r_2} X_{(s,t]}\,,
\end{equation}
where $\zeta(\theta)= \sum_{k=1}^\infty k^{-\theta}$.
\end{lemma}

\begin{proof}
Recall that the Young integral $Z_t$ is given as the sewing of the germ
\[
	(s,t)\mapsto A_{s,t} = Y_s X_{s,t}\,.
\]
With
\[
	\abs{\delta A_{s,u,t}} = \abs{Y_{s,u}}\abs{X_{s,u}} \le V^{r_1} Y_{[s,u]} V^{r_2} X_{[u,t]}\,. 
\]
Since $\theta > 1$, the sewing Lemma (Theorem \ref{thm:sewing}) immediately gives us that $Z$ exists and \eqref{ineq:YoungEstimateRVar1}.
\end{proof}

\noindent We are especially interested in the case in which $Y$ is composed with a Hölder continuous function $\phi$:

\begin{corollary}
\label{cor:1}
Let $r \in [1,2)$ and $\alpha \in (0,1]$ with $r < 1+\alpha$. Let $\theta = (1+\alpha)/r$.
Let $X,Y \in V^{r}$ and $\phi \in C^{\alpha}$.
Then, the Young integral $Z_{T} := \int_{0}^{T} \phi(Y) \dif X$ exists and satisfies
\begin{equation}\label{ineq:YoungEstimateRVar}
\abs{Z_{s,t}- \phi(Y_{s}) X_{s,t}}
\leq
\zeta(\theta) \norm{\phi}_{C^{\alpha}} V^r Y_{[s,t)}^{\alpha} V^r X_{(s,t]}\,,
\end{equation}
where $\zeta(\theta)= \sum_{k=1}^\infty k^{-\theta}$.
\end{corollary}
\begin{proof}
For any path $Y$ and interval $I\subset[0,T]$, it is easy to see that
\begin{equation*}
V^{r/\alpha}(\phi(Y))_I
\leq
\norm{\phi}_{C^{\alpha}} (V^r Y_I)^\alpha.
\end{equation*}
The claim is then a direct consequence of Lemma \ref{lem:1}.
\end{proof}

\begin{proposition}[A priori estimate]
\label{prop:Young-apriori}
Let $r,\alpha$ be as above.
Let $\phi\in C^{0,\alpha}_{b}$ with two constants $\Phi_0,\Phi_\alpha>0$, such that
\begin{equation}
\label{eq:C^0,alpha-bounds}
\norm{\phi}_{\sup} \leq \Phi_{0},
\quad
\norm{\phi}_{C^{\alpha}} \leq \Phi_{\alpha}.
\end{equation}
For any $0\le s\le t\le T$, let $X,Y \in V^{r}_{[s,t]}$ with
\[
Y_{u} = Y_{s} + \int_{s}^{u} \phi(Y_z) \dif X_z,
\quad
\text{for all }u\in [s,t].
\]
Then, assuming that there is an $\epsilon > 0$ with
\begin{equation}
\label{eq:VpX-small}
V^r X_{(s,t]} \leq \epsilon, \quad
2^{\alpha} \zeta(\theta) \Phi_\alpha \epsilon^{\alpha} \leq \Phi_0^{1-\alpha},
\end{equation}
we have
\begin{equation}
\label{eq:Young-apriori:difference}
V^r Y_{[s,t]}
\leq
2 \Phi_{0} V^r X_{[s,t]}.
\end{equation}
\end{proposition}

\begin{proof}
We write $\phi(Y)\delta X$ for the $2$-parameter proccess $(s,t)\mapsto \phi(Y_s) X_{s,t}$. By Corollary~\ref{cor:1}, we have
\begin{equation}
\label{eq:7}
V^{r/(1+\alpha)}(Y - \phi(Y) \delta X)_{[s,t]}
\leq
\zeta(\theta) \Phi_{\alpha} V^r Y_{[s,t)}^{\alpha} V^r X_{(s,t]}.
\end{equation}
By \eqref{eq:7}, the hypothesis \eqref{eq:VpX-small}, and the AMGM inequality, we find that
\begin{align*}
V^{r/(1+\alpha)}(Y - \phi(Y) \delta X)_{[s,t]}
&\leq
2^{-\alpha}V^r Y_{[s,t)}^{\alpha} (\Phi_{0} V^r X_{(s,t]})^{1-\alpha}
\\ &\leq
\alpha 2^{-1} V^r Y_{[s,t)} + (1-\alpha) \Phi_{0} V^r X_{(s,t]}.
\end{align*}
Hence,
\begin{align*}
V^r Y_{[s,t]}
&\leq
V^{r/(1+\alpha)} (Y - \phi(Y) \delta X)_{[s,t]}
+ V^r (\phi(Y) \delta X)_{[s,t]}
\\ &\leq
\alpha 2^{-1} V^r Y_{[s,t)} + (1-\alpha) \Phi_{0} V^r X_{(s,t]}
+ \Phi_{0} V^r X_{[s,t]},
\end{align*}
so that
\[
(1-\alpha/2) V^r Y_{[s,t]}
\leq
(2-\alpha) \Phi_{0} V^r X_{[s,t]}.
\]
This shows \eqref{eq:Young-apriori:difference}.
\end{proof}

\begin{lemma}[Solution space]
\label{lem:solution-space}
Let $\phi\in C^{0,\alpha}_b, \Phi_0,\Phi_\alpha > 0$ be such that \eqref{eq:C^0,alpha-bounds} holds and $X\in V^r$ such that \eqref{eq:VpX-small} holds for some $\epsilon > 0$.
Recall our solution space
\begin{equation}
\label{eq:3}
\calY(X) := \Set{ Y \given \forall I \quad V^r Y_{I} \leq 2 \Phi_{0} V^r X_{I} }.
\end{equation}
Then, every solution to the Young ODE \eqref{eq:YoungODE} lies in $\calY(X)$, and $\calY(X)$ is invariant under
\begin{equation}
\label{eq:Picard-step}
Y \mapsto \Bigl( y_{0} + \int_{0}^{t} \phi(Y) \dif X \Bigr)_{t}.
\end{equation}
\end{lemma}
\begin{proof}
Proposition~\ref{prop:Young-apriori} shows that every solution lies in $\calY(X)$.
By Corollary~\ref{cor:1} and the hypothesis \eqref{eq:VpX-small}, for every $Y \in \calY(X)$ and any $s\leq t$, we have
\[
\abs{\int_{s}^{t}\phi(Y)\dif X - \phi(Y_{s}) X_{s,t}}
\leq
2^{\alpha} \zeta(\theta) \Phi_\alpha \Phi_{0}^{\alpha}
V^r X_{[s,t)}^{\alpha} V^r X_{(s,t]}
\leq
\Phi_{0} V^r X_{[s,t]}.
\]
This implies that \eqref{eq:Picard-step} maps $\calY(X)$ to $\calY(X)$.
\end{proof}

\subsection{Convergence of Picard iteration}

Before we show that the map \eqref{eq:Picard-step} is a contraction, we need the following estimate:

\begin{lemma}[Stability of composition]
\label{lem:composition-stable}
Let $r\geq 1$ and $\alpha \in (0,1]$.
Suppose $\phi \in C^{1,\alpha}_{b}$ and $\tphi \in C^{0,\alpha}_{b}$.
Let $Y,\tY\in V^r$ be paths.
Then we have for any interval $I\subset[0,T]$:
\begin{equation}
\label{eq:composition-Delta-phi}
V^{r/\alpha}(\phi(Y)-\tphi(\tY))_I
\leq
\norm{D\phi}_{\sup} V^{r/\alpha}\Delta Y_I
+ \frac{\norm{D\phi}_{C^{\alpha}}}{1+\alpha} (V^rY^{\alpha}_I + V^r\tY^{\alpha}_I) \norm{\Delta Y}_{\sup}
+ \norm{\Delta\phi}_{C^{\alpha}} V^r\tY^{\alpha}_I.
\end{equation}
\end{lemma}

\begin{proof}
This is a direct consequence of applying $V^{r/\alpha}$ to both sides of the 4 point Taylor formula \eqref{eq:4point-Taylor-ord-1}.
\end{proof}

\begin{proposition}[Contractivity of Picard iteration]
\label{prop:Picard-contraction}
Let $\Phi_{\alpha},\Phi_{0},\Phi_1,\Phi_{1,\alpha} > 0$.
Let $\phi \in C^{1,\alpha}_{b}$ with \eqref{eq:C^0,alpha-bounds} and
\begin{equation}
\label{eq:C^1,alpha-bounds}
\norm{D\phi}_{\sup} \leq \Phi_{1},
\quad
\norm{D\phi}_{C^{\alpha}} \leq \Phi_{1,\alpha}.
\end{equation}
Let $\epsilon>0$ with \eqref{eq:VpX-small} and
\begin{equation}
\label{eq:Young-Picard-interval1}
\zeta(\theta) \Phi_{1} \epsilon \leq 1/8,
\quad
\zeta(\theta) 2^\alpha \Phi_0^{\alpha}\frac{\Phi_{1,\alpha}}{1+\alpha} \epsilon^{1+\alpha} \leq 1/8.
\end{equation}
Then, for any $X\in V^r$ with $V^rX_{[0,T]} < \epsilon$ and any $y_{0}\in \bbR^d$, the map \eqref{eq:Picard-step} is a strict contraction on the set
\[
\Set{ Y \in \calY(X) \given Y_{0} = y_{0}}
\]
with the metric
\[
d(Y,\tY) = \inf\Set{ \delta\ge 0 \given \forall I \subseteq [0,T] \quad V^r\Delta Y_{I} \leq \delta V^rX_{I} }\,.
\]
\end{proposition}

\noindent Proposition~\ref{prop:Picard-contraction} and Lemma~\ref{lem:solution-space} imply that the Young differential equation locally has a unique solution for $\phi \in C^{1,\alpha}_{b}$. Under some natural conditions on our constants, one can reduce some of our assumptions. Classical estimates give
\begin{align*}
\norm{D\phi}_{\sup}^{1+\alpha} &\lesssim \norm{\phi}_{\sup}^{\alpha} \norm{D\phi}_{C^{\alpha}}\\
\norm{\phi}_{C^\alpha} &\le \norm{D\phi}_{\sup}^\alpha (2\norm{\phi}_{\sup})^{1-\alpha}\,,
\end{align*}
allowing us to choose $\Phi_0,\Phi_\alpha,\Phi_1,\Phi_{1,\alpha}$ in such a way that
\begin{align*}
\Phi_1^{1+\alpha}&\lesssim \Phi_0^\alpha \Phi_{1,\alpha}\\
\Phi_\alpha &\leq \Phi_{1}^{\alpha} (2\Phi_0)^{1-\alpha}\,.
\end{align*}
Under this condition, the second bound in \eqref{eq:Young-Picard-interval1} implies the first up to a multiplicative constant, and the first bound in \eqref{eq:Young-Picard-interval1} implies \eqref{eq:VpX-small} up to a multiplicative constant.

\begin{proof}
Let $Y,\tY \in \calY(X)$ with $Y_0 = \tilde Y_0 = y_0$ and $d(Y,\tY)=\delta$.
Then,
\[
\abs{\Delta Y_{s}}
=
\abs{\Delta Y_{0,s}}
\leq
\delta \epsilon.
\]
By Lemma~\ref{lem:composition-stable}, we have
\begin{align*}
V^{r/\alpha}(\phi(Y) - \phi(\tY))_I
&\leq
\Phi_{1} V^{r/\alpha}\Delta Y_I
+ \frac{\Phi_{1,\alpha}}{1+\alpha} (V^rY_{I}^{\alpha} + V^r\tY_{I}^{\alpha}) \norm{\Delta Y}_{\sup,I}
\\ \text{(by \eqref{eq:3})}
&\leq
\Phi_{1} V^r\Delta Y_{I}
+ 2^{1+\alpha} \Phi_0^{\alpha}\frac{\Phi_{1,\alpha}}{1+\alpha} V^rX_{I}^{\alpha} \norm{\Delta Y}_{\sup,I}
\\ &\leq
\Phi_{1} \delta V^rX_{I}
+ 2^{1+\alpha}\Phi_0^\alpha \frac{\Phi_{1,\alpha}}{1+\alpha} V^rX_{I}^{\alpha} \delta \epsilon
\\ &\leq
\delta \epsilon \left( \Phi_{1}
+ 2^{1+\alpha}\Phi_0^{\alpha} \frac{\Phi_{1,\alpha}}{1+\alpha}  \epsilon^{\alpha} \right).
\end{align*}
Here, \[
\norm{\Delta Y}_{\sup,I} := \sup_{x\in I} \abs{\Delta Y_x}\,.
\]
Let $Z_{T} := y_{0} + \int_{0}^{T} \phi(Y) \dif X$, $\tZ_{T} := y_{0} + \int_{0}^{T} \phi(\tY) \dif X$.
By the sewing lemma (Theorem~\ref{thm:sewing}, cf.\ Lemma~\ref{lem:integration-stable} with $\tX=X$), we obtain
\begin{align*}
\abs{Z_{s,u}-\tZ_{s,u}}
&\leq
\zeta(\theta) V^{r/\alpha}(\phi(Y) - \phi(\tY))_{[s,u) }  V^rX_{(s,u]}
+  \Phi_{1} \abs{Y_s-\tY_s} \abs{X_{s,u}}
\\ &\leq
\zeta(\theta) \delta \epsilon ( \Phi_{1}
+ 2^{1+\alpha}\Phi_0^{\alpha} \frac{\Phi_{1,\alpha}}{1+\alpha}  \epsilon^{\alpha} )
V^r X_{(s,u]}
+ \delta \epsilon \Phi_{1} \abs{X_{s,u}}.
\end{align*}
This implies
\[
d(Z,\tZ)
\leq
\zeta(\theta)\delta\epsilon (\Phi_{1} + 2^{1+\alpha} \Phi_0^{\alpha}\frac{\Phi_{1,\alpha}}{1+\alpha} \epsilon^{\alpha}) + \delta \Phi_{1}\epsilon
\leq \delta/2,
\]
where we used the hypothesis \eqref{eq:Young-Picard-interval1} in the last step.
This shows the strict contractivity of the map \eqref{eq:Picard-step}.
\end{proof}

\subsection{Lipschitz dependence on data}
Before we show the desired Lipschitz estimate in this section, we show a preliminary result.

\begin{lemma}[Stability of Young integration]
\label{lem:integration-stable}
Let $X,\tX, \in V^{r}, Y,\tilde Y\in V^{r/\alpha}$ for $r\in [1,2),\alpha\in(0,1]$ such that $r<1+\alpha$.
Then,
\begin{multline}
\label{eq:integral-Delta}
\abs{\int_{s}^{t} Y \dif X - \int_{s}^{t} \tY \dif \tX - (Y_{s}X_{s,t} - \tY_{s}\tX_{s,t})}
\leq
\\
2^{\theta-1}\zeta(\theta) V^{r/\alpha}\Delta Y_{[s,t)} V^rX_{(s,t]}
+ 2^{\theta-1}\zeta(\theta) V^{r/\alpha}\tY_{[s,t)} V^r\Delta X_{(s,t]}\,,
\end{multline}
where $\theta = (1+\alpha)/r$ and $\zeta(\theta) = \sum_{k=1}^\infty k^{-\theta}$.
\end{lemma}

\begin{proof}
Let $A_{s,t} := Y_s X_{s,t}-\tY_s \tX_{s,t}$.
Then,
\begin{equation*}
\abs{\int_s^t Y \dif X - \int_s^t \tilde Y \dif\tilde X - (Y_{s}X_{s,t} - \tY_{s}\tX_{s,t}) }
=
\abs{\calI A_{s,t} - A_{s,t} }.
\end{equation*}
We have
\[
\delta A_{\tau,u,\nu}
= Y_{\tau,u}X_{u,\nu} - \tY_{\tau,u}\tX_{u,\nu}
= \Delta Y_{\tau,u} X_{u,\nu} + \tY_{\tau,\nu} \Delta X_{u,\nu}.
\]
By the sewing lemma (Theorem~\ref{thm:sewing}), this gives the claimed bound.
\end{proof}

\noindent For our main result, we need to strengthen the smallness conditions on $\epsilon$ by the constant $2^{\theta-1}$. To this end, we strengthen \eqref{eq:VpX-small} to
\begin{equation}\label{eq:VpX-small1}
2^{\alpha+\theta-1} \zeta(\theta) \Phi_\alpha \epsilon^{\alpha} \leq \Phi_0^{1-\alpha},
\end{equation}
and \eqref{eq:Young-Picard-interval1} to
\begin{equation}\label{eq:Young-Picard-interval2}
2^{\theta-1}\zeta(\theta) \Phi_{1} \epsilon \leq 1/8,
\quad
2^{\theta-1}\zeta(\theta) 2^\alpha \Phi_0^{\alpha}\frac{\Phi_{1,\alpha}}{1+\alpha} \epsilon^{1+\alpha} \leq 1/8.
\end{equation}
With this in hand, we can show this section's main result:

\begin{theorem}[Lipschitz stability]
\label{thm:Lipschitz-stability}
Let $\phi \in C^{1,\alpha}_{b}$ satisfy \eqref{eq:C^0,alpha-bounds} and \eqref{eq:C^1,alpha-bounds}.
Let $\tphi \in C^{0,\alpha}_{b}$ satisfy \eqref{eq:C^0,alpha-bounds}.
Suppose that $\epsilon$ satisfies the smallness conditions \eqref{eq:VpX-small1} and \eqref{eq:Young-Picard-interval2}.

Let $X,\tX$ be paths with $V^r X_{[0,T]} < \epsilon$ and $V^r\tX_{[0,T]} < \epsilon$.
Let $y_{0}, \tilde{y}_{0}\in\bbR^d$ be initial data.
Let $Y$, $\tY$ be solutions of Young ODEs with respective drivers, coefficients, and initial data.
Then, for every interval $I \subseteq [0,T]$, we have
\begin{equation}
\label{eq:Young-loc-Lip1}
V^r\Delta Y_{I}
\leq
\gamma V^rX_{I} + 2\Phi_{0} V^r\Delta X_I
\end{equation}
with
\begin{align*}
\gamma \lesssim &
((\Csew+1)\Phi_{1} + \Csew 2^{\alpha+1} \frac{\Phi_{1,\alpha}}{1+\alpha} \Phi_0^{\alpha} \epsilon^{\alpha}) \Phi_{0} V^r \Delta X_{[0,T]}  +
(\Csew 2^{\alpha+1} \frac{\Phi_{1,\alpha}}{1+\alpha} \Phi_0^{\alpha} \epsilon^{\alpha} + \Phi_{1}) \abs{y_{0}-\tilde{y}_{0}}
\\&\quad+ \Csew 2^\alpha \Phi_0^\alpha \epsilon^\alpha \norm{\Delta\phi}_{C^{\alpha}}
+ \norm{\Delta\phi}_{\sup}\,,
\end{align*}
where $\Csew = 2^{\theta-1}\zeta(\theta)$.
\end{theorem}

\noindent Setting $X=0$, $\phi=0$, this qualitatively recovers the a priori estimate in Proposition~\ref{prop:Young-apriori} (up to a larger sewing constant).
\begin{proof}
By Proposition~\ref{prop:Young-apriori}, we have $Y \in \calY(X)$ and $\tY \in \calY(\tX)$.
It follows that the conclusion \eqref{eq:Young-loc-Lip1} holds with $\gamma=4\Phi_{0}$.

Let $\gamma$ be the smallest number such that \eqref{eq:Young-loc-Lip1} holds.
Let $\gamma' = 2\Phi_{0}$.
Let
\[
\beta := V^r\Delta X_{[0,T]},
\quad
\beta' := \abs{y_{0}-\tilde{y}_{0}},
\quad
\beta'' := \norm{\Delta\phi}_{C^{\alpha}},
\quad
\beta''' := \norm{\Delta\phi}_{\sup}.
\]
By our choice of $\gamma,\gamma'$, we have
\[
\delta'' := \norm{\Delta Y}_{\sup}
\leq \beta' + \gamma \epsilon + \gamma' \beta.
\]
Note
\begin{align*}
\abs{\Delta \phi(Y_{s}) X_{s,t}}
&\leq
\abs{\phi(Y_{s}) X_{s,t} - \phi(\tY_{s}) X_{s,t}} + \abs{\phi(\tY_{s}) X_{s,t} - \tphi(\tY_{s}) X_{s,t}} + \abs{\tphi(\tY_{s}) X_{s,t} - \tphi(\tY_{s}) \tX_{s,t}}
\\ &\leq
\Phi_{1} \abs{\Delta Y_{s}} \abs{X_{s,t}}
+
\beta''' \abs{X_{s,t}}
+
\tPhi_0 \abs{\Delta X_{s,t}},
\end{align*}
and thus
\begin{align*}
V^r(\Delta \phi(Y_{s}) X_{s,t})_{I}
&\leq
(\Phi_{1} \norm{\Delta Y}_{\sup}+\beta''') V^rX_{I}
+
\tPhi_0 V^r\Delta X_{I}
\\ &\leq
(\Phi_{1} \gamma \epsilon + \Phi_{1} \gamma' \beta + \Phi_{1}\beta' + \beta''') V^r X_{I}
+
\tPhi_0 V^r\Delta X_{I}.
\end{align*}
By Lemma~\ref{lem:composition-stable}, we have
\begin{align*}
V^{r/\alpha} (\phi(Y) - \tphi(\tY))_{I}
&\leq
\Phi_{1} V^{r/\alpha}\Delta Y_{I}
+ \frac{\Phi_{1,\alpha}}{1+\alpha} (V^r Y_{I}^{\alpha} + V^r\tY_{I}^{\alpha}) \norm{\Delta Y}_{\sup,I}
+ \norm{\Delta\phi}_{C^{\alpha}} V^r\tY_{I}^{\alpha}
\\ \text{(by \eqref{eq:3}) } &\leq
\Phi_{1} V^r\Delta Y_{I}
+ 2^{\alpha}\frac{\Phi_{1,\alpha}}{1+\alpha} (\Phi_0^{\alpha}V^rX_{I}^{\alpha} + \tPhi_0^{\alpha}V^r \tX_{I}^{\alpha}) \delta''
+ 2^{\alpha} \beta'' \tPhi_0^{\alpha} V^r\tX_{I}^{\alpha}
\\ \text{(by \eqref{eq:Young-loc-Lip1}) } &\leq
\Phi_{1} ( \gamma \epsilon + \gamma' \beta )
+
2^{1+\alpha} \frac{\Phi_{1,\alpha}}{1+\alpha} \Phi_0^{\alpha} \epsilon^{\alpha} \delta''
+ 2^{\alpha} \beta'' \tPhi_0^{\alpha} \epsilon^{\alpha}.
\end{align*}
By \eqref{eq:3} and \eqref{eq:VpX-small1}, we have
\[
V^{r/\alpha}\tphi(\tY)_{I}
\leq
\tPhi_{\alpha} V^r\tY_{I}^{\alpha}
\leq
2^{\alpha} \tPhi_{\alpha} \tPhi_{0}^{\alpha} V^r\tX_{I}^{\alpha}
\leq
2^{\alpha} \tPhi_{\alpha} \tPhi_{0}^{\alpha} \epsilon^{\alpha}
\leq
\Csew^{-1} \Phi_{0}\,.
\]
By Lemma~\ref{lem:integration-stable}, we have
\begin{align*}
\MoveEqLeft
\abs{\Delta Y_{s,t} - \Delta \phi(Y_{s}) X_{s,t}}
\\ & \leq
\Csew V^{r/\alpha}\Delta \phi(Y)_{[s,t)} V^rX_{(s,t]}
+ \Csew V^{r/\alpha}\tphi(\tY)_{[s,t)} V^r\Delta X_{(s,t]}
\\ &\leq
\Csew (\Phi_{1} ( \gamma \epsilon + \gamma' \beta ) +
2^{\alpha+1} \frac{\Phi_{1,\alpha}}{1+\alpha} \Phi_0^{\alpha} \epsilon^{\alpha} \delta''
+ 2^{\alpha} \beta'' \tPhi_0^{\alpha} \epsilon^{\alpha}) V^rX_{(s,t]}
+ \Phi_0 V^r\Delta X_{(s,t]}.
\end{align*}
Hence, using \eqref{eq:VpX-small1} and \eqref{eq:Young-Picard-interval2} together with $\Csew \ge 1$,
\begin{equation}
\label{eq:6}
\begin{aligned}
V^r\Delta Y_{I}
&\leq
V^r(\Delta (Y_{s,t} - \phi(Y_{s})X_{s,t}))_{I} + V^r (\Delta\phi(Y_{s})X_{s,t})_{I}
\\ & \leq
\Csew (\Phi_{1} ( \gamma \epsilon + \gamma' \beta ) +
2^{\alpha+1} \frac{\Phi_{1,\alpha}}{1+\alpha} \Phi_0^{\alpha} \epsilon^{\alpha} (\beta' + \gamma\epsilon + \gamma'\beta)
+ 2^{\alpha} \beta'' \tPhi_0^{\alpha} \epsilon^{\alpha}) V^rX_{I}
\\ &\quad +
\Phi_0 V^r(\Delta X)_{I}
\\ &\quad +
(\Phi_{1} \gamma \epsilon + \Phi_{1} \gamma' \beta + \Phi_{1}\beta' + \beta''') V^r X_{I}
+
\tPhi_0 V^r \Delta X_{I}
\\ & \leq
(\gamma/2 + \Csew \Phi_1\gamma' \beta  +
\Csew 2^{\alpha+1} \frac{\Phi_{1,\alpha}}{1+\alpha} \Phi_0^{\alpha} \epsilon^{\alpha} (\beta'+\gamma'\beta)
+ \Csew 2^\alpha \Phi_0^\alpha \epsilon^\alpha \beta''
\\ & \quad + \Phi_{1} \gamma' \beta + \Phi_{1} \beta' + \beta''') V^rX_{I}
+
2 \tPhi_0 V^r\Delta X_{I}.
\end{aligned}
\end{equation}
So, we obtain
\begin{multline*}
\gamma \leq
\gamma/2 + \Csew \Phi_1\gamma' \beta  +
\Csew 2^{\alpha+1} \frac{\Phi_{1,\alpha}}{1+\alpha} \Phi_0^{\alpha} \epsilon^{\alpha} (\beta'+\gamma'\beta)
+ \Csew 2^\alpha \Phi_0^\alpha \epsilon^\alpha \beta''
+ \Phi_{1} \gamma' \beta + \Phi_{1} \beta' + \beta'''.
\end{multline*}
This implies
\[
\gamma \lesssim
\Csew \Phi_{0}\Phi_1 \beta  +
\Csew 2^{\alpha+1} \frac{\Phi_{1,\alpha}}{1+\alpha} \Phi_0^{\alpha} \epsilon^{\alpha} (\beta'+\Phi_{0}\beta)
+ \Csew 2^\alpha \Phi_0^\alpha \epsilon^\alpha \beta''
+ \Phi_{1} \Phi_{0} \beta + \Phi_{1} \beta' + \beta'''.
\]
\end{proof}

\section{Lipschitz estimates for RDEs}\label{sec:LipschitzRP}

The goal of this section is to derive a similar estimate to \eqref{eq:Young-loc-Lip1} for solutions of rough differential equations
\begin{equation}\label{eq:RDESec6}
Y_t = y_0 + \int_0^t \phi(\bfY_s)\dif \bfX_s\,.
\end{equation}
Although similar estimates can be found in textbooks \cite{FH20, FV10}, the existing versions do not seem sufficiently strong to recover Besov norm estimates from variation norm estimates.
For instance, unpacking \cite[Corollary 10.27]{FV10} leads to quotients of the form $V^{r/2}(\Delta \bbX) /(V^r X + (V^{r/2}\bbX)^{\frac 12} + V^r \tilde X + (V^{r/2}\tilde\bbX)^{\frac 12})$, which we have been unable to bound in Besov norm.

Throughout this section, we fix an $r\in [2,3)$ and an $\alpha\in(0,1]$. Let $\bfX = (X,\bbX)$ be an $r$-rough path over $\bbR^n$ and we assume that $\phi:\bbR^d\to L(\bbR^n,\bbR^d)$. Eventually, we will assume that $\phi\in C^{2,1}$, but we will clarify for each Lemma the minimal regularity for $\phi$.

To keep this section more concise, we omit the dependence on $\Delta\phi$ by assuming $\phi =\tilde\phi$ and only track the dependences on $\Delta X, \Delta\bbX$ and $\Delta y_0$. We further use a solution space depending on $V^r X_I, V^{r/2}\bbX_I$ for each $I\subset[0,T]$, similar to \eqref{eq:SolutionSpaceYoung}. The precise definition can be found in \eqref{eq:RDE-solution-space}.

As in Section \ref{sec:LipschitzYoung}, we show an a priori estimate (Proposition \ref{prop:Young-apriori-RDE}), local existence and uniqueness (Theorem \ref{thm:RDE-solution-local-uniqueness}) and the desired Lipschitz estimate (Theorem \ref{thm:LipschitzDependenceRDE}).

\subsection{Operations on controlled rough paths}

Let us recall some stability estimates on controlled rough paths. Similar estimates can be found in Gubinelli's original paper \cite{MR2091358}, as well as \cite{FH20} in the Hölder setting.

Throughout this subsection, we fix two $r$-rough paths $\bfX,\tilde\bfX$ with controlled rough paths $\bfY,\tilde\bfY$. As in the Young case, we use the notation
\begin{multline*}
\Delta Y = Y - \tY,
\quad
\Delta Y' = Y' - \tY',
\quad
\Delta R^{\bfY,\bfX} = R^{\bfY,\bfX} - R^{\tilde\bfY,\tilde\bfX},
\quad
\Delta X = X - \tX,
\quad
\Delta \bbX = \bbX - \tilde{\bbX},
\\ \quad
\Delta \phi(Y) = \phi(Y) - \phi(\tY),
\quad
\Delta R^{\phi(\bfY),\bfX} = R^{\phi(\bfY),\bfX} - R^{\phi(\tilde\bfY),\tilde\bfX}.
\end{multline*}
The next lemma defines $\phi(\bfY)$ and shows some stability estimates for the expression:

\begin{lemma}[Composition with a smooth function]
\label{lem:controlled-composition}
Let $r\in[2,3)$, $\alpha\in(0,1]$ and let  $\phi \in C^{1,\alpha}_{b}$.
Then $\phi(\bfY) := (\phi(Y), D\phi(Y)Y')$ is an $\bfX$-controlled $(r,\alpha)$-rough path, and we have for each interval $I\subset[0,T]$:
\begin{equation}
\label{eq:controlled-composition-phi'}
V^{r/\alpha}(\phi(\bfY)')_{I}
=
V^{r/\alpha}(D\phi(Y)Y')_{I}
\leq
\norm{D\phi}_{\sup} V^{r/\alpha}(Y')_{I} + \norm{D\phi}_{C^{\alpha}} V^rY_I^{\alpha} \norm{Y'}_{\sup},
\end{equation}
\begin{equation}
\label{eq:controlled-composition-R}
V^{r/(1+\alpha)} (R^{\phi(\bfY),\bfX} )_{I}
\leq
\norm{D\phi}_{\sup} V^{r/(1+\alpha)}(R^{\bfY,\bfX})_{I}
+
\frac{1}{1+\alpha} \norm{D \phi}_{C^{\alpha}} V^rY_I^{1+\alpha}.
\end{equation}
Furthermore, if $\phi\in C^{2,1}$, we have the following stabillity estimates for any interval $I\subset[0,T]$:
\begin{equation}
\label{eq:controlled-composition-Delta-phi}
V^r(\Delta \phi(Y))_{I}
\leq
\norm{D\phi}_{\sup} V^r(\Delta Y)_{I} + \norm{D\phi}_{\Lip} \norm{\Delta Y}_{\sup} V^r\tY_{I},
\end{equation}
\begin{multline}
\label{eq:controlled-composition-Delta-phi'}
V^r(\Delta\phi(\bfY)')_{I}
\leq
\norm{D\phi}_{\sup} V^r(\Delta Y')_{I}
+ \norm{D\phi}_{\Lip} V^r Y_I \norm{\Delta Y'}_{\sup}
+ \norm{D\phi}_{\Lip} \norm{\Delta Y}_{\sup} V^r (\tY')_{I}\\
+ \norm{D^{2}\phi}_{\sup} V^r(\Delta Y)_{I} \norm{\tY'}_{\sup}
+ \norm{D^{2}\phi}_{\Lip} \norm{\Delta Y}_{\sup} V^r(\tY)_{I} \norm{\tY'}_{\sup},
\end{multline}
\begin{multline}
\label{eq:controlled-composition-Delta-R}
V^{r/2}(\Delta R^{\phi(\bfY),\bfX})_{I}
\leq
\norm{D\phi}_{\sup} V^{r/2} (\Delta R^{\bfY,\bfX})_{I}
+ \norm{D\phi}_{\Lip} \norm{\Delta Y}_{\sup} V^{r/2}(R^{\tilde\bfY,\tilde\bfX})_{I}
\\+\frac12 \norm{D^{2}\phi}_{\sup} (V^r Y_I+V^r\tY_I) V^r(\Delta Y)_I
+\frac12 \norm{D^{2}\phi}_{\Lip} \norm{\Delta Y}_{\sup} V^r(\tY)_I^2.
\end{multline}
\end{lemma}

\noindent Interestingly, these estimates do not depend on $\bfX$.

\begin{proof}
Adding and subtracting $D\phi(Y_t)Y'_s$, we obtain
\begin{equation}
\label{eq:4RDE}
D\phi(Y_t)Y'_t- D\phi(Y_s)Y'_s
=
D\phi(Y_t)(Y'_t-Y'_{s}) + (D\phi(Y_t)-D\phi(Y_s)) Y'_s.
\end{equation}
This implies \eqref{eq:controlled-composition-phi'}. By Taylor's formula,
\begin{align*}
\abs{R^{\phi(\bfY),\bfX}_{s,t}}
&=
\abs{\phi(Y_t)- \phi(Y_s) -D \phi(Y_s) Y_s' X_{s,t}}
\\ &\leq
\abs{\phi(Y_t)- \phi(Y_s) -D \phi(Y_s) Y_{s,t}} + \abs{D\phi(Y_s) R^{\bfY,\bfX}_{s,t}}
\\ &\leq
\frac{1}{1+\alpha} \norm{D \phi}_{C^{\alpha}} \abs{Y_{s,t}}^{1+\alpha} + \norm{D\phi}_{\sup} \abs{R^{\bfY,\bfX}_{s,t}},
\end{align*}
which implies \eqref{eq:controlled-composition-R}.

Let us now tackle the stability estimates \eqref{eq:controlled-composition-Delta-phi} - \eqref{eq:controlled-composition-Delta-R}, starting with \eqref{eq:controlled-composition-Delta-phi}.
\begin{align}
\begin{split}
\label{eq:77}
(\phi(Y_t)-\phi(Y_s)) &- (\phi(\tY_t)-\phi(\tY_s))
\\&=
\int_{0}^{1} D\phi(Y_{s}+r Y_{s,t}) Y_{s,t} \dif r
-
\int_{0}^{1} D\phi(\tY_{s}+r\tY_{s,t}) \tY_{s,t} \dif r
\\&=
\int_{0}^{1} \Bigl(
D\phi(Y_{s}+rY_{s,t}) (Y_{s,t} - \tY_{s,t})
\\&\quad +
(D\phi(Y_{s}+rY_{s,t}) - D\phi(\tY_{s}+r\tY_{s,t})) \tY_{s,t} \Bigr) \dif r\,.
\end{split}
\end{align}
Then
\begin{align*}
\abs{D\phi(Y_{s}+r Y_{s,t}) - D\phi(\tY_{s}+r\tY_{s,t})}
 &\leq
\norm{D\phi}_{\Lip} \abs{(Y_{s}+rY_{s,t}) - (\tY_{s}+r\tY_{s,t})}
\\ &\leq
\norm{D\phi}_{\Lip} (r\abs{\Delta Y_t} + (1-r) \abs{\Delta Y_s})
\\ &\leq
\norm{D\phi}_{\Lip} \norm{\Delta Y}_{\sup}\,,
\end{align*}
shows \eqref{eq:controlled-composition-Delta-phi}. In order to show \eqref{eq:controlled-composition-Delta-phi'}, we write
\begin{multline*}
(D\phi(Y_t)Y'_t-D\phi(\tY_t)\tY'_t) - (D\phi(Y_s)Y'_s-D\phi(\tY_s)\tY'_s)
\\ =
D\phi(Y_t) (\Delta Y'_t-\Delta Y'_{s})
+ (D\phi(Y_t)-D\phi(Y_s)) \Delta Y'_s
\\ + (D\phi(Y_{t})-D\phi(\tY_t))\tY'_{s,t}
+ ((D\phi(Y_t)-D\phi(Y_s))-(D\phi(\tY_t)-D\phi(\tY_s))) \tY'_s.
\end{multline*}
The first 3 terms contribute the first 3 terms to \eqref{eq:controlled-composition-Delta-phi'}.
The 4-fold difference in the last term can be written as in \eqref{eq:77} with $D\phi$ in place of $\phi$, which gives the last 2 terms in \eqref{eq:controlled-composition-Delta-phi'}.

Now we show \eqref{eq:controlled-composition-Delta-R}.
By adding and subtracting $D\phi(Y_s)Y_{s,t}$ and $ D\phi(\tY_s) \tY_{s,t},$ one has
\begin{equation}
\label{eq:55}
\begin{split}
R^{\phi(\bfY),\bfX}_{s,t}- R^{\phi(\tilde\bfY),\tilde\bfX}_{s,t}
&= \phi(Y)_{s,t} - D \phi(Y_s) Y_s' X_{s,t} - (\phi(\tY)_{s,t} -D \phi(\tY_s) \tY_s' \tX_{s,t})
\\ &=
\phi(Y)_{s,t}-D\phi(Y_s)Y_{s,t} - (\phi(\tY)_{s,t}- D\phi(\tY_s) \tY_{s,t})
\\ &+ D\phi(Y_s)R^{\bfY,\bfX}_{s,t}- D\phi(\tY_s)R^{\tilde\bfY,\tilde\bfX}_{s,t}
\end{split}
\end{equation}
By Taylor's formula, the first line in \eqref{eq:55} equals
\[
\int_0^1 (1-r) \bigl( D^2\phi(Y_s + r Y_{s,t}) Y_{s,t}^{\otimes 2} - D^2\phi(\tY_s + r \tY_{s,t}) \tY_{s,t}^{\otimes 2} \bigr) \dif r.
\]
The integrand can be written in the form
\begin{equation*}
D^2\phi(Y_s + r Y_{s,t}) (Y_{s,t}^{\otimes 2} - \tY_{s,t}^{\otimes 2})
+
(D^2\phi(Y_s + r Y_{s,t}) - D^2\phi(\tY_s + r \tY_{s,t})) \tY_{s,t}^{\otimes 2},
\end{equation*}
and this contributes the last line to \eqref{eq:controlled-composition-Delta-R}.

The second line in \eqref{eq:55} equals
\begin{align*}
D\phi(Y_s)R^{\bfY,\bfX}_{s,t}- D\phi(\tY_s)R^{\tilde\bfY,\tilde\bfX}_{s,t}
&=
D\phi(Y_s) (R^{\bfY,\bfX}_{s,t}-R^{\tilde\bfY,\tilde\bfX}_{s,t})
+ (D\phi(Y_s)-D\phi(\tY_s)) R^{\tilde\bfY,\tilde\bfX}_{s,t}
\end{align*}
which completes the proof.
\end{proof}

\noindent Recall that the rough path integral $\int_0^t \bfY_z\dif\bfX_z$ of a $\bfX$-controlled rough path $\bfY$ is given by the sewing of
\[
	\Xi_{s,t} = Y_s X_{s,t} +Y'_s\bbX_{s,t}\,.
\]
The next lemma shows that said integral exists and gives some useful estimates.

\begin{lemma}[Rough path integration]
\label{lem:integral-bound}
Let $r \in [2,3)$ and $\alpha \in (0,1]$ with $r < 2+\alpha$. We set $\theta = (2+\alpha)/r > 1$.
Let $\bfX$ be a $r$-rough path and $\bfY$ an $\bfX$ controlled $(r,\alpha)$-rough path.
Then, $Z_{T} := \int_{0}^{T} \bfY \dif \bfX$ exists and satisfies for $\Csew = 2^{\theta-1} \zeta(\theta)$ with $\zeta(\theta) = \sum_{k=1}^\infty k^{-\theta}$
\begin{equation}\label{ineq:integral-bound}
\Csew^{-1}\abs{Z_{s,t} - Y_{s} X_{s,t} - Y'_{s} \bbX_{s,t}}
\leq
V^{r/(1+\alpha)}(R^{\bfY,\bfX})_{[s,t)} V^{r}X_{(s,t]}
+ V^{r/\alpha}(Y')_{[s,t)} V^{r/2}\bbX_{(s,t]}\,.
\end{equation}
 This especially shows that $\bfZ_t = (Z_t,Y_t)$ is an $\bfX$-controlled $r$-rough path. Furthermore, if we set $\tilde \bfZ_t := \int_0^t \tilde\bfY_s\dif\tilde\bfX_s$, we have the stability estimate
\begin{multline}
\label{eq:controlled-integral-remainder-Delta}
V^{r/2}(\Delta R^{\bfZ,\bfX})_I
\le
\norm{\Delta Y'}_{\sup} V^{r/2}\bbX_{I}
+ \norm{\tilde{Y}'}_{\sup} V^{r/2}(\Delta \bbX)_{I}
+ \zeta(\theta) 4^{\theta-1}\bigg(V^r(\Delta Y')_I V^{r/2}\bbX_I
\\ 
+ V^{r/2}(\Delta R^{\bfY,\bfX})_{I} V^r X_{I}
+ V^r(\tY')_{I} V^{r/2}(\Delta \bbX)_I
+ V^{r/2}(R^{\tilde \bfY,\tilde\bfX})_{I} V^r(\Delta X)_I\bigg)\,,
\end{multline}
for all intervals $I\subset[0,T]$.
\end{lemma}

\noindent Note that \eqref{eq:controlled-integral-remainder-Delta} with $\tilde\bfY=0$ gives:
\[
	V^{r/2}(R^{\bfZ,\bfX})_{I} \le \norm{Y'}_{\sup} V^{r/2}(\bbX)_{I} + c\bigg(V^r(Y')_I V^{r/2}(\bbX)_{I} + V^{r/2}(R^{\bfY,\bfX})_{I} V^r X_I\bigg)\,,
\]
where $c= \zeta(\theta) 4^{\theta-1}$. Although going through the proof, it is easy to see that one could also use $c=\zeta(\theta) 2^{\theta-1}$. Further observe that by Lemma \ref{lem:controlled-implicit}, we immediately get a stability estimate on $\Delta Z$:
\[
	V^r(\Delta Z)_I \le \norm{\Delta Y}_{\sup} V^r X_I + \norm{\tilde Y}_{\sup} V^r(\Delta X)_I + V^{r/2}(\Delta R^{\bfZ,\bfX})_I\,.
\]

\begin{proof}
Let
\[
\Xi_{s,t} := Y_{s}X_{s,t} + Y'_{s} \bbX_{s,t}.
\]
Then
\begin{equation}
\label{eq:delta-Xi-2-rough-integral}
\begin{split}
\delta \Xi_{s,t,u}
&=
Y_{s}  X_{s,u} + Y_{s}' \bbX_{s,u}
- (Y_{s} X_{s,t} + Y_{s}' \bbX_{s,t}) - (Y_{t} X_{t,u} + Y_{t}' \bbX_{t,u})
\\ &=
Y_{s} X_{t,u} + Y_{s}' (\bbX_{s,u} - \bbX_{s,t}) - (Y_{t} X_{t,u} + Y_{t}' \bbX_{t,u})
\\ &=
-Y_{s,t} X_{t,u} + Y_{s}' (\bbX_{t,u} + X_{s,t} \otimes X_{t,u}) - Y_{t}' \bbX_{t,u}
\\ &=
-Y_{s,t} X_{t,u} + Y'_{s,t} \bbX_{t,u} + Y_{s}' (X_{s,t} \otimes X_{t,u})
\\ &=
-R^{\bfY,\bfX}_{s,t} X_{t,u} + Y'_{s,t} \bbX_{t,u}.
\end{split}
\end{equation}
Then, since
\[
1/r + 1/(r/(1+\alpha)) = 1/(r/\alpha) + 1/(r/2) = \theta > 1,
\]
the sewing lemma (Theorem~\ref{thm:sewing}) directly shows the existence of $Z$ as well as 
\[
	\abs{Z_{s,t} - Y_{s} X_{s,t} - Y'_{s} \bbX_{s,t}}
\leq
\zeta(\theta)\left(V^{r/(1+\alpha)}(R^{\bfY,\bfX})_{[s,t)}^{\frac 1\theta} V^{r}X_{(s,t]}^{\frac 1\theta}
+ V^{r/\alpha}(Y')_{[s,t)}^{\frac 1\theta} V^{r/2}\bbX_{(s,t]}^{\frac 1\theta}\right)^\theta\,.
\]
\eqref{ineq:integral-bound} Then follows from $(a+b)^\theta \le 2^{\theta-1}(a^\theta + b^\theta)$ for all $a,b > 0$.

We show the stability estimate \eqref{eq:controlled-integral-remainder-Delta}. We do this by applying the sewing lemma to the germ
\[
\Xi_{s,t} := Y_s X_{s,t} + Y'_s \bbX_{s,t}-(\tY_s \tX_{s,t} + \tY'_s \tilde{\bbX}_{s,t})\,.
\]
Observe that
\begin{align*}
\abs{ R^{\bfZ,\bfX}_{s,t} - R^{\tilde\bfZ,\tilde\bfX}_{s,t} }
&= \abs{\int_s^t \bfY \dif\bfX - Y_s X_{s,t} -(\int_s^t \tilde\bfY \dif\tilde\bfX - \tY_s \tX_{s,t})}
\\ &\leq
\abs{Y'_s\bbX_{s,t} - \tY'_s \tilde{\bbX}_{s,t} } + \abs{\calI(\Xi)_{s,t} - \Xi_{s,t} }.
\end{align*}
We estimate the first term by
\[
\abs{\Delta Y'_{s} \bbX_{s,t}} + \abs{\tilde{Y}'_{s} \Delta \bbX_{s,t}}\,,
\]
and use the sewing lemma for the second term. Note that by \eqref{eq:delta-Xi-2-rough-integral}, $\delta\Xi$ is of the following form:
\[
\delta(\Xi)_{\tau,u,\nu}
= -R^{\bfY,\bfX}_{\tau,u} X_{u,\nu} + Y'_{\tau,u} \bbX_{u,\nu} -(-R^{\tilde\bfY,\tilde\bfX}_{\tau,u} \tX_{u,\nu} + \tY'_{\tau,u} \tilde{\bbX}_{u,\nu}).
\]
and
\begin{align*}
\abs{\delta (\Xi)_{\tau,u,\nu}} &\leq
\abs{\Delta Y'_{\tau,u} \bbX_{u,\nu } } +\abs{\tY'_{\tau,u} \Delta \bbX_{u,\nu}} +\abs{\Delta R^{\bfY,\bfX}_{\tau,u} X_{u,\nu}} + \abs{R^{\tilde \bfY,\tilde\bfX}_{\tau,u} \Delta X_{u,\nu}}\\
&\le V^{r/\alpha}(\Delta Y')_{[\tau,u]}V^{r/2}(\bbX)_{[u,\nu]} + V^{r/\alpha}(\tilde Y')_{[\tau,u]} V^{r/2}(\Delta \bbX)_{[u,\nu]} \\&\qquad+ V^{r/(1+\alpha)}(\Delta R^{\bfY,\bfX})_{[\tau, u]}V^r X_{[u,\nu]} + V^{r/(1+\alpha)}(R^{\tilde\bfY,\tilde\bfX})_{[\tau,u]}V^r(\Delta X)_{[u,\nu]}\,.
\end{align*}
Thus, the sewing lemma gives us
\begin{align}
\begin{split}\label{ineq:stabilityRoughIntegral}
\abs{I(\Xi)_{s,t} -\Xi_{s,t}} &\le \zeta(\theta) \bigg(V^{r/\alpha}(\Delta Y')_{[s,t)}^{\frac 1\theta} V^r\bbX_{(s,t]}^{\frac 1\theta} + V^{r/\alpha}(\tilde Y')_{[s,t)}^{\frac 1\theta} V^{r/2}(\Delta \bbX)_{(s,t]}^{\frac 1\theta} \\&\qquad+ V^{r/(1+\alpha)}(\Delta R^{\bfY,\bfX})_{[s, t)}^{\frac 1\theta}V^rX_{(s,t]}^{\frac 1\theta} + V^{r/(1+\alpha)}(R^{\tilde\bfY,\tilde\bfX})_{[s,t)}^{\frac 1\theta}V^r(\Delta X)_{(s,t]}^{\frac 1\theta}\bigg)^\theta\\
&\le \zeta(\theta) 4^{\theta-1}  \bigg(V^{r/\alpha}(\Delta Y')_{[s,t)} V^{r/2}\bbX_{(s,t]} + V^{r/\alpha}(\tilde Y')_{[s,t)} V^{r/2}(\Delta \bbX)_{(s,t]} \\&\qquad+ V^{r/(1+\alpha)}(\Delta R^{\bfY,\bfX})_{[s, t)} V^r X_{(s,t]} + V^{r/(1+\alpha)}(R^{\tilde\bfY,\tilde\bfX})_{[s,t)}V^r(\Delta X)_{(s,t]}\bigg)
\end{split}
\end{align}
\end{proof}

\begin{remark}\label{rem:integral-bound}
Recall $\theta = (2+\alpha)/r$. We can apply the $1/\theta$-variation norm to inequality \eqref{ineq:integral-bound} to get
\begin{align*}
	&\Csew^{-1}V^{r/(2+\alpha)}(Z_{s,t} - Y_{s} X_{s,t} - Y'_{s} \bbX_{s,t})_{[s,t]}\\
	&\qquad\leq
	2^{\theta-1}\left(V^{r/(1+\alpha)}(R^{\bfY,\bfX})_{[s,t)} V^{r}X_{(s,t]}
+ V^{r/\alpha}(Y')_{[s,t)} V^{r/2}\bbX_{(s,t]}\right)\,,
\end{align*}
for any $0\le s\le t\le T$. Here, we used the quasi-additivity
\[
V^{1/\theta}(X+Y) \le 2^{\theta - 1}(V^{1/\theta} X+V^{1/\theta} Y)
\]
for all $\theta>1$.
\end{remark}

\subsection{A priori estmate}

Let us show an apriori estimate similar to Proposition \ref{prop:Young-apriori} in the Young case. In the Hölder case, a similar estimate for RDEs can be found in \cite[Section 8.4]{FH20}. We also want to mention the similar estimate in \cite[Theorem 2.9]{MR3957994} as well as \cite[Lemma 2.4]{MR2387018}, which corresponds to piecewise constant $X$.

We will require some bounds for the function $\phi$. Let $\Phi_0,\Phi_\alpha,\Phi_1,\Phi_{1,\alpha},\Phi_2, \Phi_{2,\alpha},\Phi_{DD}>0$ fulfill
\begin{align}
\begin{split}\label{ineq:boundsPhiRP}
\norm{\phi}_{\sup} &\le \Phi_0, \qquad \hspace{14pt}\norm{\phi}_{C^\alpha} \le \Phi_\alpha, \\
\norm{D\phi}_{\sup} &\le \Phi_1, \qquad \hspace{5pt}\norm{D\phi}_{C^\alpha} \le \Phi_{1,\alpha}, \\
\norm{D^2\phi}_{\sup} &\le \Phi_2, \qquad \norm{D^2\phi}_{C^\alpha} \le \Phi_{2,\alpha}, \\
\norm{D\phi\cdot\phi}_{\sup} &\le \Phi_{DD}\,.
\end{split}
\end{align}
If $\alpha =1$, we use the notation $\Phi_\alpha = \Phi_{0,1}$ to avoid confusion with $\Phi_1$.

\begin{remark}
Since $\norm{D\phi\cdot\phi}_{\sup} \leq \Phi_0\Phi_1$, we can (and will) always assume $\Phi_{DD}\le \Phi_0\Phi_1$.
\end{remark}

\noindent Then one can show the following estimate:

\begin{proposition}[A priori estimate, cf.\ {\cite[Remark 3]{MR2387018}}]
\label{prop:Young-apriori-RDE}
Let $r\in[2,3),\alpha\in(0,1]$ with $r<(2+\alpha)$ and $\bfX$ be an $r$-rough path and set $\theta = (2+\alpha)/r$. Recall the sewing constant $\Csew = \zeta(\theta)2^{2\theta-2}$ from Remark \ref{rem:integral-bound} and
let $s \leq s'$ with
\begin{equation}
\label{eq:10}
\Csew \Phi_{1,\alpha}\Phi_0^\alpha V^r X_{[s,s')}^{\alpha}V^rX_{(s,s')} < \frac 1{8\cdot 3^\alpha}
\end{equation}
\begin{equation}
\label{eq:11}
\Csew \Phi_{1,\alpha} \Phi_{DD}^\alpha V^{r/2}\bbX_{[s,s')}^\alpha V^r X_{(s,s')} < \frac 1{8(8-4\alpha)^\alpha}\,.
\end{equation}
\begin{equation}
\label{eq:12}
\Csew (\Phi_{1}\Phi_{\alpha}+\Phi_{1,\alpha}\Phi_{0})
V^{r/2}\bbX_{(s,s')}^{\alpha}
\leq (4\alpha)^{-\alpha}
\Phi_{DD}^{1-\alpha}
\end{equation}
\begin{equation}
\label{eq:13}
\Csew \Phi_{1} V^r X_{(s,s')} \leq 1/4,
\end{equation}
Suppose that $\bfY = (Y,Y')$ with $Y \in V^{r}([s,s'])$ and $Y'=\phi(Y)$ satisfies
\[
Y_{t} = Y_{s} + \int_{s}^{t} \phi(\bfY) \dif \bfX,
\quad
t \in [s,s'],
\]
where $\phi\in C^{1,\alpha}_b$ with bounds $\Phi$.
Then, we have the following for each $s\le t < s'$:
\begin{align}
\label{eq:lvl2-apriori:variation}
\frac{1}{4} V^rY_{[s,t]}
&\leq
\Phi_{0} V^r X_{[s,t]}
+
2 \Phi_{DD}V^{r/2}\bbX_{[s,t]}\\
\frac 34 V^{r/2}(R^{\bfY,\bfX})_{[s,t]} &\le \frac 32\Phi_0V^r X_{[s,t]} + 3(2-\alpha)\Phi_{DD}V^{r/2}\bbX_{[s,t]}\,.\label{eq:lvl2-apriori:variation2}
\end{align}
\end{proposition}


\begin{remark}
Note that \eqref{eq:lvl2-apriori:variation} and \eqref{eq:lvl2-apriori:variation2} also hold for $t= s'$, as long as we have \eqref{eq:10} to \eqref{eq:13} with closed right intervals $V^r X_{[s,s']}, V^r X_{(s,s']}, V^{r/2}\bbX_{[s,s']}, V^{r/2}\bbX_{(s,s']}$.
\end{remark}


\begin{proof}
We start with some estimates that are valid for arbitrary time intervals.
For all $(s,t)\in\Delta_T$, let
\begin{equation}
\label{eq:tildeR}
\tilde{R}^{\bfY,\bfX}_{s,t} := R^{\bfY,\bfX}_{s,t} - \phi(\bfY)_s'\bbX_{s,t} = R^{\bfY,\bfX}_{s,t} - D\phi(Y_s)Y'_s\bbX_{s,t}.
\end{equation}
Since $Y'_{t} = \phi(Y_t)$, for any time interval $I$, we have
\begin{equation}
\label{eq:Y'-p/alpha}
V^{r/\alpha}(Y')_{I}
\leq
\norm{\phi}_{C^{\alpha}} V^rY_{I}^{\alpha},
\end{equation}
\begin{equation}
\label{eq:Y'-sup}
\norm{Y'}_{\sup} \leq \norm{\phi}_{\sup},
\end{equation}
\begin{equation}
\label{eq:Y-Vp-by-tildeR}
\begin{split}
V^r Y_{I}
&\leq
V^r(R^{\bfY,\bfX})_{I} + \norm{Y'}_{\sup} V^r X_{I}
\\ &\leq
V^{r}(\tilde{R}^{\bfY,\bfX})_{I} + \Phi_{DD} V^r\bbX_{I} + \Phi_{0}V^r X_{I}.
\end{split}
\end{equation}
For any $s\leq t$, we have
\begin{equation}
\label{eq:7RDE}
\begin{split}
\MoveEqLeft
\Csew^{-1} V^{r/(2+\alpha)}(R^{\bfY,\bfX} - \phi(\bfY)'\bbX)_{[s,t]}
\\ \text{(Remark~\ref{rem:integral-bound})} &\leq
V^{r/(1+\alpha)}(R^{\phi(\bfY),\bfX}_{s,t})_{[s,t)} V^r X_{(s,t]}
+ V^{r/\alpha}(\phi(\bfY)')_{[s,t)} V^{r/2}\bbX_{(s,t]}
\\ \text{(Lemma~\ref{lem:controlled-composition})} &\leq
(\Phi_{1} V^{r/(1+\alpha)}(R^{\bfY,\bfX})_{[s,t)} + \frac{\Phi_{1,\alpha}}{1+\alpha} V^r Y_{[s,t)}^{1+\alpha})
V^r X_{(s,t]}
\\ &+
(\Phi_{1}V^{r/\alpha}(Y')_{[s,t)} +  \Phi_{1,\alpha} V^rY_{[s,t)}^{\alpha} \norm{Y'}_{\sup})V^{r/2}\bbX_{(s,t]}
\\ \text{(\eqref{eq:Y'-p/alpha}, \eqref{eq:Y'-sup})} &\leq
(\Phi_{1} V^{r/(1+\alpha)}(R^{\bfY,\bfX})_{[s,t)} + \frac{\Phi_{1,\alpha}}{1+\alpha}V^r Y_{[s,t)}^{1+\alpha})
V^r X_{(s,t]}
\\ &+
(\Phi_{1}\Phi_{\alpha}+\Phi_{1,\alpha}\Phi_{0})V^rY_{[s,t)}^{\alpha}V^{r/2}\bbX_{(s,t]}.
\end{split}
\end{equation}
By a time change, we may assume that $X,Y$ are cadlag.
For $t \in [s,s']$, we consider the condition
\begin{equation}
\label{eq:Y-small}
\Csew \Phi_{1,\alpha} V^r Y_{[s,t)}^{\alpha} V^r X_{(s,s']} \leq 1/4.
\end{equation}
The set
\begin{equation*}
\calT :=
\Set{ t\in [s,s'] \given \text{\eqref{eq:Y-small} holds}}
\end{equation*}
is a closed interval that contains the point $s$.
We will later show that $\calT$ is also an open subset of $[s,s']$, implying that $\calT = [s,s']$ and thus \eqref{eq:Y-small} holds for $t = s'$.
For the time being, fix any $t \in \calT \setminus \Set{s'}$.

Using \eqref{eq:13} and \eqref{eq:12} in the last line of \eqref{eq:7RDE}, we obtain
\begin{align*}
V^{r/2} (R^{\bfY,\bfX})_{[s,t]}
&\leq
V^{r/(2+\alpha)}(\tilde{R}^{\bfY,\bfX})_{[s,t]} + \Phi_{DD} V^{r/2}\bbX_{[s,t]}
\\ &\leq  \Csew\Phi_1V^r X_{(s,t]}V^{r/(1+\alpha)}(R^{\bfY,\bfX})_{[s,t)} + \Csew \frac{\Phi_{1,\alpha}}{1+\alpha}V^r Y_{[s,t)}^{1+\alpha}V^r X_{(s,t]} 
\\ &\qquad+\Csew (\Phi_1\Phi_\alpha + \Phi_{1,\alpha}\Phi_0)V^r Y_{[s,t)}^\alpha V^{r/2}\bbX_{(s,t]}+ \Phi_{DD}V^{r/2}\bbX_{[s,t]}
\\ &\leq
\frac{1}{4} V^{r/(1+\alpha)}(R^{\bfY,\bfX})_{[s,t)}
+
\Csew \frac{\Phi_{1,\alpha}}{1+\alpha}V^r Y_{[s,t)}^{1+\alpha} V^r X_{(s,t]}
\\ &\qquad+
(4\alpha)^{-\alpha}\Phi_{DD}^{1-\alpha} V^r Y_{[s,t)}^{\alpha} V^{r/2}\bbX_{(s,t]}^{1-\alpha}
+ \Phi_{DD} V^{r/2}\bbX_{[s,t]}.
\end{align*}
This implies
\begin{multline}
\label{eq:8}
\frac{3}{4} V^{r/2}(R^{\bfY,\bfX})_{[s,t]}
\leq
\Csew \frac{\Phi_{1,\alpha}}{1+\alpha}V^r Y_{[s,t)}^{1+\alpha} V^r X_{(s,t]}
\\ +
(4\alpha)^{-\alpha} V^r Y_{[s,t)}^{\alpha} \Phi_{DD}^{1-\alpha}V^{r/2}\bbX_{(s,t]}^{1-\alpha}
+ \Phi_{DD} V^{r/2}\bbX_{[s,t]}.
\end{multline}
By \eqref{eq:Y-small}, this implies
\begin{multline}
\label{eq:5RDE}
\frac{3}{4} V^{r/2}(R^{\bfY,\bfX})_{[s,t]}
\leq
\frac1{4(1+\alpha)} V^r Y_{[s,t)}
+
(4\alpha)^{-\alpha} V^r Y_{[s,t)}^{\alpha} \Phi_{DD}^{1-\alpha}V^{r/2}\bbX_{(s,t]}^{1-\alpha}
+ \Phi_{DD} V^{r/2}\bbX_{[s,t]}.
\end{multline}
By the AMGM inequality, this implies
\begin{align}
\frac{3}{4} V^{r/2}(R^{\bfY,\bfX})_{[s,t]}
&\leq
\left(\frac{1}{4(1+\alpha)} + \frac 14\right) V^r Y_{[s,t)}
+
(2-\alpha) \Phi_{DD} V^{r/2}\bbX_{[s,t]} \nonumber\\
&\leq
\frac 12 V^r Y_{[s,t)}
+
(2-\alpha) \Phi_{DD} V^{r/2}\bbX_{[s,t]}.\label{ineq:estimateR}
\end{align}
Using \eqref{eq:Y-Vp-by-tildeR}, we obtain
\begin{align*}
\frac{3}{4} V^r Y_{[s,t]}
&\leq
\frac 34 V^{r/2}(R^{\bfY,\bfX})_{[s,t]} + \frac 34\Phi_0V^r X_{[s,t]}
\\ &\leq \frac{1}{2} V^r Y_{[s,t)}
+
\frac{3}{4} \Phi_{0} V^r X_{[s,t]}
+
(2-\alpha) \Phi_{DD} V^{r/2}\bbX_{[s,t]}.
\end{align*}
This implies
\begin{equation}
\label{eq:6RDE}
\frac{1}{4}V^rY_{[s,t]}
\leq
\frac 34 \Phi_{0} V^r X_{[s,t]}
+
(2-\alpha) \Phi_{DD} V^{r/2}\bbX_{[s,t]}
\end{equation}
Hence,
\begin{equation}
\label{eq:1}
\begin{split}
\MoveEqLeft
4^{-\alpha} \Csew \Phi_{1,\alpha} V^r Y_{[s,t]}^{\alpha} V^r X_{(s,t]}
\\ \text{(\eqref{eq:6RDE})} &\leq
\Csew \Phi_{1,\alpha} \left(\frac 34\Phi_{0}V^r X_{[s,t]}\right)^{\alpha} V^r X_{(s,t]}
\\ & \quad+
\Csew \Phi_{1,\alpha} ((2-\alpha) \Phi_{DD}V^{r/2}\bbX_{[s,t]})^{\alpha} V^rX_{(s,t]}
\\ \text{(\eqref{eq:10}, \eqref{eq:11})} &<
4^{-\alpha-1}.
\end{split}
\end{equation}
By cadlag assumption, this implies that \eqref{eq:Y-small} continues to hold with $t$ replaced by a slightly larger time.
Therefore, $\calT = [s,s']$.
In particular, \eqref{eq:6RDE} holds for every $t \in [s,s')$, giving us \eqref{eq:lvl2-apriori:variation} as a direct consequence. Putting \eqref{eq:6RDE} into \eqref{ineq:estimateR} gives \eqref{eq:lvl2-apriori:variation2}.
\end{proof}

\subsection{Solution space}

Let us now clarify the space of solutions $\calY(\bfX)$, similar to \eqref{eq:3}. We require bounds on $V^r Y_I, V^{r/\alpha} (Y')_I, V^{r/2} (R^{\bfY,\bfX})_I$ for all intervals $I\subset[0,T]$. Note that every solution $\bfY$ of \eqref{eq:RDESec6} is of the form $Y'_t = \phi(Y_t)$, giving us a bound on $\norm{Y'}_{\sup}$ for free. The precise solution space reads as follows.
\begin{lemma}[Solution space]
Let $\phi\in C^{1,\alpha}_b$ and set for all intervals $I\subset[0,T]$
\[
\omega_I = \Phi_0 V^r X_{I} + \Phi_0\Phi_1V^{r/2}\bbX_{I}\,.
\]
Let $c>2$ and recall $\Csew = 2^{2\theta-2}\zeta(\theta)$ with $\theta = (2+\alpha)/r$. Let $T$ be small enough such that
\begin{equation}\label{ineq:assumptionOmega}
\Csew \left(\frac{\Phi_1}{\Phi_0}(c-1)\omega_{[0,T]} + \frac{\Phi_{1,\alpha}}{\Phi_0(1+\alpha)}c^{1+\alpha}\omega_{[0,T]}^{1+\alpha} +\frac{\Phi_1\Phi_\alpha+\Phi_{1,\alpha}\Phi_0}{\Phi_0\Phi_1}c^\alpha\omega_{[0,T]}^\alpha\right) \le c-2\,.
\end{equation}
Then, the set of $\bfX$-controlled paths
\begin{align}
\begin{split}
\label{eq:RDE-solution-space}
\calY(\bfX) :=
\bigg\{ \bfY \given &
 V^r Y_{I} \le c\omega_I, V^{r/2}(R^{\bfY,\bfX})_{I} \le (c-1)\omega_I,
V^{r/\alpha}(Y')_{I} \le c^\alpha\Phi_{\alpha}\omega_{I}^{\alpha},\\&
\norm{Y'}_{\sup} \leq \Phi_{0}\text{ for all intervals }I\subset[0,T]\bigg\}
\end{split}
\end{align}
is invariant under the mapping
\begin{equation}
\label{eq:RDE-iteration}
\operatorname{Step} : \bfY \mapsto \left(y_{0} + \int_{0}^{\cdot} \phi(\bfY) \dif \bfX, \phi(Y)\right)
\end{equation}
for any $y_{0} \in \bbR^{d}$.
\end{lemma}

\noindent While it is convenient to add the first condition in \eqref{eq:RDE-solution-space} ($V^r Y_I \le c\omega_I$) to the definition of $\calY(\bfX)$, it is not necessary as
\[
V^r Y_I \le V^r(R^{\bfY,\bfX})_I + \norm{Y'}_{\sup}V^r(X)_{I} \le c\omega_I\,,
\]
already implies it.

\begin{proof}
Suppose $\bfY \in \calY$. We call $\operatorname{Step}(\bfY) = \bfZ$. Then $Z_{s,t} = \int_s^t \phi(\bfY_r)\dif\bfX_r$, $Z' = \phi(Y)$. 
Direct estimates show
\begin{align*}
V^{r/\alpha}(Z')_{I} &= V^{r/\alpha}(\phi(Y))_{I}\le \norm{\phi}_{C^\alpha}V^r Y_{I}^\alpha\le \Phi_\alpha (c\omega_I)^\alpha\\
\norm{Z'}_{\sup} &= \norm{\phi(Y)}_{\sup}\le \Phi_0\,.
\end{align*}
By Lemma~\ref{lem:controlled-composition}, we have
\begin{align}
\label{eq:3-1}
V^{r/\alpha}(\phi(\bfY)')_{I} \leq
\Phi_{1} \Phi_{\alpha} (c\omega_{I})^{\alpha} + \Phi_{1,\alpha} \Phi_{0} (c\omega_{I})^{\alpha},
\\
V^{r/(1+\alpha)}(R^{\phi(\bfY),\bfX})_{I} \leq \Phi_{1} (c-1)\omega_{I} + \frac{1}{1+\alpha}\Phi_{1,\alpha} (c\omega_{I})^{1+\alpha}\,.\label{eq:3-2}
\end{align}
\[
\norm{\phi(\bfY)'}_{\sup}
\leq
\Phi_{1} \norm{Y'}_{\sup}
\leq
\Phi_{1} \Phi_{0}.
\]
By Remark~\ref{rem:integral-bound}, we obtain
\begin{align*}
V^{r/2}(R^{\int_{0}^{\cdot}\phi(\bfY)\dif\bfX, \bfX})_{[s,t]}
\leq&
\Csew V^{r/(1+\alpha)}(R^{\phi(\bfY),\bfX})_{[s,t)} V^r X_{(s,t]}
\\&+
\Csew V^{r/\alpha}(\phi(\bfY)')_{[s,t)} V^{r/2}\bbX_{(s,t]}
\\&+
\norm{\phi(\bfY)'}_{\sup} V^{r/2}\bbX_{[s,t]}
\\ \leq&
\Csew (\Phi_{1} (c-1)\omega_{[s,t)} + \frac{\Phi_{1,\alpha}}{1+\alpha} c^{1+\alpha}\omega_{[s,t)}^{1+\alpha}) V^r X_{(s,t]}
\\&+
\Csew (\Phi_{1} \Phi_{\alpha} + \Phi_{1,\alpha} \Phi_{0}) c^\alpha\omega_{[s,t)}^{\alpha} V^{r/2}\bbX_{(s,t]}
+ \Phi_{0}\Phi_{1} V^{r/2}\bbX_{[s,t]}
\\ \le & \Csew\bigg(\frac{\Phi_1}{\Phi_0} (c-1)\omega_{[s,t]} + \frac{\Phi_{1,\alpha}}{\Phi_0(1+\alpha)}c^{1+\alpha}\omega_{[s,t]}^{1+\alpha} \\&\qquad+  \frac{\Phi_1\Phi_\alpha+\Phi_{1,\alpha}\Phi_0}{\Phi_0\Phi_1}c^\alpha\omega^\alpha_{[s,t]}\bigg)\omega_{[s,t]}+\omega_{[s,t]}\\
\eqref{ineq:assumptionOmega}\le & (c-1)\omega_{[s,t]}
\end{align*}
By Lemma~\ref{lem:controlled-implicit}, this trivially implies
\[
V^r\left(\int_{0}^{\cdot}\phi(\bfY)\dif\bfX\right)_{I}
\leq
\norm{\phi(Y)}_{\sup}V^r X_{I} + V^{r/2}(R^{\bfZ,\bfX})_{I} \le \Phi_0 V^rX_{I} +(c-1)\omega_I \le c\omega_I\,,
\]
showing the claim.
\end{proof}

\begin{remark}
For $c\ge 9$, the a priori estimates given in Proposition~\ref{prop:Young-apriori-RDE} already gives us that all solutions $Y$ of \eqref{eq:RDESec6} need to be in $\calY(\bfX)$.
\end{remark}

\subsection{Contraction of fixed point map}

\noindent It is now straight-forward to show that $\mbox{Step}$ from \eqref{eq:RDE-iteration} is a contraction on $\calY(\bfX)$. Recall that for $\alpha = 1$, we use the notation $\Phi_{0,1} := \Phi_\alpha$ and set
\begin{equation}\label{eq:BoundA}
A := \max(1,\Phi_{0,1}) c\omega_{[0,T]}\,.
\end{equation}
This especially implies that $V^r Y\le A, V^{r/2} R^{\bfY,\bfX}\le A$ and $V^r Y'\le A$ for all $Y\in\cal Y(\bfX)$. Then, one can show the following:

\begin{lemma}[Contractivity of the iteration]
\label{lem:RDE-iteration-contractive}
For every $\phi \in C^{2,1}_{b}$, there exists $\epsilon>0$ such that, for every rough path $\bfX$ with $V^r X < \epsilon$ and $V^{r/2}\bbX < \epsilon$,
the map \eqref{eq:RDE-iteration} is strictly contractive on the set
\begin{equation}
\label{eq:66}
\Set{Y\in\calY(\bfX) \given Y_{0}=y_{0}, Y'_{0}=\phi(y_{0})}
\end{equation}
with respect to the metric
\begin{equation}
\label{eq:RDE-contraction-metric}
d(\bfY,\tilde{\bfY}) = \max( V^{r/2}(\Delta R^{\bfY,\bfX}) , V^r(\Delta Y') , CV^r(\Delta Y))\,,
\end{equation}
where $C = 2(\Phi_{1}+A\Phi_{1,1})$.
\end{lemma}
\begin{proof}
Suppose $\bfY,\tilde{\bfY} \in \calY(\bfX)$ with $d(\bfY,\tilde{\bfY}) = \alpha$ and $Y_{0}=\tY_{0}=y_{0}$ as well as $Y_0' = \tilde Y_0' = \phi(y_0)$.
Then,
\[
\norm{\Delta Y}_{\sup} \leq V^r (\Delta Y) \le \frac\alpha C.
\]
By Lemma~\ref{lem:controlled-composition}, we obtain
\begin{align*}
V^r(\Delta\phi(Y))
&\leq \norm{D\phi}_{\sup} V^r(\Delta Y) + \norm{D\phi}_{\Lip} \norm{\Delta Y}_{\sup} V^r\tY\\
&\le \frac \alpha C(\Phi_{1} + A\Phi_{1,1})
\leq \alpha/2.
\end{align*}
\begin{align*}
V^r(\Delta\phi(\bfY)') &\le \Phi_{1}V^r(\Delta Y') +\Phi_{1,1} V^r(Y) \norm{\Delta Y'}_{\sup} + \Phi_{1,1} \norm{\Delta Y}_{\sup}V^r\tY' \\
&\qquad+\Phi_{2} V^r(\Delta Y) \norm{\tilde Y'}_{\sup} + \Phi_{2,1} \norm{\Delta Y}_{\sup}V^r\tY\norm{\tilde Y'}_{\sup} \\
&\le \alpha C_1\,,
\end{align*}
for $C_1 = \Phi_{1} + \Phi_{1,1} A + \frac{\Phi_{1,1} A} C + \frac{\Phi_{2}\Phi_0} C + \frac{\Phi_{2,1}\Phi_{0}A} C$. Furthermore, we have
\begin{align*}
V^{r/2}(\Delta R^{\phi(\bfY),\bfX}) &\le \Phi_{1}V^{r/2}(\Delta R^{\bfY,\bfX}) + \Phi_{1,1}\norm{\Delta Y}_{\sup} V^{r/2}(R^{\tilde\bfY,\tilde\bfX}) \\&\quad+ \frac 12\Phi_{2} (V^rY+V^r\tY)V^r(\Delta Y) + \frac 12\Phi_{2,1}\norm{\Delta Y}_{\sup}V^r\tY^2 \\
&\le \alpha C_2\,,
\end{align*}
where $C_2 = \Phi_{1} + \frac{\Phi_{1,1} A} C + \frac{\Phi_{2} A} C + \frac{\Phi_{2,1} A^2}{2C}$.

Let $\bfZ = \operatorname{Step}(\bfY)$, $\tilde\bfZ = \operatorname{Step}(\tilde\bfY)$.
Then
\[
V^r(\Delta Z') = V^r(\Delta\phi(Y))
\leq \alpha/2.
\]
By Lemma~\ref{lem:integral-bound}, where we use $\norm{\Delta\phi(\bfY)'}_{\sup} \le V^r(\Delta\phi(\bfY)')$:
\begin{align*}
V^{r/2}(\Delta R^{\bfZ,\bfX}) &\le \norm{\Delta \phi(\bfY)'}_{\sup}V^{r/2}\bbX \\&\quad+ \zeta(3/r) 4^{3/r-1}\left(V^r(\Delta \phi(\bfY)') V^{r/2}\bbX + V^{r/2}(\Delta R^{\phi(\bfY),\bfX})V^r X\right)\\
&\le \alpha\epsilon C_3\,.
\end{align*}
For $C_3 = \left(C_1 + \zeta(3/r) 4^{3/r-1}(C_1 + C_2)\right)$. By Lemma~\ref{lem:controlled-implicit}, this implies
\[
V^r(\Delta Z)
\le
\norm{\Delta Z'}_{\sup} V^rX
+ V^{r/2}(\Delta R^{\bfZ,\bfX})
\leq
V^r(\Delta Z') \epsilon
+ C_3 \epsilon \alpha
\le
\epsilon \alpha\left(\frac 12+C_3\right).
\]
Choosing $\epsilon$ small enough, we obtain the claim.
\end{proof}

\noindent This directly leads to the following local existence and uniqueness result:

\begin{theorem}[Existence of solutions]
\label{thm:RDE-solution-local-uniqueness}
For every $\phi \in C^{2,1}_{b}$, there exist $\epsilon>0$ such that, for every rough path $\bfX$ with $V^r X < \epsilon$ and $V^{r/2}\bbX < \epsilon$ and every $y_{0}\in\bbR^d$, there exists a unique $\bfX$-controlled path $\bfY \in \calY(\bfX)$ such that, for every $t\in [0,T]$, we have
\begin{equation}
\label{eq:RDE}
\bfY_{t} = y_{0} + \int_{0}^{t} \phi(\bfY) \dif\bfX,
\quad
Y'_{t} = \phi(Y_{t}).
\end{equation}
\end{theorem}
\begin{proof}
Let $A,\epsilon$ be as in Lemma~\ref{lem:RDE-iteration-contractive}.
Define a sequence of controlled paths by
\[
(\bfY_{0})_{t}=y_{0}, \quad (\bfY_{0})'_{t}=0,
\quad
(\bfY_{j+1}) = \operatorname{Step}(\bfY_{j}).
\]
For $j\geq 1$, the paths $(\bfY_{j})$ are elements of \eqref{eq:66}.
It follows from Lemma~\ref{lem:RDE-iteration-contractive} that this sequence is Cauchy with respect to the metric \eqref{eq:RDE-contraction-metric}, and its limit solves the RDE \eqref{eq:RDE}.
The uniqueness of the solution follows from the strict contractivity of \eqref{eq:RDE-iteration}.
\end{proof}

\noindent Since our a priori estimates already imply that any solution lies in $\calY(\bfX)$ given a high enough constant $c>0$, this implies that the solution is unique among all $\bfX$-controlled rough paths over small enough intervals. By subdividing a larger interval into smaller intervals, one can also extend this result to global solutions.

Large jumps of $\bfX$ have to be handled separately.
We omit the tedious details.

\subsection{Lipschitz dependence on data}

We can now tackle the main result of this section. Similar to \eqref{eq:BoundA}, one can find an $A>0$ such that
\[
V^rY,V^r\tilde Y,V^r(Y'),V^r(\tilde Y'),V^{r/2}(R^{\bfY,\bfX}), V^{r/2}(R^{\tilde\bfY,\tilde\bfX})\le A\,,
\]
holds for all $\bfY\in \calY(\bfX), \tilde\bfY\in\calY(\tilde\bfX)$. Then, one can show the following:

\begin{theorem}[Local Lipschitz dependence of RDE solution on the data]\label{thm:LipschitzDependenceRDE}
Let $\phi\in C^{2,1}_{b}$. There exists $\epsilon>0$ such that if $\bfX,\tilde\bfX$ are rough paths with norms $<\epsilon$,  $\bfY \in \calY(\bfX), \tilde\bfY\in\calY(\tilde\bfX)$ are solutions to \eqref{eq:RDE} with initial datum $y_0,\tilde y_0$, we have for all intervals $I\subset[0,T]$:
\begin{align}
V^r(\Delta Y)_I &\le  c_1\bigg((V^rX_I + V^{r/2}\bbX_{I} + V^{r/2}\tilde\bbX_{I}) \gamma + V^r(\Delta X)_I + V^{r/2}(\Delta \bbX)_{I}\bigg) \\
V^{r/2}(\Delta R^{\bfY,\bfX})_{I}&\le c_2\bigg((V^rX_I^2 + V^rX_I \cdot V^r\tilde X_I +V^{r/2}\bbX_{I} + V^{r/2}\tilde \bbX_{I})\gamma \\
&\qquad + (V^r X_I+V^r\tilde X_I +V^{r/2}\bbX_I + V^{r/2}\tilde\bbX_{I})V^r(\Delta X)_I + V^{r/2}(\Delta\bbX)_{I}\bigg)\nonumber
\end{align}
where $ c_1, c_2$ denote some constants depending on $\phi,A,\epsilon,r$, and
\[
	\gamma = \abs{\Delta y_0} + V^r{\Delta X}_{[0,T]} + V^{r/2}{\Delta\bbX}_{[0,T]}\,.
\]
\end{theorem}

\begin{proof}
Note that it suffices to show that for intervals $I = [s,t]\subset[0,T]$, we have
\begin{align}
V^r(\Delta Y)_I &\le \tilde c_1\bigg((V^rX_I + V^{r/2}\bbX_{I} + V^{r/2}\tilde\bbX_{I}) \abs{\Delta y_s} + V^r(\Delta X)_I + V^{r/2}(\Delta \bbX)_{I}\bigg) \label{ineq:LipschitzRDE3}\\
V^{r/2}(\Delta R^{\bfY,\bfX})_{I}&\le \tilde c_2\bigg((V^rX_I^2 + V^rX_I \cdot V^r\tilde X_I +V^{r/2}\bbX_{I} + V^{r/2}\tilde \bbX_{I})\abs{\Delta y_s} \label{ineq:LipschitzRDE4}\\
&\qquad + (V^r X_I+V^r\tilde X_I +V^{r/2}\bbX_I + V^{r/2}\tilde\bbX_{I})V^r(\Delta X)_I + V^{r/2}(\Delta\bbX)_{I}\bigg)\,,\nonumber
\end{align}
as this especially implies that
\[
	V^r(\Delta Y)_{[0,t]} \le c\gamma \qquad \text{for all }t\in[0,T]\,.
\]
Here, $c,\tilde c_1,\tilde c_2$ denote constants depending only on $\phi,A,\epsilon$ and $r$. It then follows that
\[
	\abs{\Delta y_s} \le \abs{\Delta y_0} + V^r(\Delta Y)_{[0,s]} \le (c+1)\gamma\,.
\]
Putting this into \eqref{ineq:LipschitzRDE3}, \eqref{ineq:LipschitzRDE4} shows the claim.

Thus, the rest of this proof is dedicated to showing \eqref{ineq:LipschitzRDE3}, \eqref{ineq:LipschitzRDE4}. Let $I= [s,t]$ for some $0\le s\le t\le T$. 
By using $Y' = \phi(Y)$, we can show some simple estimates:
\begin{align*}
\norm{\Delta Y}_{\sup,I}
&\leq \abs{\Delta y_s} + V^r(\Delta Y)_I,
\\
\norm{\Delta Y'}_{\sup,I}
&\leq \norm{\phi}_{\Lip} \norm{\Delta Y}_{\sup,I}
\le \Phi_{0,1}(\abs{\Delta y_s}+V^r(\Delta Y)_I)\\
V^r(Y')_I &\le \Phi_{0,1} V^rY_I\,,
\end{align*}
where $\norm{\Delta Y}_{\sup,I} = \sup_{u\in I}\abs{\Delta Y_u}$. By Lemma \ref{lem:controlled-composition}, we have
\begin{align*}
V^r(\Delta Y')_I &\le \Phi_{1}V^r(\Delta Y)_I + \Phi_{1,1}V^r\tilde Y_I \norm{\Delta Y}_{\sup,I} \\
&\le (\Phi_{1} + \Phi_{1,1} A) V^r(\Delta Y)_I + \Phi_{1,1}V^r\tilde Y_I \abs{\Delta y_s}\,.
\end{align*}
Further, the Lemma gives us
\begin{align*}
V^r(\Delta \phi(\bfY)')_I &\le \Phi_{1}V^r(\Delta Y')_I +\Phi_{1,1}V^r Y_I\norm{\Delta Y'}_{\sup,I} + \Phi_{1,1}V^r\tilde Y_I'\norm{\Delta Y}_{\sup,I} \\ 
&\quad + \Phi_{2}\Phi_0 V^r(\Delta Y)_I + \Phi_{2,1}\Phi_0 \norm{\Delta Y}_{\sup,I} V^r\tilde Y_I \\
&\le \tilde c_1V^r(\Delta Y)_I + \tilde c_2(V^r Y_I + V^r\tilde Y_I)\abs{\Delta y_s},
\end{align*}
where $\tilde c_1 = \Phi_{1}^2+\Phi_{1}\Phi_{1,1} A + \Phi_{1,1}(\Phi_{0,1}+1)A + \Phi_{2}\Phi_0 + \Phi_{2,1}\Phi_0 A,$\\$\tilde c_2 = (\Phi_{1}+\Phi_{0,1}+1)\Phi_{1,1}+\Phi_{2,1}\Phi_0$.  We use Lemma \ref{lem:controlled-composition} once more to calculate
\begin{align*}
V^{r/2}(\Delta R^{\phi(\bfY),\bfX})_{I} &\le \Phi_{1}V^{r/2}(\Delta R^{\bfY,\bfX})_{I}+\Phi_{1,1}(\abs{\Delta y_s}+V^r(\Delta Y)_I)V^{r/2}(R^{\tilde \bfY,\tilde\bfX})_I \\&\quad+ \frac 12 \Phi_{2}(V^rY_I+V^r \tilde Y_I)V^r(\Delta Y)_I + \frac 12\Phi_{2,1}(\abs{\Delta y_0}V^r\tilde Y_I^2 + AV^r(\Delta Y)_IV^r\tilde Y_I)\\
&\le \Phi_{1}V^{r/2}(\Delta R^{\bfY,\bfX})_I +\tilde c_3 (V^r Y_I+V^r\tilde Y_I+V^{r/2}(R^{\tilde\bfY,\tilde\bfX})_I)V^r(\Delta Y)_I \\ &\quad+ \tilde c_4(V^r\tilde Y_I^2 + V^{r/2}(R^{\tilde\bfY,\tilde\bfX})_{I})\abs{\Delta y_s}\,,
\end{align*}
where $\tilde c_3 = \max(\Phi_{1,1},\frac 12(\Phi_{2}+ \Phi_{2,1} A))$ and $\tilde c_4 = \max(\Phi_{1,1}, \frac 12\Phi_{2,1})$. By using $\phi(\bfY)' = D\phi(Y) \phi(Y)$, it is easy to see that
\begin{equation*}
\norm{\Delta\phi(\bfY)'}_{\sup,I} \le (\Phi_{1,1}\Phi_0 + \Phi_{1}\Phi_{0,1})\norm{\Delta Y}_{\sup,I} \le (\Phi_{1,1}\Phi_0 + \Phi_{1}\Phi_{0,1})(\abs{\Delta y_s} +V^r(\Delta Y)_I)\,.
\end{equation*}
It then follows from Lemma \ref{lem:integral-bound}, where we use $\tilde c_5 = \zeta(3/r) 4^{3/r-1}$ and $V^r(\phi(\tilde \bfY)')\le (\Phi_{1}+\Phi_0\Phi_{1,1})A$ as well as $V^{r/2}(R^{\phi(\tilde \bfY),\tilde\bfX})_{I}\le \Phi_{1}V^{r/2}(R^{\tilde\bfY,\tilde\bfX})_{I}+\frac 12\Phi_{1,1}V^r \tilde Y_I^2$ by Lemma \ref{lem:controlled-composition}:
\begin{align}
V^{r/2}(\Delta R^{\bfY,\bfX})_I &\le \norm{\Delta \phi(\bfY)'}_{\sup,I}V^{r/2}\bbX_{I} + \Phi_{DD}V^{r/2}(\Delta\bbX)_I \nonumber\\&\qquad +\tilde c_5\bigg(V^r(\Delta\phi(\bfY)')_I V^{r/2}\bbX_{I} + V^{r/2}(\Delta R^{\phi(\bfY),\bfX})_I V^rX_I \nonumber\\ &\qquad+ V^r(\phi(\tilde \bfY)')_IV^{r/2}(\Delta\bbX)_I + V^{r/2}(R^{\phi(\tilde \bfY),\tilde\bfX})_{I}V^r(\Delta X)_I\bigg)\nonumber\\
&\le \tilde c\bigg(\abs{\Delta y_s}(V^{r/2}\bbX_I + V^{r/2}\tilde\bbX_I + V^r X_I\cdot V^r\tilde X_I) + V^r(\Delta Y)_I(V^{r/2}\bbX_{I} + V^rX_I) \nonumber\\
&\qquad+ V^{r/2}(\Delta R^{\bfY,\bfX})_{I}V^rX_I + V^{r/2}(\Delta\bbX)_{I} + V^r(\Delta X)_I(V^r(\tilde X)_I+V^{r/2}\tilde\bbX_I)\bigg)\,.\label{ineq:lipEstimateR}
\end{align}
Here, we used that $\bfY\in\calY(\bfX)$ implies $V^r Y_I+V^{r/2}(R^{\bfY,\bfX})_{I}\lesssim V^r X_I+V^{r/2}\bbX_I$ with similar bounds for $\tilde \bfY$ and $\tilde c$ is some constant bigger than $0$, depending on $\phi, A, \epsilon, r$. In the following, we allow $\tilde c$ to change for each line (although it will always denote some constant depending only on $\phi, A, \epsilon, r$). Then we can write
\begin{align*}
V^{r/2}(\Delta R^{\bfY,\bfX})_I&\le \tilde c\epsilon(V^{r/2}(\Delta R^{\bfY,\bfX})_I + V^r  \Delta Y_I) \\
&\quad+ \tilde c\bigg(\abs{\Delta y_s}(V^{r/2}\bbX_{I}+V^{r/2}\tilde \bbX_{I} +V^r X_I\cdot V^r\tilde X_I) + V^{r/2}(\Delta\bbX)_I+V^r(\Delta X)_I\bigg)\,.
\end{align*}
Moreover, Lemma \ref{lem:controlled-implicit} implies
\begin{align*}
V^r(\Delta Y)_I &\le \norm{\Delta Y'}_{\sup,I}V^rX_I +\norm{\tilde Y'}_{\sup,I}V^r(\Delta X)_I +V^{r/2}(\Delta R^{\bfY,\bfX})_{I} \\
&\le \tilde c\epsilon(V^{r/2}(\Delta R^{\bfY,\bfX})_I + V^r \Delta Y_I) \\
&\quad+ \tilde c\bigg(\abs{\Delta y_s}(V^{r/2}\bbX_I+V^{r/2}\tilde \bbX_I +V^r X_I) + V^{r/2}(\Delta\bbX)_I+V^r(\Delta X)_I\bigg)\,.
\end{align*}
Choosing $\epsilon$ small enough, such that $\tilde c\epsilon\le \frac 12$ then gives
\begin{align*}
V^{r/2}(\Delta R^{\bfY,\bfX})_{I}& +V^r(\Delta Y)_I \\ &\le \tilde c\bigg(\abs{\Delta y_s}(V^{r/2}\bbX_I+V^{r/2}\tilde \bbX_I +V^r X_I) + V^{r/2}(\Delta\bbX)_{I}+V^r(\Delta X)_{I}\bigg)\,,
\end{align*}
especially implying \eqref{ineq:LipschitzRDE3}. Putting \eqref{ineq:LipschitzRDE3} into \eqref{ineq:lipEstimateR} gives
\begin{align*}
V^{r/2}(\Delta R^{\bfY,\bfX})_{I} &\le \tilde c\epsilon V^{r/2}(\Delta R^{\bfX,\bfY})_I + \tilde c\bigg(\abs{\Delta y_s} (V^{r/2}\bbX_I+V^{r/2}\tilde\bbX_I + V^r X_I V^r \tilde X_I + V^rX_I^2) \\ 
&\quad + V^{r/2}(\Delta \bbX)_I + V^r(\Delta X)_I(V^{r/2}\bbX_I+V^{r/2}\tilde\bbX_I+V^r X_I+V^r\tilde X_I)\bigg)\,.
\end{align*}
Choosing $\epsilon>0$ such that $\epsilon\tilde c<\frac 12$ shows the claim.
\end{proof}

\section{Lipschitz estimates in the Besov setting}\label{sec:BesovEstimates}

As an application of our results, we show in this section how one can derive Besov estimates from variation estimates. To this end, we recall some results from \cite{FS22} and show that they can be seen as corollaries of their variation version. In particular, we recover key estimates for Young and rough path-integration and show how Lipschitz estimates for Young and rough differential equation follow from Theorem \ref{thm:Lipschitz-stability} and Theorem \ref{thm:LipschitzDependenceRDE}.

\subsection{Young integration}

Let $f\in B^{\alpha_0}_{p_0,q_0}, g\in B^{\alpha_1}_{p_1,q_1}$ for some parameters $0<\alpha_0,\alpha_1\le1$ and $0<p_0,p_1,q_0,q_1\le\infty$. We also assume a fixed time horizon $[0,T]$, $T>0$. Let
\begin{equation*}
\alpha = \alpha_0+\alpha_1,\qquad \frac 1p = \frac 1{p_0} + \frac 1{p_1},\qquad \frac 1q = \frac 1{q_0}+\frac 1{q_1}\,.
\end{equation*}
We can then recall Young integration in the Besov setting:
\begin{theorem}[\cite{FS22}, Theorem 4.1]\label{theo:YoungIntegralBesov}
Let $\alpha > 1, \alpha_0 > \frac 1{p_0}, \alpha_1 > \frac 1{p_1}$ and consider the two-parameter object
\[
	\Xi_{s,t} = f(s)(g(t)-g(s))\,.
\]
Then the Young integral
\begin{equation*}
\cI\Xi_{s,t} := \int_s^t f(z) \dif g(z)
\end{equation*}
exists and fulfills the following estimates:
\begin{align}
\norm{\cI\Xi-\Xi}_{B^{\alpha}_{p,q}} &\le c_1(\alpha,p,q) \norm{f}_{B^{\alpha_0}_{p_0,q_0}}\norm{g}_{B^{\alpha_1}_{p_1,q_1}} \label{ineq:Besov_Estimate_Young}\\
\norm{\cI\Xi}_{B^{\alpha_1}_{p,q_1}} &\le c_2(\alpha_0,\alpha_1,p_0,p_1,q_0,q_1)(\norm{f}_{L^{p_0}}+\norm{f}_{B^{\alpha_0}_{p_0,q_0}})\norm{g}_{B^{\alpha_1}_{p_1,q_1}}\,,\nonumber
\end{align}
for some constants $c_1(\alpha,p,q),c_2(\alpha_0,\alpha_1,p_0,p_1,q_0,q_1) > 0$.
\end{theorem}

\noindent Note that the assumption $\alpha_i > \frac 1{p_i}, i=0,1$ especially imply $p_0,p_1 > 1$.

\begin{remark}
Let us mention that directly constructing this integral in Besov spaces allows one to replace $\alpha_0 > \frac 1{p_0}, \alpha_1 > \frac 1{p_1}$ with the less restrictive condition $\alpha > \frac 1p$. Our approach requires the stricter condition as constructing the Young integral in the variation setting requires $f,g$ to have some finite $r_0,r_1$-variations, which puts individual requirements on $\alpha_0,p_0$ and $\alpha_1,p_1$.
\end{remark}

\noindent 
Let us make some observations before we tackle the proof. Note that $\norm{\Xi}_{B^{\alpha_1}_{p,q_1}}\le \norm{f}_{L^{p_0}}\norm{g}_{B^{\alpha_1}_{p_1,q_1}}$ holds, so the second inequality easily follows from the first one. Also, by the assumptions on $\alpha,\alpha_0,\alpha_1,p_0,p_1$ we can choose $r_0,r_1 \ge 1$ in such a way that for $i=0,1$
\begin{align*}
\frac 1{\alpha_i}&< r_i\le p_i \\
\frac 1{r_0}+\frac 1{r_1} &=\theta > 1\,.
\end{align*}
Thus, by Corollary \ref{cor:EmbeddingVrInBesov}, it follows that $f,g$ have unique continuous representations fulfilling
\[
	V^{r_0} f \lesssim \norm{f}_{B^{\alpha_0}_{p_0,q_0}},\qquad V^{r_1} f \lesssim \norm{f}_{B^{\alpha_1}_{p_1,q_1}}\,,
\]
showing especially that these variations are finite. We can thus prove our theorem by Young integration for variation norms:

\begin{proof}[Proof of Theorem \ref{theo:YoungIntegralBesov}]
Since $\theta>1$, Lemma \ref{lem:1} gives us the existence of the Young integral $\calI\Xi_{s,t} = \int_s^t f(z)\dif g(z)$, as well as the inequality \eqref{ineq:YoungEstimateRVar1}:
\[
	\abs{\calI\Xi_{s,t} -\Xi_{s,t}} \lesssim V^{r_0} f_{[s,t)} V^{r_1} g_{(s,t]}\,.
\]
Applying the $B^\alpha_{p,q}$ norm to both sides of the inequality gives
\begin{align*}
\norm{\cI\Xi-\Xi}_{B^{\alpha}_{p,q}} &\lesssim \norm{V^{r_0} f V^{r_1} g}_{B^{\alpha}_{p,q}} \\
&\le \norm{V^{r_0} f}_{B^{\alpha_0}_{p_0,q_0}} \norm{V^{r_1} g}_{B^{\alpha_1}_{p_1,q_1}}\\
&\lesssim \norm{f}_{B^{\alpha_0}_{p_0,q_0}}\norm{g}_{B^{\alpha_1}_{p_1,q_1}}\,,
\end{align*}
where we used Hölder's inequality as well as Theorem \ref{theo:VrBesovEstimate} (or more precisely Theorem \ref{theo:IntroVrBesovEstimate}).
\end{proof}

\subsection{Rough integration}

Let us move to the rough case, in which we show how rough path integration (\cite{FS22}, Theorem 5.4) can be derived from the variation case (Lemma \ref{lem:integral-bound}). To do this, we show that the integral exists as well as a stability estimate.

We fix Besov rough paths $\bfX,\tilde\bfX$ with controlled rough paths $\bfY,\tilde\bfY$ for parameters $\alpha>\frac 13, 0<q\le\infty, 2\le p\le \infty$ with  $\frac 1\alpha < p$. We can then chose $r\in[2,3)$, such that
\[
	\frac 1\alpha < r\le p\,.
\]
By Lemma \ref{lem:BesovRPAreVariationRP} and \ref{lem:BesovControlledRPAreVariationControlledRP}, it follows that $\bfX,\tilde\bfX$ are $r$-rough paths with controlled $r$-rough paths $\bfY,\tilde\bfY$. Recall that the rough path integral of $\bfY$ against $\bfX$ is given by the sewing of
\[
	\Xi_{s,t} = Y_s X_{s,t} + Y'_s \bbX_{s,t}\,,
\]
which we denote by $Z_{s,t} = \cI\Xi_{s,t}$. We define $\tilde\Xi,\tilde Z$ analogously. With these notations, one can easily show the Besov rough integration:

\begin{proposition}[\cite{FS22}, Theorem 5.4]
Let $\alpha >\frac 13$ such that $0<\frac 1\alpha\le p\le\infty$, $2\le p\le \infty$, $0<q\le\infty$ and let $\bfX,\tilde \bfX\in B^{\alpha}_{p,q}$ be rough paths with controlled rough paths $\bfY,\tilde\bfY$ in the Besov setting. Then the rough path integrals $Z,\tilde Z$ exist and we have the stability estimates
\begin{align*}
\norm{\Delta Z-\Delta\Xi}_{B^{3\alpha}_{p/3,q/3}} \lesssim& \left(\norm{\Delta R^{\bfY,\bfX}}_{B^{2\alpha}_{p/2,q/2}} + \norm{\Delta Y'}_{B^{\alpha}_{p,q}} \norm{X}_{B^\alpha_{p,q}} + \norm{\tilde Y'}_{B^\alpha_{p,q}}\norm{\Delta X}_{B^\alpha_{p,q}}\right)\norm{X}_{B^\alpha_{p,q}} \\
&+\left(\norm{R^{\tilde\bfY,\tilde\bfX}}_{B^{2\alpha}_{p/2,q/2}}+ \norm{\tilde Y'}_{B^\alpha_{p,q}}\norm{\tilde X}_{B^\alpha_{p,q}}\right)\norm{\Delta X}_{B^\alpha_{p,q}}\\
&+\norm{\Delta Y'}_{B^\alpha_{p,q}}\left(\norm{\bbX}_{B^{2\alpha}_{p/2,q/2}} + \norm{X}_{B^\alpha_{p,q}}^2\right) \\
&+\norm{\tilde Y'}_{B^\alpha_{p,q}}\left(\norm{\Delta\bbX}_{B^{2\alpha}_{p/2,q/2}} + \norm{\Delta X}_{B^\alpha_{p,q}}(\norm{X}_{B^{\alpha}_{p,q}}+\norm{\tilde X}_{B^\alpha_{p,q}})\right)\,,
\end{align*}
where the constant in $\lesssim$ might depend on the parameters $r,\alpha, p,q$.
\end{proposition}

\begin{proof}
We choose $r\in[2,3)$, such that $\frac 1\alpha<r\le p$. Then $\bfX,\tilde\bfX$ are $r$-rough paths with controlled $r$-rough paths $\bfY,\tilde\bfY$. Thus, the existence of $Z,\tilde Z$ immediately follows from Lemma \ref{lem:integral-bound}. Applying the Besov norm $\norm{\cdot}_{B^{3\alpha}_{p/3,q/3}}$ to \eqref{ineq:stabilityRoughIntegral} gives
\begin{align*}
\norm{\Delta Z-\Delta\Xi}_{B^{3\alpha}_{p/3,q/3}} \lesssim& \norm{V^{r/2}\Delta R^{\bfY,\bfX} V^r X}_{B^{3\alpha}_{p/3,q/3}} + \norm{V^{r/2}R^{\tilde\bfY,\tilde\bfX} V^r\Delta X}_{B^{3\alpha}_{p/3,q/3}}\\
&+\norm{V^r\Delta Y' V^{r/2} \bbX}_{B^{3\alpha}_{p/3,q/3}} + \norm{V^r\tilde Y'V^{r/2}\Delta\bbX}_{B^{3\alpha}_{p/3,q/3}}\\
\lesssim & \norm{V^{r/2}\Delta R^{\bfY,\bfX}}_{B^{2\alpha}_{p/2,q/2}} \norm{V^r X}_{B^{\alpha}_{p,q}} + \norm{V^{r/2}R^{\tilde\bfY,\tilde\bfX}}_{B^{2\alpha}_{p/2,q/2}} \norm{V^r\Delta X}_{B^{\alpha}_{p,q}}\\
&+\norm{V^r\Delta Y'}_{B^\alpha_{p,q}} \norm{V^{r/2} \bbX}_{B^{2\alpha}_{p/2,q/2}} + \norm{V^r\tilde Y'}_{B^\alpha_{p,q}}\norm{V^{r/2}\Delta\bbX}_{B^{2\alpha}_{p/2,q/2}}\,,
\end{align*}
where we used Hölder's inequality for Besov norms in the last line. Proposition \ref{prop:Vr-Besov-Embedding-RYX} and Corollary \ref{cor:Vr_Besov_Embedding_Rough_Paths} now immediately give the claim.
\end{proof}

\subsection{Young ODEs}

We consider the Young equation
\begin{equation}\label{eq:YoungODE2}
	dY_t = \phi(Y_t) dX_t
\end{equation}
with $X\in B^\alpha_{p,q}, \phi\in C^{1,\beta}_b$ for $\alpha>\frac 12,\frac 1\alpha <1+\beta, 1\le p\le\infty$ with $\frac 1\alpha <p$ and $0<q\le \infty$ and show that in this regime, one can recover existence, uniqueness, and Lipschitz estimates of the solution in this case (see \cite[Theorem 4.2 and Theorem 4.3]{FS22}) from our variation analysis Theorem \ref{thm:Lipschitz-stability}. As in the previous sections, we fix an $r\in[1,2)$ such that
\begin{align*}
\frac 1\alpha <r\le p \\
r<1+\beta\,.
\end{align*}
By Corollary \ref{cor:EmbeddingVrInBesov}, we get $V^r X\lesssim \norm{X}_{B^\alpha_{p,q}}$, so the (local) existence and uniqueness of $Y$ immediately follows from Proposition \ref{prop:Picard-contraction}. We also quickly recall our Lipschitz estimate: By choosing $\Phi_0,\Phi_\beta,\Phi_1,\Phi_{1,\beta} > 0$ such that \eqref{eq:C^0,alpha-bounds}, \eqref{eq:C^1,alpha-bounds} hold and for $V^r X<\epsilon$ for small enough $\epsilon > 0$, Theorem \ref{thm:Lipschitz-stability} gives
\begin{equation}\label{eq:Young-loc-Lip}
V^{r}\Delta Y_I \leq \gamma V^{r}X_I + 2\Phi_{0} V^{r}\Delta X_I
\end{equation}
for all intervals $I\subset[0,T]$, with
\begin{align*}
\gamma \lesssim &
((\Csew+1)\Phi_{1} + \Csew 2^{\alpha+1} \frac{\Phi_{1,\alpha}}{1+\alpha} \Phi_0^{\alpha} \epsilon^{\alpha}) \Phi_{0} V^r \Delta X_{[0,T]}  +
(\Csew 2^{\alpha+1} \frac{\Phi_{1,\alpha}}{1+\alpha} \Phi_0^{\alpha} \epsilon^{\alpha} + \Phi_{1}) \abs{y_{0}-\tilde{y}_{0}}
\\&\quad+ \Csew 2^\alpha \Phi_0^\alpha \epsilon^\alpha \norm{\Delta\phi}_{C^{\alpha}}
+ \norm{\Delta\phi}_{\sup}\,,
\end{align*}

\noindent As before, Theorem \ref{theo:VrBesovEstimate} allows us to easily recover the Besov version of these results:

\begin{theorem}\label{thm:LipschitzBesovEstimateYoung}
Let  $\alpha >\frac 12, \beta > 0$ with $\frac 1\alpha < 1+\beta$ and $1 \le p\le \infty$ with $\frac 1\alpha < p$, as well as $0<q\le\infty$. Let $X\in B^\alpha_{p,q}$ and $\phi\in C^{1,\beta}_b$ with bounds $\Phi_0,\Phi_\beta,\Phi_1,\Phi_{1,\beta}$ according to \eqref{eq:C^0,alpha-bounds}, \eqref{eq:C^1,alpha-bounds}. Then, possibly for a shorter time $T>0$, the Young ODE \eqref{eq:YoungODE2} has a unique, local solution $Y\in B^\alpha_{p,q}$. Further, this solution fulfills
\begin{equation*}
\norm{\Delta Y}_{B^{\alpha}_{p,q}} \lesssim \tilde\gamma\norm{X}_{B^{\alpha}_{p,q}} + 2\phi_0\norm{\Delta X}_{B^{\alpha}_{p,q}}\,,
\end{equation*}
where
\begin{align*}
\tilde \gamma \lesssim &
((\Csew+1)\Phi_{1} + \Csew 2^{\alpha+1} \frac{\Phi_{1,\alpha}}{1+\alpha} \Phi_0^{\alpha} \epsilon^{\alpha}) \Phi_{0} \norm{\Delta X}_{B^\alpha_{p,q}}  +
(\Csew 2^{\alpha+1} \frac{\Phi_{1,\alpha}}{1+\alpha} \Phi_0^{\alpha} \epsilon^{\alpha} + \Phi_{1}) \abs{y_{0}-\tilde{y}_{0}}
\\&\quad+ \Csew 2^\alpha \Phi_0^\alpha \epsilon^\alpha \norm{\Delta\phi}_{C^{\alpha}}
+ \norm{\Delta\phi}_{\sup}\,,
\end{align*}
and $\Csew = 2^{\frac{1+\beta} r -1}\zeta(\frac{1+\beta} r)$.
\end{theorem}

\begin{proof}
By Theorem \ref{theo:VrBesovEstimate}, $\norm{V^r X}_{B^\alpha_{p,q}}\lesssim \norm{X}_{B^\alpha_{p,q}}<\infty$, which especially shows that $t\mapsto V^r X_{0,t}$ is $\alpha-\frac 1p$ Hölder continuous. This implies that for small enough $T>0$, we have $V^r X <\epsilon$ for $\epsilon > 0$ fulfilling \eqref{eq:VpX-small1} and \eqref{eq:Young-Picard-interval2}. Thus, Theorem \ref{thm:Lipschitz-stability} gives us \eqref{eq:Young-loc-Lip}. Using Theorem \ref{theo:VrBesovEstimate} (or more precisely Theorem \ref{theo:IntroVrBesovEstimate}), we conclude
\begin{align*}
\norm{\Delta Y}_{B^\alpha_{p,q}} &\le \norm{V^r\Delta Y}_{B^\alpha_{p,q}} \\
&\lesssim \gamma\norm{V^r X}_{B^{\alpha}_{p,q}} + 2\phi_0\norm{V^r\Delta X}_{B^\alpha_{p,q}} \\ 
&\lesssim \gamma\norm{X}_{B^{\alpha}_{p,q}} + 2\phi_0\norm{\Delta X}_{B^\alpha_{p,q}}\,. 
\end{align*}
Since $V^r\Delta X \lesssim \norm{\Delta X}_{B^\alpha_{p,q}}$ by the Besov embedding given by Corollary \ref{cor:EmbeddingVrInBesov}, we have
\begin{equation*}
\gamma\lesssim\tilde\gamma\,,
\end{equation*}
showing the claim.
\end{proof}

\subsection{RDEs}

We consider now the RDE
\begin{equation}
Y_t = y_0 +\int_0^t\phi(\bfY_r)\dif \bfX_r \label{eq:RDE2}\,,
\end{equation}
where $\bfX$ is an Besov rough path with $\bfX\in B^\alpha_{p,q}$ for $\alpha>\frac 13$, and $2\le p<\infty, \frac 12\le q\le\infty$ such that $\frac 1\alpha <p$. $\phi \in C^{2,1}_b$. Similar to the Young case, we want to recover existence, uniqueness and Lipschitz estimates in the Besov case (see \cite[Theorem 5.6, Theorem 5.7]{FS22}) from the variation case shown in Theorem \ref{thm:LipschitzDependenceRDE}. We choose an $r\in[2,3)$ such that
\begin{equation*}
\frac 1\alpha <r\le p \,.
\end{equation*}
Then Lemma \ref{lem:BesovRPAreVariationRP} shows that $\bfX$ is an $r$-rough path and thus, there exists a (local) solution $\bfY$ to \eqref{eq:RDE2} as a controlled $r$-rough path. We further recall the Lipschitz estimate from Theorem \ref{thm:LipschitzDependenceRDE}: For $\phi$ being bound by $\Phi_0,\Phi_{0,1}, \Phi_1,\Phi_{1,1}, \Phi_2, \Phi_{2,1}, \Phi_{DD}$ according to \eqref{ineq:boundsPhiRP} and for small enough $V^r X, V^{r/2}\bbX$, we have
\begin{align}
V^r(\Delta Y)_I &\lesssim  (V^rX_I + V^{r/2}\bbX_{I} + V^{r/2}\tilde\bbX_{I}) \gamma + V^r(\Delta X)_I + V^{r/2}(\Delta \bbX)_{I} \label{ineq:LipschitzRDE2}\\
V^{r/2}(\Delta R^{\bfY,\bfX})_{I}&\lesssim (V^rX_I^2 + V^rX_I \cdot V^r\tilde X_I +V^{r/2}\bbX_{I} + V^{r/2}\tilde \bbX_{I})\gamma \label{ineq:LipschitzRDE3-7}\\
&\qquad + (V^r X_I+V^r\tilde X_I +V^{r/2}\bbX_I + V^{r/2}\tilde\bbX_{I})V^r(\Delta X)_I + V^{r/2}(\Delta\bbX)_{I}\nonumber\,,
\end{align}
where
\[
	\gamma = \abs{\Delta y_0} + V^r{\Delta X}_{[0,T]} + V^{r/2}{\Delta\bbX}_{[0,T]}\,.
\]
As in the Young case, this can be translated into the Besov case with Proposition \ref{prop:Vr-Besov-Embedding-RYX} and Corollary \ref{cor:Vr_Besov_Embedding_Rough_Paths}. To do so, we recall the Besov norm of a rough path, see \cite[equation (5.3)]{FS22} as reference:
\[
	\nnorm{\bfX}_{B^\alpha_{p,q}} = \norm{X}_{B^\alpha_{p,q}} + \norm{\bbX}_{B^{2^\alpha}_{p/2,q/2}}^{\frac 12}\,.
\]

\begin{theorem}\label{thm:LipschitzBesovEstimateRDE}
Let $\bfX,\tilde\bfX\in B^\alpha_{p,q}$ be Besov rough paths for $\alpha >\frac 13$, $\frac 1\alpha<p$ and $2\le p\le\infty, \frac 12\le q\le\infty$. Let $\phi\in C^{2,1}_b$ be bound by $\Phi_0,\Phi_{0,1}, \Phi_1,\Phi_{1,1}, \Phi_2, \Phi_{2,1}, \Phi_{DD}$ according to \eqref{ineq:boundsPhiRP}. Then for a possibly smaller $T>0$ the RDE \eqref{eq:RDE2} has a unique solution $\bfY$ which is a Besov-controlled rough path with parameters $\alpha,p,q$ and we have the Lipschitz-estimate in the Besov scale
\begin{align}
\norm{\Delta Y}_{B^\alpha_{p,q}} &\lesssim (\nnorm{\bfX}_{B^\alpha_{p,q}} + \nnorm{\tilde\bfX}_{B^\alpha_{p,q}})\abs{\Delta y_0} \nonumber\\
&\qquad + (1+ \norm{X}_{B^\alpha_{p,q}} + \norm{\tilde X}_{B^\alpha_{p,q}}) \norm{\Delta X}_{B^\alpha_{p,q}} + \norm{\Delta \bbX}_{B^{2\alpha}_{p/2,q/2}}\label{ineq:LipschitzRDE3-SectionBesov}\\
\norm{\Delta R^{\bfY,\bfX}}_{B^{2\alpha}_{p/2,q/2}} &\lesssim (\nnorm{\bfX}_{B^\alpha_{p,q}}^2 + \nnorm{\tilde \bfX}_{B^\alpha_{p,q}}^2 + \norm{X}_{B^\alpha_{p,q}}\norm{\tilde X}_{B^\alpha_{p,q}}) \abs{\Delta y_0} \nonumber\\
&\qquad + (\nnorm{\bfX}_{B^\alpha_{p,q}} + \nnorm{\tilde\bfX}_{B^\alpha_{p,q}})\norm{\Delta X}_{B^\alpha_{p,q}} + \norm{\Delta \bbX}_{B^{2\alpha}_{p/2,q/2}}\,.\label{ineq:LipschitzRDE4-SectionBesov}
\end{align}
\end{theorem}

\begin{proof}
We quickly check the conditions of Theorem \ref{thm:RDE-solution-local-uniqueness} and Theorem \ref{thm:LipschitzDependenceRDE} regarding $\bfX,\tilde\bfX$. By Lemma \ref{lem:BesovRPAreVariationRP}, $\bfX,\tilde\bfX$ are $r$-rough paths and by Corollary \ref{cor:Vr_Besov_Embedding_Rough_Paths}, $\norm{V^r X}_{B^\alpha_{p,q}} < \infty$, \\$\norm{V^{r/2}\bbX}_{B^{2\alpha}_{p/2,q/2}}<\infty$. Thus, $V^r(X,\bbX)\in B^\alpha_{p,q}$ in the metric space $G_1$ from Example \ref{ex:GroupForX}. It especially follows that $V^r(X,\bbX)$ is $\alpha-\frac  1p$ Hölder continuous, implying that $V^rX$ is $\alpha-\frac 1p$ Hölder continuous and $V^{r/2} \bbX$ is $2(\alpha-\frac 1p)$ Hölder continuous. We conclude that $V^r X, V^{r/2}\bbX <\epsilon$ for any $\epsilon>0$ for small enough $T>0$, with the same argumentation giving one $V^r \tilde X, V^{r/2}\tilde\bbX <\epsilon$.

Thus, we can apply Theorem \ref{thm:RDE-solution-local-uniqueness} and Theorem \ref{thm:LipschitzDependenceRDE}, showing that unique solutions $\bfY,\tilde\bfY$ exist as controlled $r$-rough paths and giving us the Lipschitz estimates \eqref{ineq:LipschitzRDE2}, \eqref{ineq:LipschitzRDE3-7}. Applying the Besov norm to $\Delta Y$ gives
\begin{align*}
\norm{\Delta Y}_{B^\alpha_{p,q}} &\le \norm{V^r \Delta Y}_{B^\alpha_{p,q}} \\
&\lesssim (\norm{V^r X}_{B^\alpha_{p,q}} + \norm{V^{r/2} \bbX^\frac 12}_{B^\alpha_{p,q}} + \norm{V^{r/2} \tilde\bbX^\frac 12}_{B^\alpha_{p,q}})\gamma + \norm{V^r\Delta X}_{B^\alpha_{p,q}} + \norm{V^{r/2} \Delta \bbX}_{B^\alpha_{p,q}}\,,
\end{align*}
where we used $V^{r/2} \bbX \le \epsilon^{\frac 12} V^{r/2}\bbX^{\frac 12}$. We use $\norm{V^{r/2}\bbX^\frac 12}_{B^\alpha_{p,q}} = \norm{V^{r/2}\bbX}_{B^{2\alpha}_{p/2,q/2}}^{\frac 12}$ together with Theorem \ref{theo:VrBesovEstimate} and Corollary \ref{cor:Vr_Besov_Embedding_Rough_Paths} to get
\begin{align*}
\norm{V^r \Delta X}_{B^\alpha_{p,q}}&\lesssim \norm{\Delta X}_{B^\alpha_{p,q}}\\
\norm{V^r X}_{B^\alpha_{p,q}} + \norm{V^{r/2} \bbX^\frac 12}_{B^\alpha_{p,q}} + \norm{V^{r/2} \tilde\bbX^\frac 12}_{B^\alpha_{p,q}} & \lesssim \nnorm{\bfX}_{B^\alpha_{p,q}} + \nnorm{\bfX}_{B^\alpha_{p,q}}\,.
\end{align*}
This together with Proposition \ref{prop:VrBesovEmbeddingDeltaX2} gives:
\begin{align*}
\norm{\Delta Y}_{B^\alpha_{p,q}} &\lesssim (\nnorm{\bfX}_{B^\alpha_{p,q}} + \nnorm{\tilde\bfX}_{B^\alpha_{p,q}})\gamma \\
&\qquad + (1+ \norm{X}_{B^\alpha_{p,q}} + \norm{\tilde X}_{B^\alpha_{p,q}}) \norm{\Delta X}_{B^\alpha_{p,q}} + \norm{\Delta \bbX}_{B^{2\alpha}_{p/2,q/2}}\,.
\end{align*}
\eqref{ineq:LipschitzRDE3-SectionBesov} then follows from
\begin{equation}\label{ineq:EstimateGamma}
	\gamma \lesssim \abs{\Delta y_0} +\norm{\Delta X}_{B^\alpha_{p,q}} + \norm{\Delta\bbX}_{B^{2\alpha}_{p/2,q/2}} + (\norm{X}_{B^\alpha_{p,q}} + \norm{\tilde X}_{B^{\alpha}_{p,q}})\norm{\Delta X}_{B^\alpha_{p,q}}\,,
\end{equation}
where we used Lemma \ref{lem:BesovRPAreVariationRP}.

Let us move to proving \eqref{ineq:LipschitzRDE4-SectionBesov}. By putting the Besov norm $\norm\cdot_{B^{2\alpha}_{p/2,q/2}}$ on \eqref{ineq:LipschitzRDE3-7} and using Hölder's inequality, we get
\begin{align}
\begin{split}\label{ineq:BesovEstimateR}
\norm{\Delta R^{\bfY,\bfX}}_{B^{2\alpha}_{p/2,q/2}} &\le \norm{V^{r/2} \Delta R^{\bfY,\bfX}}_{B^{2\alpha}_{p/2,q/2}}\\
&\lesssim (\norm{V^r X}_{B^\alpha_{p,q}}^2 + \norm{V^r X}_{B^\alpha_{p,q}}\norm{V^r \tilde X}_{B^\alpha_{p,q}} + \norm{V^{r/2} \bbX}_{B^{2\alpha}_{p/2,q/2}} + \norm{V^{r/2} \tilde\bbX}_{B^{2\alpha}_{p/2,q/2}})\gamma\\
&\quad + (\norm{V^r X}_{B^\alpha_{p,q}} + \norm{V^r \tilde X}_{B^\alpha_{p,q}} + \norm{V^{r/2} \bbX}_{B^{2\alpha}_{p/2,q/2}}^{\frac 12} + \norm{V^{r/2} \tilde\bbX}_{B^{2\alpha}_{p/2,q/2}}^{\frac 12})\norm{V^r \Delta X}_{B^\alpha_{p,q}} \\
&\quad+\norm{V^{r/2}\Delta \bbX}_{B^{2\alpha}_{p/2,q/2}}\,,
\end{split}
\end{align}
where we again used $V^{r/2} \bbX\le \epsilon^{\frac 12} V^{r/2}\bbX^{\frac 12}$ and $V^{r/2} \tilde\bbX\le \epsilon^{\frac 12} V^{r/2}\tilde\bbX^{\frac 12}$ when necessary. Direct application of Corollary \ref{cor:Vr_Besov_Embedding_Rough_Paths} gives
\begin{align*}
\norm{V^r X}_{B^\alpha_{p,q}}^2 &+ \norm{V^r X}_{B^\alpha_{p,q}}\norm{V^r \tilde X}_{B^\alpha_{p,q}} + \norm{V^{r/2} \bbX}_{B^{2\alpha}_{p/2,q/2}} + \norm{V^{r/2} \tilde\bbX}_{B^{2\alpha}_{p/2,q/2}} \\
&\lesssim \norm{X}_{B^\alpha_{p,q}}^2 + \norm{\bbX}_{B^{2\alpha}_{p/2,q/2}} + \norm{\tilde X}_{B^\alpha_{p,q}}^2 + \norm{\tilde\bbX}_{B^{2\alpha}_{p/2,q/2}}  + \norm{X}_{B^\alpha_{p,q}} \norm{\tilde X}_{B^\alpha_{p,q}}\\
&\lesssim \nnorm{\bfX}_{B^\alpha_{p,q}}^2 + \nnorm{\tilde\bfX}_{B^\alpha_{p,q}}^2 + \norm{X}_{B^\alpha_{p,q}} \norm{\tilde X}_{B^\alpha_{p,q}}\,,
\end{align*}
as well as
\begin{align*}
(\norm{V^r X}_{B^\alpha_{p,q}} + \norm{V^r \tilde X}_{B^\alpha_{p,q}} &+ \norm{V^{r/2} \bbX}_{B^{2\alpha}_{p/2,q/2}}^{\frac 12} + \norm{V^{r/2} \tilde\bbX}_{B^{2\alpha}_{p/2,q/2}}^{\frac 12}) \norm{V^r \Delta X}_{B^\alpha_{p,q}} \\ &\lesssim (\nnorm{\bfX}_{B^\alpha_{p,q}} + \nnorm{\tilde\bfX}_{B^\alpha_{p,q}}) \norm{\Delta X}_{B^\alpha_{p,q}}\,.
\end{align*}
Finally, we have
\begin{equation*}
\norm{V^{r/2} \Delta\bbX}_{B^{2\alpha}_{p/2,q/2}} \lesssim \norm{\Delta\bbX}_{B^{2\alpha}_{p/2,q/2}} + \norm{\Delta X}_{B^\alpha_{p,q}}(\norm{X}_{B^\alpha_{p,q}} + \norm{\tilde X}_{B^\alpha_{p,q}})\,.
\end{equation*}
Putting all these as well as \eqref{ineq:EstimateGamma} in \eqref{ineq:BesovEstimateR} shows \eqref{ineq:LipschitzRDE4-SectionBesov}.
\end{proof}

\begin{remark}
Reducing the condition on $q$ in Proposition \ref{prop:Besov-Embedding-P-Alpha} (see Remark \ref{rem:reduceQ}) would lead to a reduction in Proposition \ref{prop:VrBesovEmbeddingDeltaX2} and therefore in this proof. Thus, we expect Theorem \ref{thm:LipschitzBesovEstimateRDE} to hold for all $0< q\le\infty$.
\end{remark}

\printbibliography

\end{document}